\documentclass[draft]{amsart}

\usepackage[pdfborder={0 0 0},draft=false]{hyperref}

\usepackage{esint,amssymb,bm,tikz}
\usetikzlibrary{calc}

\newtheorem{thm}[equation]{Theorem}
\newtheorem{lem}[equation]{Lemma}
\newtheorem{cor}[equation]{Corollary}

\theoremstyle{definition}

\newtheorem{dfn}[equation]{Definition}
\newtheorem{rmk}[equation]{Remark}

\numberwithin{equation}{section}
\numberwithin{figure}{section}

\newcommand\abs[2][empty]{\csname#1\endcsname \lvert{#2}\csname#1\endcsname\rvert}
\newcommand\doublebar[2][empty]{\csname#1\endcsname \lVert{#2}\csname#1\endcsname\rVert}

\newcommand\mat[1]{\bm{#1}}
\newcommand\arr[1]{\dot{\bm{#1}}}

\newcommand\dist{\mathop{\mathrm{dist}}\nolimits}
\newcommand\Div{\mathop{\mathrm{div}}\nolimits}
\newcommand\Tr{\mathop{\smash{\arr{\mathrm{Tr}}}\vphantom{T}}\nolimits}
\newcommand\Trace{\mathop{\mathrm{Tr}}\nolimits}
\newcommand\M{\mathop{\smash{\arr{\mathrm{M}}}\vphantom{M}}\nolimits}
\newcommand\MM{\mathop{\mathrm{M{}}}\nolimits}

\newcommand\supp{\mathop{\mathrm{supp}}\nolimits}
\newcommand\diam{\mathop{\mathrm{diam}}\nolimits}

\newcommand\R{\mathbb{R}} 
\newcommand\C{\mathbb{C}} 
\newcommand\Z{\mathbb{Z}} 
\newcommand\N{\mathbb{N}}
\newcommand\1{\mathbf{1}}

\newcommand\NN{\mathfrak{N}}
\newcommand\XX{\mathfrak{X}}
\newcommand\YY{\mathfrak{Y}}

\newcommand\pmin{\pdmnMinusOne/\allowbreak(\dmnMinusOne+\theta)}

\newcommand\dmn{d}
\newcommand\pdmn{d}
\newcommand\dmnMinusOne{{d-1}}

\newcommand\pdmnMinusOne{{(d-1)}}

\begin{document}

\title[Traces and extensions]{Trace and extension theorems relating Besov spaces to weighted averaged Sobolev spaces}


\author{Ariel Barton}
\address{Ariel Barton, Department of Mathematical Sciences,
			309 SCEN,
			University of Ar\-kan\-sas,
			Fayette\-ville, AR 72701}
\email{aeb019@uark.edu}



\begin{abstract}
There are known trace and extension theorems relating functions in a weighted Sobolev space in a domain~$\Omega$ to functions in a Besov space on the boundary~$\partial\Omega$. We extend these theorems to the case where the Sobolev exponent $p$ is less than one by modifying our Sobolev spaces to consider averages of functions in Whitney balls. Averaged Sobolev spaces are also of interest in the applications in the case where $p>1$, and so we also provide trace and extension results in that case. Finally, we provide some comparable results for Neumann traces and extensions. 
%
\end{abstract}

\maketitle

\setcounter{tocdepth}{1}
\tableofcontents
\setcounter{tocdepth}{2}

\section{Introduction}

Suppose that $u$ is a function defined in some domain~$\Omega$. We are interested in the boundary values of~$u$. Specifically, we wish to identify a space $\XX$ such that if $u$ lies in~$\XX$, then the boundary traces $\Trace\nabla^{m-1} u$ of the derivatives of order $m-1$ lie in the Besov space $\dot B^{p,p}_\theta(\partial\Omega)$.

We would like our result to be sharp in the sense that, if $\arr f$ is an array of functions in $\dot B^{p,p}_\theta(\partial\Omega)$, and if $\arr f=\Trace \nabla^{m-1}\varphi$ for some function~$\varphi$, then $\arr f=\Trace \nabla^{m-1} F$ for some $F\in\XX$. (Recall that the partial derivatives of a function must satisfy some compatibility conditions; thus, the requirement that $\arr f = \Trace\nabla^{m-1}\varphi$ for some~$\varphi$ is a nontrivial restriction if $m\geq 2$.)

Such trace and extension theorems bear a deep connection to the theory of Dirichlet boundary value problems. For example, consider the harmonic Dirichlet problem 
\begin{equation}\label{eqn:dirichlet:harmonic}
\Delta u=0\text{ in }\Omega,\quad u=\varphi\text{ on }\partial\Omega,\quad \doublebar{u}_\XX\leq C\doublebar{\varphi}_{\dot B^{p,p}_\theta(\partial\Omega)}
\end{equation}
or more generally the higher-order boundary value problem
\begin{equation}\label{eqn:dirichlet:boundary}L u=0\text{ in }\Omega,\quad \nabla^{m-1}u=\nabla^{m-1}\varphi\text{ on }\partial\Omega,\quad \doublebar{u}_\XX\leq C\doublebar{\Trace \nabla^{m-1}\varphi}_{\dot B^{p,p}_\theta(\partial\Omega)}\end{equation}
for some differential operator $L$ of the form $Lu=\sum_{\abs{\alpha}=\abs{\beta}=m} \partial^\alpha(A_{\alpha\beta} \partial^\beta u)$. If we have an extension theorem as indicated above, then there is some $F\in\XX$ with $ \nabla^{m-1}F=\nabla^{m-1}\varphi$ on~$\partial\Omega$. If $L:\XX\mapsto \YY$ is bounded, then we may reduce the problem \eqref{eqn:dirichlet:boundary} to the problem 
\begin{equation}\label{eqn:dirichlet:poisson}
L v=h\text{ in }\Omega,\quad \nabla^{m-1}v=0\text{ on }\partial\Omega,\quad \doublebar{v}_\XX\leq C\doublebar{h}_\YY\end{equation}
with zero boundary data by letting $h=-LF$ and then letting $u=v+F$. In some cases we may reverse the argument, going from well-posedness of the problem \eqref{eqn:dirichlet:boundary} to well-posedness of the problem~\eqref{eqn:dirichlet:poisson}. See the papers \cite{JerK95,AdoP98,MayMit04A,Agr07, MazMS10,MitMW11, MitM13B,MitM13A,BarM16A} for examples of such arguments with various choices of~$L$; the trace and extension theorems of the present paper will be used in \cite{Bar16pA} for this purpose.

In this paper we will introduce the weighted averaged Lebesgue spaces $L_{av}^{p,\theta,q}(\Omega)$ and Sobolev spaces $\dot W^{p,\theta,q}_{m,av}(\Omega)$, where $\dot W_{m,av}^{p,\theta,q}(\Omega)$ is defined to be the space of functions $u$ with 
$\doublebar{u}_{\dot W_{m,av}^{p,\theta,q}(\Omega)}=\doublebar{\nabla^m u}_{L_{av}^{p,\theta,q}(\Omega)}<\infty$ and where the
$L_{av}^{p,\theta,q}(\Omega)$-norm is given by
\begin{equation}\label{eqn:norm:L}
\doublebar{H}_{L_{av}^{p,\theta,q}(\Omega)}
= 
\biggl(\int_\Omega \biggl(\fint_{B(x,\dist(x,\partial\Omega)/2)} \abs{H}^q \biggr)^{p/q}  \dist(x,\partial\Omega)^{p-1-p\theta}\,dx\biggl)^{1/p}
.\end{equation}
The main result of this paper for Dirichlet boundary data is the following theorem. 
\begin{thm} \label{thm:dirichlet}
Let $\Omega\subset\R^\dmn$ be a Lipschitz domain with connected boundary. Let $1\leq q\leq \infty$, let $0<\theta<1$ and let $\pmin<p\leq\infty$.

If $u\in \dot W_{m,av}^{p,\theta,q}(\Omega)$, then $\Trace \partial^\gamma u\in \dot B^{p,p}_\theta(\partial\Omega)$ for any multiindex $\gamma$ with $\abs{\gamma}=m-1$, and $$\doublebar{\Trace \partial^\gamma u}_{\dot B^{p,p}_\theta(\partial\Omega)}\leq C \doublebar{u}_{\dot W_{m,av}^{p,\theta,q}(\Omega)}$$ for some constant $C$ depending only on $p$, $\theta$, the Lipschitz character of $\Omega$ and the ambient dimension~$\dmn$.

Conversely, let $F$ be a function such that $\Trace \partial^\gamma F\in \dot B^{p,p}_\theta(\partial\Omega)$ for any $\abs\gamma=m-1$.  Then there is some $u\in \dot W_{m,av}^{p,\theta,q}(\Omega)$ with $$\doublebar{u}_{\dot W_{m,av}^{p,\theta,q}(\Omega)}\leq C\doublebar{\Trace \nabla^{m-1} F}_{\dot B^{p,p}_{\theta}(\partial\Omega)}
\quad\text{and}\quad \Trace \nabla^{m-1} u=\Trace \nabla^{m-1} F.$$
\end{thm}

Also of great importance in the theory of boundary value problems is the second-order Neumann problem 
\begin{equation}\label{eqn:neumann:harmonic}
\Div \mat A\nabla u=0\text{ in }\Omega,\quad \nu\cdot \mat A\nabla u=g\text{ on } \partial\Omega\end{equation}
where $\nu$ is the unit outward normal vector to~$\Omega$ and where $\mat A$ is a coefficient matrix. We are interested in the Neumann problem for higher order equations; the second main result of this paper (Theorem~\ref{thm:neumann} below) is an analogue of Theorem~\ref{thm:dirichlet} for Neumann boundary data.

The appropriate generalization of Neumann boundary values to the higher order case is a complicated issue. 
We are interested in the following generalization of Neumann boundary values; this is the formulation used in \cite{Bar17pA,BarHM15p}, and is related to but subtly different from that of \cite{CohG85,Ver05,Agr07,Ver10,MitM13A}. 
We refer the reader to \cite{BarHM15p,BarM16B} for a discussion of various formulations of Neumann boundary data.

If $\vec G$ is a smooth vector field on $\overline\Omega$, then $\nu\cdot \vec G$ may be regarded as its Neumann boundary values. If $\vec G$ is
divergence free (in particular, if $\vec G=\mat A\nabla u$ for some solution $u$ to the problem~\eqref{eqn:neumann:harmonic}), then $\nu\cdot\vec G$ satisfies
\begin{equation}\label{eqn:neumann:1}\int_{\partial\Omega} \Trace\varphi\,(\nu\cdot \vec G) \,d\sigma 
= \int_\Omega \nabla\varphi\cdot \vec G
=\sum_{j=1}^\dmn \int_\Omega \partial_j\varphi\, G_j
\quad\text{for all $\varphi\in C^\infty_0(\R^\dmn)$.}\end{equation}
This formula may be used to define the Neumann boundary values of $\vec G$ even if $\vec G$ is not smooth. Furthermore, this formula generalizes to the higher order case: if $\arr G$ is an array of locally integrable functions indexed by multiindices $\alpha$ of length~$m$, then the analogue of formula~\eqref{eqn:neumann:1} is 
\begin{equation}\label{dfn:neumann}\sum_{\abs\gamma=m-1}\int_{\partial\Omega} \Trace\partial^\gamma\varphi\,(\M_m^\Omega \arr G)_\gamma \,d\sigma =  \sum_{\abs{\alpha}=m} \int_\Omega\partial^\alpha\varphi\,G_\alpha
\quad\text{for all $\varphi\in C^\infty_0(\R^\dmn)$}\end{equation}
where the array of distributions $\M_m^\Omega \arr G$ represents the Neumann boundary values of~$\arr G$.

We remark on two subtleties of formula~\eqref{dfn:neumann} in the case $m\geq 2$. 

First, the left-hand side of formula~\eqref{dfn:neumann} depends only on the boundary values $\Trace \nabla^{m-1}\varphi$ of $\varphi$ on~$\partial\Omega$ and not on the values of $\varphi$ in~$\Omega$; in this way $\M_m^\Omega\arr G$ may indeed be said to be Neumann \emph{boundary} values of~$\arr G$. For this equation to be meaningful, we must have that the right-hand side depends only on the boundary values of~$\varphi$ as well; thus, $\M_m^\Omega\arr G$ is defined only for arrays $\arr G$ that satisfy
\begin{equation}\label{eqn:divergence-free}
\int_\Omega\partial^\alpha\varphi\,G_\alpha=0
\quad\text{for all $\varphi\in C^\infty_0(\Omega)$.}\end{equation}
An array $\arr G$ that satisfies formula~\eqref{eqn:divergence-free} is said to satisfy $\Div_m\arr G=0$ in $\Omega$ in the weak sense; this condition is analogous to the requirement that $\Div \vec G=0$ in formula~\eqref{eqn:neumann:1}.
We remark that if $\arr G$ is smooth then $\Div_m \arr G=0$ if and only if $\sum_{\abs{\alpha}=m}\partial^\alpha G_\alpha=0$.

Second, if $\arr G$ is divergence-free in the sense of formula~\eqref{eqn:divergence-free}, then formula~\eqref{dfn:neumann} does define $\M_m^\Omega\arr G$ as an operator on the space $\{\Trace \nabla^{m-1}\varphi:\varphi\in C^\infty_0(\R^\dmn)\}$. This space is a \emph{proper} subspace of the space of arrays of smooth, compactly supported functions. Thus, $\M_m^\Omega\arr G$ is not an array of well-defined distributions; instead it is an equivalence class of such arrays, defined only up to adding arrays of distributions $\arr g$ for which $\langle \Trace\nabla^{m-1}\varphi,\arr g\rangle_{\partial\Omega}=0$ for all $\varphi\in C^\infty_0(\R^\dmn)$. Thus, Neumann boundary data naturally lies in quotient spaces of distribution spaces, as will be seen in the following theorem.
(This theorem is the second main result of this paper.)


\begin{thm} \label{thm:neumann}
Let $\Omega\subset\R^\dmn$ be a Lipschitz domain with connected boundary. Let $1\leq q\leq \infty$, let $0<\theta<1$ and let $\pmin<p\leq\infty$.

Suppose that $\arr g$ is an array of functions lying in $\dot B^{p,p}_{\theta-1}(\partial\Omega)$. Then there is some $\arr G\in L_{av}^{p,\theta,q}(\Omega)$, with $\Div_m\arr G=0$ in~$\Omega$, such that $\M_m^\Omega \arr G=\arr g$ in the sense that
\begin{equation*}\sum_{\abs\gamma=m-1}\int_{\partial\Omega} \Trace\partial^\gamma\varphi\,g_\gamma\,d\sigma =  \sum_{\abs{\alpha}=m} \int_\Omega\partial^\alpha\varphi\,G_\alpha\end{equation*}
for all smooth, compactly supported functions~$\varphi$. Furthermore, $$\doublebar{\arr G}_{L_{av}^{p,\theta,q}(\Omega)}\leq C\doublebar{\arr g}_{\dot B^{p,p}_{\theta-1}(\partial\Omega)}$$ for some constant $C$ depending only on $p$, $\theta$, the Lipschitz character of $\Omega$ and the ambient dimension~$\dmn$.

Conversely, let $\arr G\in L_{av}^{p,\theta,q}(\Omega)$ with $\Div_m\arr G=0$. Suppose that either $p>1$, or that $\Omega=\R^\dmn_+$ is a half-space, or that $m=1$ and  $\Omega=\{(x',t):x'\in\R^\dmnMinusOne,\>t>\psi(x')\}$ for some Lipschitz function $\psi:\R^\dmnMinusOne\mapsto\R$. 
Then the equivalence class of distributions $\M_m^\Omega \arr G$ contains a representative in $\dot B^{p,p}_{\theta-1}(\partial\Omega)$, and furthermore
\begin{equation*}\inf\{\doublebar{\arr g}_{\dot B^{p,p}_{\theta-1}(\partial\Omega)}:\arr g\in \M_m^\Omega\arr G\}
\leq C\doublebar{\arr G}_{L_{av}^{p,\theta,q}(\Omega)}.
\end{equation*}
\end{thm}


We now review the history of trace and extension theorems for boundary data in Besov spaces. To simplify our notation, we will introduce some terminology. Loosely, let 
\begin{equation*}\dot W\!A^p_{m-1,\theta}(\partial\Omega)
=\{\arr\varphi\in \dot B^{p,p}_{\theta}(\partial\Omega): \arr\varphi=\Trace\nabla^{m-1}\varphi\text{ for some }\varphi\}.\end{equation*}
(We will provide precise definitions in Section~\ref{sec:space:dfn}.) $\dot W\!A^p_{m-1,\theta}(\partial\Omega)$ is thus the space of all arrays of functions in a Besov space that may reasonably be expected to arise as boundary traces. Many of the results in the literature concern the ``inhomogeneous'' spaces $W\!A^p_{m-1,\theta}(\partial\Omega)$; these are defined analogously to $\dot W\!A^p_{m-1,\theta}(\partial\Omega)$ but in addition have some estimates on the lower order derivatives.

If $\Omega\subset\R^\dmn$ is a sufficiently smooth domain, it is well known that the operator
\begin{equation*}\Trace \nabla^{m-1}: B^{p,p}_{m-1+\theta+1/p}(\Omega) \mapsto W\!A^p_{m-1,\theta}(\partial\Omega)\end{equation*}
is bounded and has a right inverse (an extension operator defined on $W\!A^p_{m-1,\theta}(\partial\Omega)$) provided $m\geq 1$, $\theta>0$ and $p>\pmin$. (If $\Omega$ is a Lipschitz domain then we need the additional restriction $\theta<1$.) 
In the case of the half-space $\Omega=\R^\dmn_+$, see \cite[Section~2.7.2]{Tri83} for the full result, and the earlier works \cite[Appendix~A]{Pee76}, \cite{Nik77B} and \cite[Section~2.9.3]{Tri78}, and \cite{Jaw78} for the result under various restrictions.
%
%
In the case where $\Omega$ is smooth, see \cite[Section~3.3.3]{Tri83}. 
In the case where $\Omega$ is a Lipschitz domain, see 
\cite{JonW84} in the case $p\geq 1$, \cite{MayMit04A} in the case $m=1$, and \cite[Theorem~3.9]{MitM13A} for the general case.

Another well known family of extensions of Besov functions are the weighted Sobolev spaces. Define the $W^{p,\theta}_m(\Omega)$-norm by 
\begin{equation*}\doublebar{u}_{W^{p,\theta}_m(\Omega)}
=\doublebar{u}_{L^p(\Omega)}+\biggl(\int_{\Omega} \abs{\nabla^m u(x)}^p\,\dist(x,\partial\Omega)^{p-1-p\theta}\,dx\biggr)^{1/p}
.\end{equation*}
Notice that this is similar to the $\dot W_{m,av}^{p,\theta,q}(\Omega)$-norm of Theorem~\ref{thm:dirichlet}, but is somewhat simpler in that we do not take local $L^q$ averages. (The $ \doublebar{u}_{L^p(\Omega)}$ term is an ``inhomogeneous'' term as mentioned above.) We consider averaged spaces both because they are somewhat better suited to the setting of differential equations with rough coefficients, and also because taking averages allows us to establish trace results in the case $p<1$; this issue is discussed further below.

If $\Omega$ is sufficiently smooth, then we have that the trace operator \begin{equation*}\Trace \nabla^{m-1}: W^{p,\theta}_m(\Omega)\mapsto  W\!A^p_{m-1,\theta}(\partial\Omega)\end{equation*}
is bounded and has a bounded right inverse provided $0<\theta<1$ and $1<p<\infty$. 
In the case where $\Omega=\R^\dmn_+$ is a half-space, see
\cite{Liz60,Usp61} (a shorter proof of Uspenski\u\i's results with some generalization may be found in \cite{MirR15}) or \cite[Section 2.9.2]{Tri78}. 
In the case where $\Omega$ is a domain with a reasonably smooth boundary (for example, a $C^{k,\delta}$ domain for some $k+\delta>\theta$), see \cite{Nik77B,Sha85,NikLM88,Kim07}. 
In a Lipschitz domain, see \cite{BreM13} (the case $m=1$) and \cite[Section~7]{MazMS10} (for $m\geq 1$). A few results are known in the cases $p=1$ and $p=\infty$; in particular, \cite{MirR15} considers trace and extension results (in the half-space and with $m\geq 1$) for boundary data in the Besov space~$\dot B^{p,r}_\theta(\partial\R^\dmn_+)$ for $1\leq p\leq \infty$ and $1\leq r<\infty$. (In particular, these results are valid for boundary data in $\dot B^{1,1}_\theta(\partial\R^\dmn_+)$ but not in $\dot B^{\infty,\infty}_\theta(\partial\R^\dmn_+)$.)



The spaces $W^{p,\theta}_m(\Omega)$ and $B^{p,p}_{m-1+\theta+1/p}(\Omega)$ in some circumstances are related; for example, by \cite[Theorem~4.1]{JerK95}, if $1\leq p\leq \infty$ and $u$ is harmonic, then $u\in{W^{p,\theta}_m(\Omega)}$ if and only if $u\in B^{p,p}_{m-1+\theta+1/p}(\Omega)$.

We now discuss the history of Neumann trace and extension theorems. Recall that Neumann boundary values are in some sense dual to Dirichlet boundary values; thus, if $p>1$, then by duality between $B^{p,p}_{\theta-1}(\partial\Omega)$ and $B^{p',p'}_{1-\theta}(\partial\Omega)$, with some careful attention to the definitions, Neumann trace and extension theorems (such as our Theorem~\ref{thm:neumann}) follow from the corresponding Dirichlet extension and trace theorems. See Section~\ref{sec:neumann:extension:p>1} and Theorem~\ref{thm:trace:Neumann} below. This is essentially the approach taken in \cite{FabMM98,Zan00,MitM13B} and in the $p>1$ theory of \cite{MitM13A,BarM16A}. 

If $p\leq 1$, then $B^{p,p}_{\theta-1}(\partial\Omega)$ is not a dual space, and so another approach is needed. In \cite{MayMit04A}, the authors established a result similar to the $m=1$ case of Theorem~\ref{thm:neumann} with Besov spaces instead of weighted Sobolev spaces. Specifically, if $\Delta u=f$ for some $f$ supported in a Lipschitz domain~$\Omega$, they formulated a notion of normal derivative $\partial_\nu^f u$, coinciding with $\nu\cdot\nabla u$ if $u$ and $\Omega$ are sufficiently smooth, such that if $u\in B^{p,p}_{\theta+1/p}(\Omega)$ for some $0<\theta<1$ and some $p>\pmin$, then $\partial_\nu^f u\in B^{p,p}_{\theta-1}(\partial\Omega)$.
They also showed that this Neumann trace operator had a bounded right inverse. 

The author's paper \cite{BarM16A} with Svitlana Mayboroda introduced the weighted averaged Sobolev spaces $\dot W_{1,av}^{p,\theta,q}(\R^\dmn_+)$ in the half-space and in the case $m=1$. Therein Dirichlet and Neumann trace results were established for $p>\pmin$, rather than $p>1$. 

The present paper extends the results of \cite{BarM16A} concerning weighted averaged Sobolev spaces to the case $m\geq 2$, the case of arbitrary Lipschitz domains with connected boundary, and also provides extension theorems. As compared with known results for $m\geq 2$, the major innovation of this paper is to consider the case $p< 1$ in the weighted Sobolev space (rather than the Besov space) setting, and also to provide some new results in the case $p=\infty$. 

The case $p<1$ has been the subject of much recent study in the theory of elliptic boundary value problems. Specifically, in \cite{MayMit04A}, the authors considered the harmonic Dirichlet problem~\eqref{eqn:dirichlet:harmonic} with boundary data in $\dot B^{p,p}_\theta(\partial\Omega)$, $p<1$, $0<\theta<1$, and the corresponding harmonic Neumann problem with boundary data in $\dot B^{p,p}_{\theta-1}(\partial\Omega)$. In \cite{BarM16A}, the authors considered the Neumann problem \eqref{eqn:neumann:harmonic} and the corresponding Dirichlet problem (\eqref{eqn:dirichlet:boundary} with $m=1$) for more general second order operators, again with boundary data in Besov spaces $\dot B^{p,p}_{\theta}(\partial\Omega)$ or $\dot B^{p,p}_{\theta-1}(\partial\Omega)$ with $p<1$. (The case $p<1$ has also been of interest in the integer smoothness case, that is, in the case of boundary data in a Hardy space $H^p(\partial\Omega)$ for $p<1$; see \cite{AusM14,HofMayMou15,HofMitMor15}.) In \cite{Bar16pA} we intend to generalize some of the results of \cite{MayMit04A,BarM16A} to the higher order case (that is, to boundary value problems such as \eqref{eqn:dirichlet:boundary}, $m\geq 2$, and the corresponding Neumann problem) and to extend to even more general second-order equations; the trace and extension results of this paper will be very useful in that context.


Weighted Sobolev spaces are more appropriate to rough boundary value problems than Besov spaces. 
Recall from the theory of partial differential equations that $u$ is defined to be a weak solution to $Lu=\sum_{\abs{\alpha}=\abs{\beta}=m} \partial^\alpha(A_{\alpha\beta} \partial^\beta u)=0$ in~$\Omega$ provided $\sum_{\abs{\alpha}=\abs{\beta}=m} \int_\Omega \partial^\alpha \varphi \,A_{\alpha\beta}\,\partial^\beta u=0$ for all $\varphi\in C^\infty_0(\Omega)$. This definition is meaningful even for rough coefficients~$\mat A$ if $\nabla^m u$ is merely locally integrable. Some regularity results exist; however, for general coefficients, the most that may be said is that $\nabla^m u$ is locally square-integrable, or at best $(2+\varepsilon)$th-power integrable for some possibly small $\varepsilon>0$. (In the second-order case, this is the well known Caccioppoli inequality and Meyers's reverse H\"older inequality \cite{Mey63}. Both may be generalized to the higher order case; see \cite{Cam80,AusQ00,Bar16}.)

Thus, we wish to study functions $u$ with at most $m$ degrees of smoothness; we do not wish to consider $u\in \dot B^{p,p}_{m-1+\theta+1/p}(\Omega)$, for if $\theta+1/p>1$ then $u$ is required to be too smooth. See \cite[Chapter~10]{BarM16A} for further discussion. Thus, weighted Sobolev spaces are more appropriate to our applications than Besov spaces. (If $p>2$, then weighted \emph{averaged} Sobolev spaces with $q=2$ are even more appropriate, as the gradient of a solution $\nabla^m u$ is known a priori to be locally square-integrable but not locally $p$th-power integrable.)

We introduce the averages in the spaces $\dot W_{m,av}^{p,\theta,q}(\Omega)$ both because of the applications to partial differential equations mentioned above, and also in order to establish trace theorems for $p<1$. Observe that if $u\in W_m^{p,\theta}(\Omega)$, then $\nabla^m u$ is only locally in $L^p$; if $p<1$ then $\nabla^m u$ need not be locally integrable and it is not clear that the trace operator can be extended to $W_m^{p,\theta}(\Omega)$. In Lemma~\ref{lem:L:L1} below, we will see that if $u\in\dot W_{m,av}^{p,\theta,q}(\Omega)$ for some~$q\geq 1$, then $\nabla^m u$ is locally integrable up to the boundary provided $p>\pmin$, and so the trace operator is well-defined. We remark that the existing theorems for $p<1$ and $u\in B^{p,p}_{m-1+\theta+1/p}(\Omega)$ also require $p>\pmin$, and for precisely this reason: by standard embedding theorems (see, for example, \cite{RunS96}), the condition $p>\pmin$ is precisely the range of $p$ such that gradients of $B^{p,p}_{\theta+1/p}(\Omega)$-functions are locally integrable up to the boundary. 

We have included results in the case $p>1$. In the Neumann case these results follow by duality as usual. In the Dirichlet case, our results are not quite the same as but do owe a great deal to those of \cite{MazMS10}. To allow for a better treatment of unbounded domains such as the half-space, we have chosen to work with boundary data in homogeneous Besov spaces rather than inhomogeneous spaces, that is, to bound only $\Trace \nabla^{m-1} u$ and not the lower order derivatives $\Trace \nabla^k u$, $0\leq k\leq m-2$; this requires some additional careful estimates. See in particular the bound \eqref{eqn:extension:bound}; in the case of inhomogeneous data the earlier bound~\eqref{eqn:extension:bound:2} (the bound (7.48) in \cite{MazMS10}) suffices. We also work with weighted, averaged Sobolev spaces $\dot W_{m,av}^{p,\theta,q}(\Omega)$ rather than weighted Sobolev spaces $W_m^{p,\theta}(\Omega)$; this presents no additional difficulties in the case of extension theorems but does require some care in the case of trace theorems. 

The outline of this paper is as follows. In Section~\ref{sec:dfn} we will define our terminology and the function spaces under consideration, in particular boundary spaces of Whitney arrays. In Section~\ref{sec:L} we will establish some basic properties of the weighted averaged spaces $L^{p,\theta,q}_{av}$. We will prove Theorem~\ref{thm:dirichlet} in Sections~\ref{sec:extension:dirichlet} and~\ref{sec:trace:dirichlet}, and finally will prove Theorem~\ref{thm:neumann} in Sections~\ref{sec:extension:neumann} and~\ref{sec:trace:neumann}.

\section{Definitions}
\label{sec:dfn}

Throughout this paper, we will work in domains contained in~$\R^\dmn$.

We will generally use lowercase Greek letters to denote multiindices in $\N^\dmn$, where $\N$ denotes the nonnegative integers. If $\gamma$ is a multiindex, then we define $\abs{\gamma}$, $\partial^\gamma$ and $\gamma!$ in the usual ways, via $\abs{\gamma}=\gamma_1+\gamma_2+\dots+\gamma_\dmn$, $\partial^\gamma=\partial_{x_1}^{\gamma_1}\partial_{x_2}^{\gamma_2} \cdots\partial_{x_\dmn}^{\gamma_\dmn}$,
and $\gamma!=\gamma_1!\,\gamma_2!\dots\gamma_\dmn!$. 
If $\gamma=(\gamma_1,\dots,\gamma_\dmn)$ and $\delta=(\delta_1,\dots,\delta_\dmn)$ are two multiindices, then we say that $\delta\leq \gamma$ if $\delta_i\leq \gamma_i$ for all $1\leq i\leq\dmn$, and we say that $\delta<\gamma$ if in addition the strict inequality $\delta_i< \gamma_i$ holds for at least one such~$i$.

We will routinely deal with arrays $\arr F=\begin{pmatrix}F_{\gamma}\end{pmatrix}$ indexed by multiindices~$\gamma$ with $\abs{\gamma}=m$ for some~$m$.
In particular, if $\varphi$ is a function with weak derivatives of order up to~$m$, then we view $\nabla^m \varphi$ as such an array, with \begin{equation*}(\nabla^m\varphi)_{\gamma}=\partial^\gamma \varphi.\end{equation*}
The inner product of two such arrays of numbers $\arr F$ and $\arr G$ is given by
\begin{equation*}\bigl\langle \arr F,\arr G\bigr\rangle = 
\sum_{\abs{\gamma}=m}
\overline{F_{\gamma}}\, G_{\gamma}.\end{equation*}
If $\arr F$ and $\arr G$ are two arrays of functions defined in an open set $\Omega$ or on its boundary, then the inner product of $\arr F$ and $\arr G$ is given by
\begin{equation*}\bigl\langle \arr F,\arr G\bigr\rangle_\Omega = 
\sum_{\abs{\gamma}=m}
\int_{\Omega} \overline{F_{\gamma}}\, G_{\gamma}
\quad\text{or}\quad
\bigl\langle \arr F,\arr G\bigr\rangle_{\partial\Omega} = 
\sum_{\abs{\gamma}=m}
\int_{\partial\Omega} \overline{F_{\gamma}}\, G_{\gamma}\,d\sigma\end{equation*}
where $\sigma$ denotes surface measure. (In this paper we will consider  only domains  with rectifiable boundary.)

Recall from formula~\eqref{eqn:divergence-free} that, if $\arr G$ is an array of functions defined in an open set~$\Omega\subset\R^\dmn$ and indexed by multiindices~$\alpha$ with $\abs{\alpha}=m$, then $\Div_m\arr G=0$ in~$\Omega$ in the weak sense if and only if $\langle \nabla^m\varphi,\arr G\rangle_\Omega=0$ for  all smooth test functions~$\varphi$ supported in~$\Omega$.


If $E$ is a set, we let $\1_E$ denote the characteristic function of~$E$.
If $\mu$ is a measure and $E$ is a $\mu$-measurable set, with $\mu(E)<\infty$, we let \begin{equation*}\fint_E f\,d\mu=\frac{1}{\mu(E)}\int_E f\,d\mu.\end{equation*}

We let $L^p(U)$ and $L^\infty(U)$ denote the standard Lebesgue spaces with respect to either Lebesgue measure (if $U$ is a domain) or surface measure (if $U$ is a subset of the boundary of a domain). We let $C^\infty_0(U)$ denote the space of functions that are smooth and compactly supported in~$U$.

If $U$ is a connected open set, then we let the homogeneous Sobolev space $\dot W^p_m(U)$ be the space of equivalence classes of functions $u$ that are locally integrable in~$\Omega$ and have weak derivatives in $\Omega$ of order up to~$m$ in the distributional sense, and whose $m$th gradient $\nabla^m u$ lies in $L^p(U)$. Two functions are equivalent if their difference is a polynomial of order~$m-1$.
We impose the norm 
\begin{equation*}\doublebar{u}_{\dot W^p_m(U)}=\doublebar{\nabla^m u}_{L^p(U)}.\end{equation*}
Then $u$ is equal to a polynomial of order $m-1$ (and thus equivalent to zero) if and only if its $\dot W^p_m(U)$-norm is zero. 

We say that $u\in L^p_{loc}(U)$ or $u\in\dot W^p_{m,loc}(U)$ if $u\in L^p(V)$ or $u\in \dot W^p_m(V)$ for every bounded set $V$ with $\overline V\subset U$. In particular, if $U$ is a set and $\overline U$ is its closure, then functions in $L^p_{loc}(\overline U)$ are required to be locally integrable even near the boundary~$\partial U$; if $U$ is open this is not true of $L^p_{loc}(U)$.

If $Q\subset\R^\dmnMinusOne$ is a cube, then we let $\ell(Q)$ denote its side-length. 

Recall that a Banach space is a complete normed vector space. We define quasi-Banach spaces as follows.
\begin{dfn}\label{dfn:quasi-Banach}
We say that a vector space $B$ is a \emph{quasi-Banach space} if it possesses a quasi-norm $\doublebar{\,\cdot\,}$ and is complete with respect to the topology induced by that quasi-norm.

We say that $\doublebar{\,\cdot\,}$ is a quasi-norm on the vector space $B$ if 
\begin{itemize}
\item $\doublebar{b}=0$ if and only if $b=0$,
\item if $b\in B$ and $c\in\C$, then $\doublebar{cb}=\abs{c}\,\doublebar{b}$,
\item there is some constant $C_B\geq 1$ such that, if $b_1\in B$ and $b_2\in B$, then $\doublebar{b_1+b_2}\leq C_B\doublebar{b_1}+C_B\doublebar{b_2}$.
\end{itemize}
\end{dfn}
If $C_B=1$ then $B$ is a Banach space and its quasi-norm is a norm.

In this paper, rather than the quasi-norm inequality $\doublebar{b_1+b_2}\leq C_B\doublebar{b_1}+C_B\doublebar{b_2}$, we will usually use the $p$-norm inequality 
\begin{equation*}\doublebar{b_1+b_2}^p\leq \doublebar{b_1}^p+\doublebar{b_2}^p\end{equation*}
for some $0<p\leq 1$.
We remark that if $0<p\leq 1$ then the $p$-norm inequality implies the quasi-norm inequality with $C_B=2^{1/p-1}$. (The converse result, that is, that any quasi-norm is equivalent to a $p$-norm for $p$ satisfying $2^{1/p-1}=C_B$, is also true; see \cite{Aok42,Rol57}.)

If $B$ is a quasi-Banach space we will let $B^*$ denote its dual space. If $1\leq p\leq \infty$ then we will let $p'$ be the extended real number that satisfies $1/p+1/p'=1$. Thus, if $1\leq p<\infty$, then $(L^p(U))^*=L^{p'}(U)$.

In this paper we will work in Lipschitz domains,  defined as follows.

\begin{dfn}\label{dfn:domain}
	We say that the domain $V\subset\R^\dmn$ is a \emph{Lipschitz graph domain} if there is some Lipschitz function $\psi:\R^\dmnMinusOne\mapsto\R$ and some coordinate system such that
	\begin{equation*}
	V=\{(x',t):x'\in\R^\dmnMinusOne,\>t>\psi(x')\}.
	\end{equation*}
	We refer to $ M =\doublebar{\nabla\psi}_{L^\infty(\R^\dmnMinusOne)}$ as the \emph{Lipschitz constant of~$V$.}
	
	We say that the domain $\Omega$ is a \emph{Lipschitz domain} if either $\Omega$ is a Lipschitz graph domain, or if there is some positive scale $r=r_\Omega$, some constants $M>0$ and $c_0\geq 1$, and some finite set $\{x_j\}_{j=1}^{ n }$ of points with $x_j\in\partial\Omega$, such that the following conditions hold. First,
	\begin{equation*}
	\partial \Omega\subset \bigcup_{j=1}^{ n } B(x_j,r_j)
	\quad \text{for some }r_j\text{ with }\frac{1}{c_0}r<r_j<c_0 r.
	\end{equation*}
	Second, for each $x_j$, there is some Lipschitz graph domain $V_j$ with $x_j\in\partial V_j$ and with Lipschitz constant at most~$M$, such that 
	\begin{equation*}
	Z_j\cap\Omega = Z_j\cap V_j\end{equation*}
	where $Z_j$ is a cylinder of height $(8+8M)r_j$, radius~$2r_j$, and with axis parallel to the $t$-axis (in the coordinates associated with~$V_j$).
	
	If $\Omega$ is a Lipschitz graph domain let ${ n }=c_0=1$; otherwise let $M$, ${ n }$, $c_0$ be as above.
	We refer to the triple $(M,{ n },c_0)$ as the \emph{Lipschitz character} of~$\Omega$. We will occasionally refer to $r_\Omega$ as the \emph{natural length scale} of~$\Omega$; if $\Omega$ is a Lipschitz graph domain then $r_\Omega=\infty$.
\end{dfn}
Notice that if $\Omega$ is a Lipschitz domain, then either $\Omega$ is a Lipschitz graph domain or $\partial\Omega$ is bounded. If $\partial\Omega$ is bounded and connected, then the natural length scale $r_\Omega$ is comparable to $\diam\partial\Omega$. 

Throughout we will let $C$ denote a constant whose value may change from line to line, but that depends only on the ambient dimension, the number $m$ in the operators $\Tr_{m-1}^\Omega$ and $\M_m^\Omega$, and the Lipschitz character of any relevant domains; any other dependencies will be indicated explicitly. We say that $A\approx B$ if $A\leq CB$ and $B\leq CA$ for some such~$C$.

\subsection{Function spaces in domains and their traces}

The spaces $L^{p,\theta,q}_{av}(\Omega)$ and $\dot W_{m,av}^{p,\theta,q}(\Omega)$ were defined in the introduction; for completeness, we include their definitions here.

\begin{dfn}
Let $\Omega$ be a connected Lipschitz domain and let $0<p\leq \infty$, $1\leq q\leq \infty$ and $-\infty<\theta<\infty$.

We let $L^{p,\theta,q}_{av}(\Omega)$
be the space of locally integrable functions $H$ such that the $L_{av}^{p,\theta,q}(\Omega)$-norm given by formula~\eqref{eqn:norm:L}
\begin{equation*}
\doublebar{H}_{L_{av}^{p,\theta,q}(\Omega)}
= 
\biggl(\int_\Omega \biggl(\fint_{B(x,\dist(x,\partial\Omega)/2)} \abs{H(y)}^q\,dy \biggr)^{p/q}  \dist(x,\partial\Omega)^{p-1-p\theta}\,dx\biggl)^{1/p}
\end{equation*}
is finite.

If $m$ is a positive integer, we let $\dot W_{m,av}^{p,\theta,q}(\Omega)$ be the space of equivalence classes (given by adding polynomials of degree $m-1$) of functions $u$ that are locally integrable in~$\Omega$ and have weak derivatives in $\Omega$ of order up to~$m$ in the distributional sense, and for which $\nabla^m u\in L^{p,\theta,q}_{av}(\Omega)$.
\end{dfn}

Observe that if $p\geq 1$ then $L^{p,\theta,q}_{av}(\Omega)$ (and $\dot W_{m,av}^{p,\theta,q}(\Omega)$) is a Banach space. If $0<p<1$ then $L^{p,\theta,q}_{av}(\Omega)$ is a quasi-Banach space with a $p$-norm, that is,
\begin{equation*}\doublebar{F+G}_{L^{p,\theta,q}_{av}(\Omega)}^p \leq \doublebar{F}_{L^{p,\theta,q}_{av}(\Omega)}^p + \doublebar{G}_{L^{p,\theta,q}_{av}(\Omega)}^p.\end{equation*}

The main results of this paper concern the Dirichlet and Neumann trace operators acting on $\dot W_{m,av}^{p,\theta,q}(\Omega)$ and $L^{p,\theta,q}_{av}(\Omega)$, respectively. Thus we must define these trace operators.
We will see (Section~\ref{sec:L}) that if $0<\theta<1$ and $p>\pmin$, then $L^{p,\theta,q}_{av}(\Omega)\subset L^1_{loc}(\overline\Omega)$. It thus suffices to define the Dirichlet and Neumann traces of functions in $\dot W^1_{m,loc}(\overline\Omega)$ and $L^1_{loc}(\overline\Omega)$, respectively.

\begin{dfn}
If $u\in \dot W^1_{m,loc}(\overline\Omega)$ then the Dirichlet boundary values of $u$ are the traces of the $m-1$th derivatives; for ease of notation we define $\Tr_{m-1}^\Omega  u$ as the array given by
\begin{equation}
\label{eqn:Dirichlet}
\begin{pmatrix}\Tr_{m-1}^\Omega  u\end{pmatrix}_{\gamma}
=\Trace \partial^\gamma u
\quad\text{for all $\abs{\gamma}=m-1$.}\end{equation}

If $\arr G\in L^1_{loc}(\overline\Omega)$ satisfies $\Div_m\arr G=0$ in~$\Omega$ in the sense of formula~\eqref{eqn:divergence-free}, then the Neumann boundary values $\M_m^\Omega\arr G$ of~$\arr G$ are given by formula~\eqref{dfn:neumann}; as discussed in the introduction, $\M_m^\Omega\arr G$ is an equivalence class of distributions under the relation $\arr g\equiv \arr h$ if $\langle \arr g,\Tr_{m-1}^\Omega\varphi\rangle_{\partial\Omega}=\langle \arr h,\Tr_{m-1}^\Omega\varphi\rangle_{\partial\Omega}$ for all $\varphi\in C^\infty_0(\R^\dmn)$.

\end{dfn}

\subsection{Function spaces on the boundary}
\label{sec:space:dfn}

In this section, we will define Besov spaces and Whitney-Besov spaces; in Sections~\ref{sec:extension:dirichlet}--\ref{sec:trace:neumann} we will show that these spaces are, in fact, the Dirichlet and Neumann trace spaces of weighted averaged spaces.

The homogeneous Besov spaces $\dot B^{p,r}_\theta(\R^\dmnMinusOne)$ on a Euclidean space, for $-\infty<\theta<\infty$, $0<p\leq\infty$, and $0<r\leq\infty$, have traditionally been defined using the Fourier transform (the classic Littlewood-Paley definition); this definition may be found in many standard references, including \cite[Section~5.1.3]{Tri83} or \cite[Section~2.6]{RunS96}.  There are many equivalent characterizations, valid for different ranges of the parameters $p$, $r$ and~$\theta$. Because we wish to consider boundary values of functions in domains, we must generalize some of these characterizations from $\R^\dmnMinusOne$ to $\partial\Omega$ for more general Lipschitz domains~$\Omega$; the Littlewood-Paley characterization does not generalize easily to such regimes.

In this paper, we will be concerned only with the space $\dot B^{p,p}_{\theta-1}(\partial\Omega)$ (for Neumann boundary values) or $\dot B^{p,p}_{\theta}(\partial\Omega)$ (for Dirichlet boundary values), with $0<\theta<1$ and $\pmin<p\leq\infty$. It will be convenient to use different definitions in the cases $p\geq 1$ and $p\leq 1$, and in the case of positive and negative smoothness spaces; the four characterizations we use are as follows.

\begin{dfn}\label{dfn:besov}

Let $0<\theta<1$, and let $\Omega\subset\R^\dmn$ be a Lipschitz domain with connected boundary.

If $\pdmnMinusOne/(\dmnMinusOne+\theta)<p\leq \infty$, then we say that $a$ is a $\dot B^{p,p}_{\theta}(\partial\Omega)$-atom if there is some $x_0\in \partial\Omega$ and some $r>0$ such that
\begin{itemize}
\item $\supp a \subseteq B(x_0,r)\cap\partial\Omega$,
\item $\doublebar{a}_{L^\infty(\partial\Omega)} \leq r^{\theta-\pdmnMinusOne/p}$,
\item $\doublebar{\nabla a}_{L^\infty(\partial\Omega)} \leq r^{\theta-1-\pdmnMinusOne/p}$,
\end{itemize}
where the $L^\infty$ norm is taken with respect to surface measure~$d\sigma$ and where the gradient denotes the tangential gradient of $a$ along~$\partial\Omega$.
We say that $a$ is a $\dot B^{p,p}_{\theta-1}(\partial\Omega)$-atom if there is some $x_0\in \partial\Omega$ and some $r>0$ such that
\begin{itemize}
\item $\supp a \subseteq B(x_0,r)\cap\partial\Omega$,
\item $\doublebar{a}_{L^\infty(\partial\Omega)} \leq r^{\theta-1-\pdmnMinusOne/p}$,
\item $\int_{\partial\Omega} \,a(x) \,d\sigma(x)=0$.
\end{itemize}


If $p\leq 1$ then we let $\dot B^{p,p}_{\theta-1}(\partial\Omega)$ be the space of distributions
\begin{equation*}\dot B^{p,p}_{\theta-1}(\partial\Omega)=\Bigl\{\sum_{j=1}^\infty \lambda_j a_j:\lambda_j\in\C,\>a_j\text{ a $\dot B^{p,p}_{\theta-1}$-atom,}\>\sum_{j=1}^\infty \abs{\lambda_j}^p<\infty\Bigr\}\end{equation*}
with the norm
\begin{equation*}
\doublebar{f}_{\dot B^{p,p}_{\theta-1}(\partial\Omega)} = 
\inf\Bigl\{ \Bigl(\sum_{j=1}^\infty \abs{\lambda_j}^p\Bigr)^{1/p}: f=\sum_{j=1}^\infty \lambda_j a_j,\>a_j\text{ a $\dot B^{p,p}_{\theta-1}$-atom,}\>\lambda_j\in\C\Bigr\}.
\end{equation*}

If $p\leq 1$ then we let $\dot B^{p,p}_\theta(\partial\Omega)$ be the space of equivalence classes of locally integral functions modulo constants
\begin{equation*}\dot B^{p,p}_\theta(\partial\Omega)=\Bigl\{\Bigl(c_0+\sum_{j=1}^\infty \lambda_j a_j:c_0\in\C\Bigr),\>\lambda_j\in\C,\>a_j\text{ a $\dot B^{p,p}_\theta$-atom,}\>\sum_{j=1}^\infty \abs{\lambda_j}^p<\infty\Bigr\}\end{equation*}
and impose the norm
\begin{multline*}
\doublebar{f}_{\dot B^{p,p}_{\theta}(\partial\Omega)} = 
\inf\Bigl\{ \Bigl(\smash{\sum_{j=1}^\infty}\vphantom{\sum^\infty} \abs{\lambda_j}^p\Bigr)^{1/p}: \\f=c_0+\smash{\sum_{j=1}^\infty}\vphantom{\sum_1} \lambda_j a_j,\>c_0\in\C,\>a_j\text{ a $\dot B^{p,p}_\theta$-atom,}\>\lambda_j\in\C\Bigr\}.
\end{multline*}

If the $a_j$s are atoms and the $\lambda_j$s are complex numbers with $\sum_j \abs{\lambda_j}^p<\infty$, then the sums $\sum_j \lambda_j a_j$ converge to distributions or functions; see Remark~\ref{rmk:atoms:converge}.

If $1< p\leq \infty$ and $0<\theta<1$, then we let ${\dot B^{p,p}_{\theta}(\partial\Omega)}$ be the set of all equivalence classes modulo constants of locally integrable functions $f$ defined on $\partial\Omega$ for which the ${\dot B^{p,p}_{\theta}(\partial\Omega)}$-norm given by
\begin{equation}
\label{eqn:Slobodekij}
\doublebar{f}_{\dot B^{p,p}_{\theta}(\partial\Omega)} = \biggl(\int_{\partial\Omega} \int_{\partial\Omega} \frac{\abs{f(x)-f(y)}^p}{\abs{x-y}^{\dmnMinusOne+p\theta}}\,d\sigma(x)\,d\sigma(y)\biggr)^{1/p}\end{equation}
is finite. If $p=\infty$ we modify the definition appropriately by taking the $L^\infty$ norm; then $\dot B^{p,p}_{\theta}(\partial\Omega)=\dot C^\theta(\partial\Omega)$, the space of H\"older continuous functions with exponent~$\theta$. 

Finally, if $1<p\leq \infty$ and $-1<\theta-1<0$, then we let $\dot B^{p,p}_{\theta-1}(\partial\Omega)$ be the dual space $\bigl(\dot B^{p',p'}_{1-\theta}(\partial\Omega)\bigr)^*$, where $1/p+1/p'=1$.

\end{dfn}

\begin{rmk}\label{rmk:atoms:converge}
The sums of atoms $\sum_{j=1}^\infty \lambda_j a_j$ are meaningful as locally integrable functions (if the $a_j$s are $\dot B^{p,p}_\theta$-atoms) or as distributions (if the $a_j$s are $\dot B^{p,p}_{\theta-1}$-atoms). 


Specifically, 
observe that if $\pmin<p\leq 1$ and $0<\theta<1$, then any $\dot B^{p,p}_\theta(\partial\Omega)$-atom is in $L^{\tilde p}(\partial\Omega)$ with uniformly bounded norm (depending on the Lipschitz constants of~$\Omega$), where $\tilde p=p\pdmnMinusOne/(\dmnMinusOne-p\theta)$; observe $\tilde p>1$. If $p\leq 1$ and $\sum_{j=1}^\infty \abs{\lambda_j}^p<\infty$, then $\sum_{j=1}^\infty \abs{\lambda_j}<\infty$. Thus, if $a_j$ is a $\dot B^{p,p}_\theta$-atom for each~$j$, then the infinite sum $\sum_{j=1}^\infty \lambda_j a_j$ converges in the $L^{\tilde p}$-norm; thus, that sum denotes a unique locally $L^1$ function.

If $a$ is a $\dot B^{p,p}_{\theta-1}(\partial\Omega)$-atom for some $\pmin<p\leq 1$ and $\theta-1<0$, then for any smooth function~$\varphi$, we have that by the Poincar\'e inequality
\begin{equation*}\abs[bigg]{\int_{\partial\Omega} \varphi\,a\,d\sigma}\leq C\doublebar{\nabla\varphi}_{L^{\tilde p'}(B(x_0,r)\cap\partial\Omega)},\end{equation*}
where again $\tilde p=p\pdmnMinusOne/(\dmnMinusOne-p\theta)$ and where $1/\tilde p+1/\tilde p'=1$.
Thus, such atoms may be viewed as distributions. If $\sum_{j=1}^\infty \abs{\lambda_j}<\infty$, and if $a_j$ is an atom for each~$j$, then the infinite sum $\sum_{j=1}^\infty \lambda_j a_j$ converges to a distribution (that is, the sum $\sum_{j=1}^\infty \lambda_j \langle \varphi,a_j\rangle_{\partial\Omega}$ converges absolutely for any smooth function~$\varphi$).
\end{rmk}

\begin{rmk}\label{rmk:quasi-Banach} 
If $0<\theta<1$ and $\pmin<p\leq \infty$, then $\dot B^{p,p}_\theta(\partial\Omega)$ and $\dot B^{p,p}_{\theta-1}(\partial\Omega)$ are quasi-Banach spaces; if $p\geq 1$ they are Banach spaces.

%
%
\end{rmk}

\begin{rmk}
The duality characterization of the negative smoothness spaces for $p>1$ is well known; see, for example, \cite[Sections~2.11 and~5.2.5]{Tri83}. Recall that in some sense Neumann boundary data is dual to Dirichlet boundary data, and so a duality characterization is appropriate. However, the space $\dot B^{p,p}_{\theta-1}(\R^\dmnMinusOne)$, for $p\leq 1$, is not the dual of a naturally arising space; thus we need an alternative characterization.
The atomic characterization comes from the atomic decomposition of Frazier and Jawerth in \cite{FraJ85}. If $p\leq 1$, then atomic characterizations are very convenient, and so we use them to define $\dot B^{p,p}_\theta(\partial\Omega)$ as well as $\dot B^{p,p}_{\theta-1}(\partial\Omega)$. Atoms are less convenient in the case $p>1$, and so in this case we use another characterization.
The norm \eqref{eqn:Slobodekij} comes from the definition of Slobodekij spaces, one of many function spaces that may be realized as a special case of Besov or Triebel-Lizorkin spaces; see \cite[Section~5.2.3]{Tri83}. 
\end{rmk}

\begin{rmk}\label{rmk:p=1} If $p=1$ and $0<\theta<1$, then we shall see that the atomic norm and the norm~\eqref{eqn:Slobodekij} are equivalent. 
Specifically, in Remark~\ref{rmk:whitney:0:besov} we shall see that $\dot B^{p,p}_\theta(\partial\Omega)=\dot W\!A^p_{0,\theta}(\partial\Omega)$, where the Whitney space $\dot W\!A^p_{m-1,\theta}(\partial\Omega)$ will be defined in Definition~\ref{dfn:whitney}. The $m=1$, $p=1$ case of  Theorem~\ref{thm:extension} will imply that if $ \varphi \in \dot W\!A^1_{0,\theta}(\partial\Omega)$ then $\varphi = \Trace^\Omega \Phi$ for some $\Phi\in \dot W^{1,\theta,q}_{1,av}(\Omega)$ that satisfies both of the inequalities 
\begin{align*}\doublebar{\Phi}_{\dot W^{1,\theta,q}_{1,av}(\Omega)}
&\leq C \int_{\partial\Omega} \int_{\partial\Omega} \frac{\abs{ \varphi(x)-\varphi(y)}}{\abs{x-y}^{\dmnMinusOne+\theta}}\,d\sigma(x)\,d\sigma(y)
,\\
\doublebar{\Phi}_{\dot W^{1,\theta,q}_{1,av}(\Omega)}
&\leq C \inf\Bigl\{\sum_j\abs{\lambda_j}: \varphi= c_0+\sum_j \lambda_j\, a_j,\> c_0\text{ constant},\>  a_j\text{  atoms}\Bigr\}.\end{align*}
The $m=1$, $p=1$ case of Theorem~\ref{thm:trace} will establish the converses, that is, that if $\Phi\in {\dot W^{1,\theta,q}_{1,av}(\Omega)}$ then 
\begin{align*}\int_{\partial\Omega} \int_{\partial\Omega} \frac{\abs{\Trace^\Omega \Phi(x)-\Trace^\Omega \Phi(y)}}{\abs{x-y}^{\dmnMinusOne+\theta}}\,d\sigma(x)\,d\sigma(y)
&\leq C \doublebar{\Phi}_{\dot W^{1,\theta,q}_{1,av}(\Omega)}
,\\
\inf\Bigl\{\sum_j\abs{\lambda_j}:\Trace^\Omega \Phi= c_0+\sum_j \lambda_j\, a_j,\> c_0\text{ constant},\>  a_j\text{  atoms}\Bigr\}
&\leq C\doublebar{\Phi}_{\dot W^{1,\theta,q}_{1,av}(\Omega)}
.\end{align*}
Combining these results yields the equivalence of norms
\begin{multline*}\int_{\partial\Omega} \int_{\partial\Omega} \frac{\abs{ \varphi(x)-\varphi(y)}}{\abs{x-y}^{\dmnMinusOne+\theta}}\,d\sigma(x)\,d\sigma(y) \approx \doublebar{\Phi}_{\dot W^{1,\theta,q}_{1,av}(\Omega)}
\\\approx
\inf\Bigl\{\sum_j\abs{\lambda_j}: \varphi= c_0+\sum_j \lambda_j\, a_j,\> c_0\text{ constant},\>  a_j\text{  atoms}\Bigr\}\end{multline*}
for any $\varphi$ such that either side is finite.

Although we shall not use this fact, we mention that it is possible to establish this equivalence in other ways: controlling the norm~\eqref{eqn:Slobodekij} by the atomic norm is straightforward if $p\leq1$, and the reverse implication in the case where $\Omega$ is a half-space and so $\partial\Omega=\R^\dmnMinusOne$ denotes Euclidean space is a main result of \cite{FraJ85}.
\end{rmk}


Now, recall that we seek spaces of Dirichlet traces $\{\Tr_{m-1}^\Omega  u:  u\in \dot W_{m,av}^{p,\theta,q}(\Omega)\}$; in particular, we seek spaces of boundary data that may be extended to such functions. But if $m\geq 2$, then $\Tr_{m-1}^\Omega u$ is not a function; it is an array of functions that must satisfy certain compatibility conditions. Thus, if $r$ is the number of multiindices $\gamma$ of length $m-1$, we do not expect to be able to extend an arbitrary element of $(\dot B^{p,p}_\theta(\partial\Omega))^r$ to a $\dot W_{m,av}^{p,\theta,q}(\Omega)$-function; extension will only be possible in a distinguished subspace, called a Whitney-Besov space.

\begin{dfn}\label{dfn:whitney} Suppose that $\Omega\subset\R^\dmn$ is a Lipschitz domain, and  consider arrays of functions $\arr f=\begin{pmatrix}f_\gamma\end{pmatrix}_{\abs{\gamma}={m-1}}$, where $f_\gamma:\partial\Omega\mapsto \C$.

If $0<\theta<1$ and 
$\pdmnMinusOne/(\dmnMinusOne+\theta)< p<\infty$, 
then we let the homogeneous Whitney-Besov space $\dot W\!A^p_{m-1,\theta}(\partial\Omega)$ be the closure of the set of arrays
\begin{equation}\label{eqn:Whitney:dense}
\bigl\{\arr\psi=\Tr_{m-1}^\Omega\Psi :\nabla^m\Psi\in L^\infty(\R^\dmn),\>\Psi \text{ compactly supported}\bigr\}\end{equation}
in $\dot B^{p,p}_\theta(\partial\Omega)$, under the (quasi)-norm
\begin{equation*}
\doublebar{\arr \psi}_{\dot W\!A^p_{m-1,\theta}(\partial\Omega)}
=
\sum_{\abs{\gamma}={m-1}} \doublebar{\psi_\gamma}_{\dot B^{p,p}_\theta(\partial\Omega)}
.\end{equation*}
Notice that $\dot W\!A^p_{m-1,\theta}(\partial\Omega)$ is a subspace of $(\dot B^{p,p}_{\theta}(\partial\Omega))^r$, where $r$ is the number of multiindices~$\gamma$ of length~${m-1}$.

If $0<\theta<1$ and $p=\infty$, then we let 
$\dot W\!A^p_{m-1,\theta}(\partial\Omega)=\dot W\!A^\infty_{m-1,\theta}(\partial\Omega)$ be the set of arrays
\begin{equation*}
\bigl\{\arr\psi=\Tr_{m-1}^\Omega\Psi :\nabla^{m-1}\Psi\in \dot C^\theta(\Omega)\bigr\}
\end{equation*}
equipped with the norm 
\begin{equation*}\doublebar{\arr\psi}_{\dot W\!A^\infty_{m-1,\theta}(\partial\Omega)}=
\doublebar{\arr \psi}_{\dot B^{\infty,\infty}_\theta(\partial\Omega)} =
\sup_{\abs{\gamma}={m-1}}
\sup_{\substack{x\neq y\\ x,y\in\partial\Omega} } \frac{\abs{\psi_\gamma(x)-\psi_\gamma(y)}}{\abs{x-y}^\theta}.
\end{equation*}


When no ambiguity arises we will omit the $m-1$ subscript. 
\end{dfn}

\begin{rmk}\label{rmk:whitney:0:besov}
We remark that if $m=1$ then $\dot W\!A^p_{m-1,\theta}(\partial\Omega)=\dot W\!A^p_{0,\theta}(\partial\Omega)=\dot B^{p,p}_\theta(\partial\Omega)$. 

The relation $\dot W\!A^p_{0,\theta}(\partial\Omega)\subseteq\dot B^{p,p}_\theta(\partial\Omega)$ is clear from the definition. Thus we need only show the reverse inclusion.

If $p=\infty$, the reverse inclusion is merely the statement that any H\"older continuous function defined on~$\partial\Omega$ has a H\"older continuous extension to $\R^\dmn$. If $p\leq 1$ and $0<\theta<1$, then all atoms lie in the space given in formula~\eqref{eqn:Whitney:dense} and so this space is dense in $\dot B^{p,p}_\theta(\partial\Omega)$ as well as $\dot W\!A^p_{0,\theta}(\partial\Omega)$. Finally, if $1<p<\infty$ then the argument that functions with bounded derivative (and in fact smooth functions) are dense in $\dot B^{p,p}_\theta(\partial\Omega)$ is similar to the argument that they are dense in $L^p(\partial\Omega)$.
\end{rmk}

We are also interested in the spaces of Neumann traces of (divergence-free) arrays $\arr G\in L^{p,\theta,q}_{av}(\Omega)$. Recall that in this case, the main complication is that $\M_m^\Omega\arr G$ is only defined up to adding arrays $\arr g$ that satisfy $\langle\Tr_{m-1}^\Omega\varphi,\arr g\rangle_{\partial\Omega}=0$. This may be dealt with by simply defining $\dot N\!A^p_{\theta-1}(\partial\Omega)$ as a quotient space.

\begin{dfn}\label{dfn:whitney:neumann} 
Let $\Omega\subset\R^\dmn$ be a Lipschitz domain with connected boundary, let $0<\theta<1$, and let $\pmin<p\leq\infty$. Let $r$ be the number of multiindices of length~${m-1}$.

Then $\dot N\!A^{p}_{\theta-1}(\partial\Omega)=\dot N\!A^{p}_{m-1,\theta-1}(\partial\Omega)$ is the quotient space of $(\dot B^{p,p}_{\theta-1}(\partial\Omega))^r$ under the equivalence relation
\begin{equation*}\arr g\equiv\arr h \text{ if and only if } \langle \Tr_{m-1}^\Omega\varphi,\arr g\rangle_{\partial\Omega} = \langle \Tr_{m-1}^\Omega\varphi,\arr h\rangle_{\partial\Omega} \text{ for all $\varphi\in C^\infty_0(\R^\dmn)$.}\end{equation*}
\end{dfn}

Observe that by the duality or atomic characterization of $\dot B^{p,p}_{\theta-1}(\partial\Omega)$, if $\varphi$ is smooth and compactly supported then $\abs{ \langle \Tr_{m-1}^\Omega\varphi,\arr g\rangle_{\partial\Omega}}<\infty$ for all $\arr g\in \dot B^{p,p}_{\theta-1}(\partial\Omega)$; thus, this equivalence relation is meaningful in $(\dot B^{p,p}_{\theta-1}(\partial\Omega))^r$.

\begin{rmk} \label{rmk:neumann-whitney:p>1} We have an alternative characterization of $\dot N\!A^{p}_{\theta-1}(\partial\Omega)$ in the case $p>1$. In this case, $1\leq p'<\infty$, and by the definitions of $\dot B^{p,p}_{\theta-1}(\partial\Omega)$ and of $\dot W\!A^{p'}_{1-\theta}(\partial\Omega)$, we have that $\dot N\!A^{p}_{\theta-1}(\partial\Omega)$ is the dual space to $\dot W\!A^{p'}_{1-\theta}(\partial\Omega)$.
\end{rmk}

\section{Properties of function spaces}
\label{sec:L}

In this section we will establish a few properties of the spaces {$L_{av}^{p,\theta,q}(\Omega)$}; we will need these results to establish the trace and extension results of Sections~\ref{sec:extension:dirichlet}--\ref{sec:trace:neumann}.

Let $\Omega$ be a Lipschitz domain, and let $\mathcal{W}$ be a grid of dyadic Whitney cubes; then $\Omega=\cup_{Q\in\mathcal{W}} Q$, the cubes in~$\mathcal{W}$ have pairwise-disjoint interiors, and if $Q\in\mathcal{W}$ then the side-length $\ell(Q)$ satisfies $\ell(Q)\approx\dist(Q,\partial\Omega)$.

If $\arr H\in L_{av}^{p,\theta,q}(\Omega)$ for $0<p<\infty$, $\theta\in\R$ and $1\leq q\leq \infty$, then
\begin{equation}
\label{eqn:L:norm:whitney}
\doublebar{\arr H}_{L_{av}^{p,\theta,q}(\Omega)}
\approx \biggl( \sum_{Q\in\mathcal{W}} \biggl(\fint_Q \abs{\arr H}^q\biggr)^{p/q} \ell(Q)^{\dmnMinusOne+p-p\theta}\biggr)^{1/p}
\end{equation}
where the comparability constants depend on~$\Omega$, $p$, $q$, $\theta$, and the comparability constants for Whitney cubes in the relation $\ell(Q)\approx\dist(Q,\partial\Omega)$. (This equivalence is still valid in the case $p=\infty$ if we replace the sum over cubes by an appropriate supremum.) Notice that this implies that we may replace the balls $B(x,\dist(x,\partial\Omega)/2)$ in the definition \eqref{eqn:norm:L} of $L_{av}^{p,\theta,q}(\Omega)$ by balls $B(x,a\dist(x,\partial\Omega))$ for any $0<a<1$, and produce an equivalent norm.

This gives us a number of results. First, if $p=q$ then $L_{av}^{p,\theta,p}(\Omega)$ is the weighted but not averaged Sobolev space given by
\begin{equation} \doublebar{\arr H}_{L_{av}^{p,\theta,p}(\Omega)}\approx \biggl(\int_\Omega \abs{\arr H(x)}^p\,\dist(x,\partial\Omega)^{p-1-p\theta}\,dx\biggr)^{1/p}. \end{equation}
In particular, if $\theta=1-1/p$ then $L_{av}^{p,1-1/p,p}(\Omega)=L^p(\Omega)$.

Second, if $1\leq q<\infty$ and $1\leq p< \infty$, then we have the duality relation
\begin{equation}
(L_{av}^{p,\theta,q}(\Omega))^*=L_{av}^{p',1-\theta,q'}(\Omega)
\end{equation}
where $1/p+1/p'=1/q+1/q'=1$. 

The final result we will prove in this section generalizes a result of \cite{BarM16A}, in which the spaces $L_{av}^{p,\theta,q}(\R^\dmn_+)$, where $\R^\dmn_+$ is the upper half-space, were investigated.

To state this result, we establish some notation. 
Suppose that $V=\{(x',t):t>\psi(x')\}$ is a Lipschitz graph domain. For each cube $Q\subset\R^\dmnMinusOne$, define
\begin{align}
\label{eqn:tent}
T(Q)&=\{(x',t):x'\in Q,\psi(x')<t<\psi(x')+8\ell(Q)\},
\\
\label{eqn:whitney}
W(Q)&=\{(x',t):x'\in Q,\psi(x')+4\ell(Q)<t<\psi(x')+8\ell(Q)\}
.\end{align}

The regions $W(Q)$ and $T(Q)$ are shown in Figure~\ref{fig:W}.

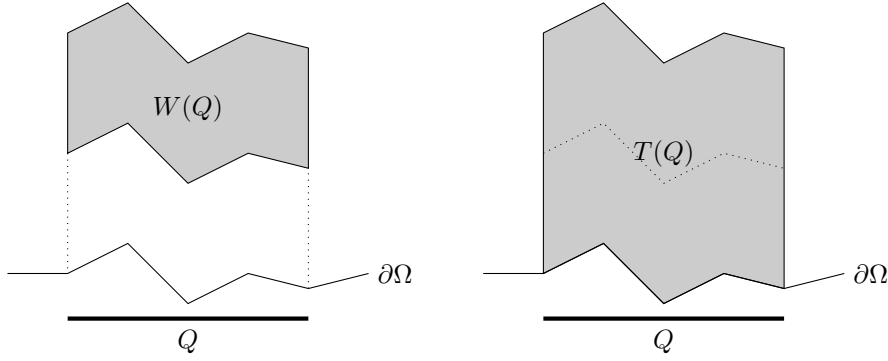
\begin{figure}
\begin{center}
\def\shape{{(1,1)},{(1,-2)},{(1,1)},{(1,-0.5)}}
\def\verticalShift{(0,4)}

\begin{tikzpicture}[xscale=0.8, yscale = 0.4]

	\foreach \point [count=\n] in \shape {
			\node[coordinate] (d-\n) at \point {};
			}
	\pgfmathtruncatemacro{\size}{\n}

	\draw[fill=white!80!black] \verticalShift 
		  \foreach \i in {1,...,\size} {-- ++(d-\i)} 
		  -- ++\verticalShift 
		  \foreach \i in {\size,...,1} {-- ++($-1*(d-\i)$)} 
		  -- cycle; 
	\node at (2,5.5) {$W(Q)$};

	\draw (-1,0)--(0,0) \foreach \i in {1,...,\size} {-- ++(d-\i)}--(5,0) node [right] {$\partial\Omega$};

	\begin{scope}[shift={(0,-3/2)}]
	\draw [ultra thick] (0,0) -- node [below] {$Q$} (4,0);
	\end{scope}

	\draw[dotted] \verticalShift--(0,0); 
	\path (0,0) \foreach \i in {1,...,\size} {-- ++(d-\i)} coordinate (w);
	\draw[dotted] (w) -- ++\verticalShift;


\end{tikzpicture}
\qquad
\begin{tikzpicture}[xscale=0.8, yscale = 0.4]

	\foreach \point [count=\n] in \shape {
			\node[coordinate] (d-\n) at \point {};
			}
	\pgfmathtruncatemacro{\size}{\n}

	\draw[fill=white!80!black] (0,0)
		  \foreach \i in {1,...,\size} {-- ++(d-\i)} 
		  -- ++\verticalShift 
		  -- ++\verticalShift 
		  \foreach \i in {\size,...,1} {-- ++($-1*(d-\i)$)} 
		  -- cycle; 
	\node at (2,4) {$T(Q)$};

	\draw (-1,0)--(0,0) \foreach \i in {1,...,\size} {-- ++(d-\i)}--(5,0) node [right] {$\partial\Omega$};

	\begin{scope}[shift={(0,-3/2)}]
	\draw [ultra thick] (0,0) -- node [below] {$Q$} (4,0);
	\end{scope}

	\draw[dotted] \verticalShift
		  \foreach \i in {1,...,\size} {-- ++(d-\i)} ;

\end{tikzpicture}
\caption{The regions $W(Q)\subset T(Q)$ and $T(Q)$. (The vertical axis has been compressed.)}
\label{fig:W}
\end{center}

\end{figure}

If $j$ is an integer, let $\mathcal{H}_j$ 
be the set of all open cubes in $\R^\dmnMinusOne$ of side-length $2^j$ whose vertices are integer multiples of $2^j$. Then the cubes in $\mathcal{H}_j$ are pairwise-disjoint and $\cup_{Q\in\mathcal{H}_j} \overline{Q}=\R^\dmnMinusOne$. Let $\mathcal{H}=\cup_{j=-\infty}^\infty \mathcal{H}_j$.

We claim that $\{W(Q):Q\in\mathcal{H}\}$ has many of the useful properties of a decomposition of $V$ into Whitney cubes (as in the norm~\eqref{eqn:L:norm:whitney}). It is clear that the diameter of $W(Q)$ is comparable to the distance from $W(Q)$ to~$\partial V$. We claim that if $Q$, $R\in \mathcal{H}$ with $Q\neq R$ then $W(Q)$ and $W(R)$ are disjoint, and that $V=\cup_{Q\in\mathcal{H}}\overline{W(Q)}$.

To see this, observe that if $Q\in \mathcal{H}_j$ for some integer~$j$ then $W(Q)= \{(x,t)\in V: x\in Q,\>\psi(x)+2^{j+2}< t < \psi(x)+2^{j+3}\}$. If $Q\in \mathcal{H}_j$ and $R\in\mathcal{H}_k$ for some integers $j\neq k$, then $W(Q)$ and $W(R)$ are clearly disjoint; otherwise, $Q\in \mathcal{H}_j$ and $R\in \mathcal{H}_j$ and so $Q$ and $R$ are disjoint, and thus $W(Q)$, $W(R)$ are disjoint.

Furthermore, $\cup_{Q\in \mathcal{H}_j} \overline{W(Q)} = \{(x,t)\in V: \psi(x)+2^{j+2}\leq t \leq \psi(x)+2^{j+3}\}$, and so $V=\cup_{j=-\infty}^\infty \cup_{Q\in \mathcal{H}_j} \overline{W(Q)} =\cup_{Q\in\mathcal{H}}\overline{W(Q)}$.

Thus, the set $\{W(Q):Q\in\mathcal{H}\}$ has many of the useful properties of a decomposition into Whitney cubes. In particular, we have a result similar to the estimate~\eqref{eqn:L:norm:whitney} in terms of such regions: if $\arr H\in {L_{av}^{p,\theta,q}(\Omega)}$, then
\begin{equation}\label{eqn:L:norm:dyadic}
\doublebar{\arr H}_{L_{av}^{p,\theta,q}(\Omega)}^p
\approx
\sum_{Q\in\mathcal{H}} \biggl(\fint_{W(Q)} \abs{\arr H}^q\biggr)^{p/q} \ell(Q)^{\dmnMinusOne+p-p\theta}
.\end{equation}

The following result states essentially that we may replace the sets $W(Q)$ by the sets $T(Q)$ in the norm~\eqref{eqn:L:norm:dyadic}. In particular, this implies that the integral over a tent $T(Q)$ is finite, and so
$L_{av}^{p,\theta,q}(V)$-functions are locally integrable up to the boundary; this second result extends from Lipschitz graph domains $V$ to general Lipschitz domains~$\Omega$.

\begin{lem}\label{lem:L:L1}
Let $V$ be a Lipschitz graph domain and let $\mathcal{H}$, $T(Q)$, and $W(Q)$ be as above.


Let $\theta\in\R$. Then, if $0<p\leq q$ and $1/q> (\dmnMinusOne+p-p\theta)/\pdmn p$, or if $0<q\leq p$ and $1/q>1-\theta$, then
\begin{equation}\label{eqn:L:L1:3}
\sum_{Q\in\mathcal{H}} \biggl(\int_{T(Q)} \abs{\arr H}^{q}\biggr)^{p/q} \ell(Q)^{\dmnMinusOne+p-p\theta-(p/q)\pdmn}
\approx
\doublebar{\arr H}_{L_{av}^{p,\theta,q}(V)}^p
.\end{equation}
In particular, if $\theta>0$, $\pmin<p\leq\infty$, and $q\geq 1$, then
\begin{equation}\label{eqn:L:L1:2}
\sum_{Q\in\mathcal{H}} \biggl(\int_{T(Q)} \abs{\arr H}\biggr)^p \ell(Q)^{\dmnMinusOne-p\theta-p\pdmnMinusOne}
\approx
\doublebar{\arr H}_{L_{av}^{p,\theta,1}(V)}^p
\leq \doublebar{\arr H}_{L_{av}^{p,\theta,q}(V)}^p
.\end{equation}

More generally, suppose that $\Omega\subset\R^\dmn$ is a Lipschitz domain, and that 
$\theta>0$, $q\geq 1$, and $\pmin<p\leq\infty$. 
If $\arr H\in L_{av}^{p,\theta,q}(\Omega)$, if $x_0\in\partial\Omega$, and if $R>0$, then
\begin{equation}\label{eqn:L:L1:1}
\doublebar{\arr H}_{L^1(B(x_0,R)\cap\Omega)} 
\leq C\doublebar{\arr H}_{L_{av}^{p,\theta,q}(\Omega)} R^{\dmnMinusOne+\theta-\pdmnMinusOne/p} 
.\end{equation}

\end{lem}

\begin{proof}
If $\Omega=\R^\dmn_+$ is a half-space, then the bound~\eqref{eqn:L:L1:3} is \cite[Theorem~6.1]{BarM16A}, and the bound \eqref{eqn:L:L1:1} follows immediately. Let $\psi$ be a Lipschitz function; by making the change of variables $(x',t)\mapsto(x',t-\psi(x'))$, we see that the lemma is still true in the domain $\Omega=\{(x',t):t>\psi(x')\}$, that is, in any Lipschitz graph domain.

There remains the case where $\Omega$ is a domain with compact boundary. (In this case we prove only the estimate~\eqref{eqn:L:L1:1}, and not the estimates \eqref{eqn:L:L1:3} or~\eqref{eqn:L:L1:2}.) 
We may control the $L^1$ norm of~$\arr H$ near $\partial\Omega$ using the bound for Lipschitz graph domains. If $R$ is sufficiently small (compared with the natural length scale $r=r_\Omega$ of Definition~\ref{dfn:domain}), this completes the proof. 

If $R>r_\Omega/C$, then we may control the $L^1$ norm of~$\arr H$ far from $\partial\Omega$ by using the norm~\eqref{eqn:L:norm:whitney} and the observation that there are at most $C(1+r_\Omega/2^j)^\dmn$ dyadic Whitney cubes of side-length $2^j$. 
\end{proof}

We have shown that if $0<\theta<1$ and $\pmin<p$, then $\dot W_{m,av}^{p,\theta,q}(\Omega)$-functions are necessarily $\dot W^1_{m,loc}(\overline\Omega)$-functions, and so $\Tr_{m-1}^\Omega u$ and $\M_m^\Omega \arr G$ are meaningful if $u\in\dot W_{m,av}^{p,\theta,q}(\Omega)$ and $\arr G \in L_{av}^{p,\theta,q}(\Omega)$. 

If $\theta\leq 0$ or $p\leq \pmin$, then this is not true and so trace theorems are not meaningful.
Conversely, if $\theta\geq 1$, then $\Tr_{m-1}^\Omega \vec u$ is constant for all $\vec u \in \dot W_{m,av}^{p,\theta,p}(\Omega)$, and so we do not expect an interesting theory of traces of functions $\vec u \in \dot W_{m,av}^{p,\theta,q}(\Omega)$.

Thus, for the remainder of this paper, we will only consider $\theta\in (0,1)$ and $p>\pmin$.

%
%

\subsection{Density of smooth functions in weighted averaged spaces}

The main result of this section is Theorem~\ref{thm:smooth:dense}, the density of smooth functions in the spaces $\dot W_{m,av}^{p,\theta,q}(\Omega)$. We will first prove the following Poincar\'e-style inequality; it will allow us to control the lower-order derivatives of a function in $\dot W_{m,av}^{p,\theta,q}(\Omega)$ by its $\dot W_{m,av}^{p,\theta,q}(\Omega)$-norm.

\begin{lem}\label{lem:Poincare:L}
Let $\Omega=\{(x',t):t>\psi(x')\}$ be a Lipschitz graph domain. Let $Q\subset\R^\dmnMinusOne$ be a cube, and let $T(Q)$, $W(Q)$ be as in formulas~\eqref{eqn:tent} and~\eqref{eqn:whitney}.
Suppose that 
$\1_{T(Q)}\nabla^m u \in L_{av}^{p,\theta,q}(\Omega)$. Let $u_Q$ be the polynomial of order $m-1$ that satisfies 
\begin{equation*}\fint_{W(Q)} \nabla^k (u-u_Q)=0\quad\text{for all integers $k$ with $0\leq k\leq m-1$}.\end{equation*}
If $1\leq q\leq\infty$ and $0<p\leq\infty$, then
\begin{equation}\label{eqn:poincare}
\doublebar{\1_{T(Q)}\nabla^k (u-u_Q)}_{L_{av}^{p,\theta,q}(\Omega)}
\leq C \ell(Q)^{m-k}\doublebar{\1_{T(Q)}\nabla^m u}_{L_{av}^{p,\theta,q}(\Omega)}.\end{equation}
If $p>\pmin$ and $\Tr_k^\Omega u=0$ along $\partial\Omega\cap\partial T(Q)$ for all $0\leq k\leq m-1$, then we have that
\begin{equation}\label{eqn:poincare:boundary}
\doublebar{\1_{T(Q)}\nabla^k u}_{L_{av}^{p,\theta,q}(\Omega)}
\leq C \ell(Q)^{m-k}\doublebar{\1_{T(Q)}\nabla^m u}_{L_{av}^{p,\theta,q}(\Omega)}.\end{equation}
\end{lem}

\begin{proof}
We begin with the bound \eqref{eqn:poincare}.
Without loss of generality we assume $u_Q\equiv 0$. Choose some multiindex $\gamma$ with $\abs{\gamma}=k\leq m-1$, and for any cube~$R\subset\R^\dmnMinusOne$, let $u_{\gamma,R}=\fint_{W(R)} \partial^\gamma u$; notice that $u_{\gamma,Q}=0$. The $k=m$ case is immediate; we will use induction to generalize to $k<m$.

Let $\mathcal{G}_0=\{Q\}$, and for each $j>0$, let $\mathcal{G}_j$ be the set of open dyadic subcubes of $Q$ of side-length $2^{-j}\ell(Q)$; then $\abs{\mathcal{G}_j}=2^{j\pdmnMinusOne}$ and $\cup_{R\in\mathcal{G}_j} \overline R = \overline Q$.
Let $\mathcal{G}=\cup_{j=0}^\infty \mathcal{G}_j$. In particular, if $\mathcal{H}$ is as in Lemma~\ref{lem:L:L1} and $Q\in\mathcal{H}$, then $\mathcal{G}=\{R\in \mathcal{H}:R\subseteq Q\}$.

By formula~\eqref{eqn:L:norm:dyadic},
\begin{equation}\label{eqn:dyadic:norm:local}
\doublebar{\1_{T(Q)}\arr H}_{L_{av}^{p,\theta,q}(\Omega)}^p
\approx
\sum_{R\in\mathcal{G}} \biggl(\fint_{W(R)} \abs{\arr H}^q\biggr)^{p/q} \ell( R)^{\dmnMinusOne+p-p\theta}
.\end{equation}
We want to bound $\1_{T(Q)} \partial^\gamma u$. 
Because $q\geq 1$, we have that the triangle inequality in $L^q(W(R))$ is valid, and so if $r={\dmnMinusOne+p-p\theta}$, then
\begin{align*}
\sum_{R\in\mathcal{G}} \biggl(\fint_{W(R)} \abs{\partial^\gamma u}^q\biggr)^{p/q} \ell( R)^r
&\le
	\sum_{R\in\mathcal{G}} 
	\biggl(\biggl(\fint_{W(R)} \abs{\partial^\gamma u-u_{\gamma,R}}^q\biggr)^{1/q}+\abs{u_{\gamma,R}}\biggr)^p \ell( R)^r
.\end{align*}
By the Poincar\'e inequality, if $\ell(R)\leq \ell(Q)$ then
\begin{equation*}\fint_{W(R)} \abs{\partial^\gamma u-u_{\gamma,R}}^q
\leq C\ell(R)^q \fint_{W(R)} \abs{\nabla\partial^\gamma u}^q
\leq C\ell(Q)^q \fint_{W(R)} \abs{\nabla^{k+1}u}^q
\end{equation*}
and so
\begin{align*}
\sum_{R\in\mathcal{G}} \biggl(\fint_{W(R)} \abs{\partial^\gamma u}^q\biggr)^{p/q} \ell( R)^r
&\le
	\sum_{R\in\mathcal{G}} 
	\biggl(\ell(Q)\biggl(\fint_{W(R)} \abs{\nabla^{k+1}u}^q\biggr)^{1/q}+\abs{u_{\gamma,R}}\biggr)^p \ell( R)^r
.\end{align*}

If $p\geq 1$, then we may apply the triangle inequality in a sequence space to see that
\begin{multline*}
\biggl(\sum_{R\in\mathcal{G}} \biggl(\fint_{W(R)} \abs{\partial^\gamma u}^q\biggr)^{p/q} \ell( R)^r\biggr)^{1/p}
\\\le
	\ell(Q)\biggl(\sum_{ R\in\mathcal{G}} \biggl(\fint_{W(R)} \abs{\nabla^{k+1}u}^q\biggr)^{p/q} 
	\ell( R)^r\biggr)^{1/p}
	+\biggl(\sum_{ R\in\mathcal{G}} \abs{u_{\gamma,R}}^p
	\ell( R)^r\biggr)^{1/p}
.\end{multline*}
If $0<p<1$, then the triangle inequality is not valid; however, by Minkowski's inequality for sums, we have that $(a+b)^p\leq a^p+b^p$ for any positive numbers $a$ and~$b$, and so
\begin{multline*}
\sum_{R\in\mathcal{G}} \biggl(\fint_{W(R)} \abs{\partial^\gamma u}^q\biggr)^{p/q} \ell( R)^r
\\\le
	\ell(Q)^p\sum_{ R\in\mathcal{G}} \smash{\biggl(\fint_{W(R)} \abs{\nabla^{k+1}u}^q\biggr)^{p/q} }
	\ell( R)^r
	+\sum_{ R\in\mathcal{G}} \abs{u_{\gamma,R}}^p
	\ell( R)^r
.\end{multline*}

Applying the equivalence of norms~\eqref{eqn:dyadic:norm:local}, we have that if $p\geq 1$ then
\begin{align*}
\doublebar{\1_{T(Q)}\partial^\gamma u}_{L_{av}^{p,\theta,q}(\Omega)}
&\le	
	C\ell(Q) \doublebar{\1_{T(Q)}\nabla^{k+1} u}_{L_{av}^{p,\theta,q}(\Omega)}
	+
	C\biggl(\sum_{ R\in\mathcal{G}} \abs{u_{\gamma,R}}^p \ell(R)^r\biggr)^{1/p}
\end{align*}
and if $p\leq 1$ then
\begin{align*}
\doublebar{\1_{T(Q)}\partial^\gamma u}_{L_{av}^{p,\theta,q}(\Omega)}^p
\le	
	C\ell(Q)^p \doublebar{\1_{T(Q)}\nabla^{k+1} u}_{L_{av}^{p,\theta,q}(\Omega)}^p
	+
	C\sum_{ R\in\mathcal{G}} \abs{u_{\gamma,R}}^p \ell(R)^r
.\end{align*}

We are working by induction and so may assume  $\ell(Q)\doublebar{\1_{T(Q)}\nabla^{k+1} u}_{L_{av}^{p,\theta,q}(\Omega)}\leq C\ell(Q)^{m-k} \doublebar{\1_{T(Q)}\nabla^m u}_{L_{av}^{p,\theta,q}(\Omega)}$.
We  consider the second term.
If $R\in \mathcal{G}_j$ and $0\leq i\leq j$, let $P_i(R)$ be the unique cube in $\mathcal{G}_i$ with $R\subseteq P_i(R)$. Then
\begin{equation*}u_{\gamma,R}=
u_{\gamma,R}-u_{\gamma,Q} = 
\sum_{i=1}^j u_{\gamma,{P_i(R)}} - u_{\gamma,{P_{i-1}(R)}}.\end{equation*}
If $p\leq 1$ then 
\begin{equation*}\abs{u_{\gamma,R}}^p \leq  \sum_{i=1}^j \abs{u_{\gamma,{P_i(R)}} - u_{\gamma,{P_{i-1}(R)}}}^p\end{equation*}
while if $p\geq 1$, then by H\"older's inequality in sequence spaces,
\begin{equation*}\abs{u_{\gamma,R}}^p \leq  j^{p-1}\sum_{i=1}^j \abs{u_{\gamma,{P_i(R)}} - u_{\gamma,{P_{i-1}(R)}}}^p.\end{equation*}
Therefore,
\begin{equation*}
\sum_{R\in\mathcal{G}} \abs{u_{\gamma,R}}^p \ell( R)^r
\le	
	C
	\sum_{j=1}^\infty\sum_{i=1}^j
	\sum_{R\in\mathcal{G}_j}
	\abs{u_{\gamma,{P_i(R)}} - u_{\gamma,{P_{i-1}(R)}}}^p
	j^{\max(p-1,0)} \ell( R)^r
.\end{equation*}
If $R\in\mathcal{G}_j$, then $\ell(R)=2^{-j}\ell(Q)$, and so
\begin{equation*}
\sum_{R\in\mathcal{G}} \abs{u_{\gamma,R}}^p \ell( R)^r
\le	
	C\ell( Q)^r
	\sum_{j=1}^\infty\sum_{i=1}^j
	\sum_{R\in\mathcal{G}_j}
	\abs{u_{\gamma,{P_i(R)}} - u_{\gamma,{P_{i-1}(R)}}}^p
	j^{\max(p-1,0)} 2^{-jr}
.\end{equation*}
Notice that if $R\in \mathcal{G}_j$, then $P_i(R)\in\mathcal{G}_i$. We now wish to sum over $S=P_i(R)\in \mathcal{G}_i$ rather than over $R\in \mathcal{G}_j$. Each such $S$ satisfies $S=P_i(R)$ for $2^{\pdmnMinusOne(j-i)}$ cubes $R\in\mathcal{G}_j$; thus, 
\begin{equation*}
\sum_{R\in\mathcal{G}} \abs{u_{\gamma,R}}^p \ell( R)^r
\le	
	C\ell( Q)^r
	\sum_{j=1}^\infty\sum_{i=1}^j
	\sum_{S\in\mathcal{G}_i}
	\abs{u_{\gamma,S} - u_{\gamma,{P(S)}}}^p
	2^{\pdmnMinusOne(j-i)}
	j^{\max(p-1,0)} 2^{-jr}
\end{equation*}
where $P(S)$ is the dyadic parent of~$S$. Recalling that $r={\dmnMinusOne+p-p\theta}$, we see that
\begin{equation*}
\sum_{R\in\mathcal{G}} \abs{u_{\gamma,R}}^p \ell( R)^r
\le	
	C\ell( Q)^r
	\sum_{j=1}^\infty\sum_{i=1}^j
	\sum_{S\in\mathcal{G}_i}
	\abs{u_{\gamma,S} - u_{\gamma,{P(S)}}}^p
	2^{-i\pdmnMinusOne}
	\frac{j^{\max(p-1,0)}}{2^{j(p-p\theta)}}
.\end{equation*}
Interchanging the order of summation, we see that
\begin{equation*}
\sum_{R\in\mathcal{G}} \abs{u_{\gamma,R}}^p \ell( R)^r
\le	
	C\ell(Q)^r 
	\sum_{i=1}^\infty 
	2^{-i\pdmnMinusOne}
	\sum_{S\in\mathcal{G}_i}
	\abs{u_{\gamma,S} - u_{\gamma,{P(S)}}}^p
	\sum_{j=i}^\infty
	\frac{j^{\max(p-1,0)}}{2^{j(p-p\theta)}}
.\end{equation*}
Let $\varepsilon=(p-p\theta)/2$, so $0<\varepsilon<p/2$.
There is some constant $C=C(p,\theta)$ such that $j^{\max(0,p-1)} < C2^{j\varepsilon}$ for all integers~$j$, and so
\begin{align*}
{\sum_{R\in\mathcal{G}}}\vphantom{\sum} \abs{u_{\gamma,R}}^p \ell( R)^r
&\le	
	C\ell(Q)^r 
	\sum_{i=1}^\infty 
	\sum_{S\in\mathcal{G}_i}
	\abs{u_{\gamma,S} - u_{\gamma,{P(S)}}}^p
	2^{-i(\dmnMinusOne+\varepsilon)}
\\&=
	C
	\sum_{i=1}^\infty 2^{i\varepsilon }
	\sum_{S\in\mathcal{G}_i}
	\abs{u_{\gamma,S} - u_{\gamma,{P(S)}}}^p	
	\ell(S)^r 
.\end{align*}

Again by the Poincar\'e inequality,
\begin{equation*}\abs{u_{\gamma,S} - u_{\gamma,{P(S)}}}
\leq C\ell(S)\fint_{W(S)\cup W(P(S))} \abs{\nabla^{k+1} u}
\end{equation*}
and so 
\begin{align*}
\sum_{R\in\mathcal{G}} \abs{u_{\gamma,R}}^p \ell( R)^r
&\le	
	C
	\sum_{i=1}^\infty 2^{i\varepsilon }
	\sum_{S\in\mathcal{G}_i}
	\ell(S)^{p+r}\biggl(\fint_{W(S)} \abs{\nabla^{k+1} u}\biggr)^p	
.\end{align*}
But $2^{i\varepsilon}\ell(S)^p<2^{ip}\ell(S)^p=\ell(Q)^p$, and so
\begin{equation*}\doublebar{\1_{T(Q)}\nabla^k(u-u_Q)}_{L_{av}^{p,\theta,q}(\Omega)} \leq C\ell(Q) \doublebar{\1_{T(Q)}\nabla^{k+1}(u-u_Q)}_{L_{av}^{p,\theta,q}(\Omega)}.\end{equation*}
By induction, the proof of the bound~\eqref{eqn:poincare} is complete.
(In the case $p=\infty$, the above argument must be modified slightly, by using suprema over the cubes $R\in\mathcal{G}$ rather than sums.)

Now suppose that $\Tr_k^\Omega u=0$ for all $0\leq k\leq m-1$. 
Observe that 
\begin{equation*}\abs[bigg]{\fint_{W(Q)} \nabla^k u} 
\leq 
\abs[bigg]{\fint_{W(Q)} \nabla^k u-{\textstyle{\fint_{T(Q)} \nabla^k u}}} + \abs[bigg]{\fint_{T(Q)} \nabla^k u} 
. \end{equation*}
If $\Trace \nabla^k u=0$ on $\partial\Omega\cap\partial T(Q)$, then we may use some form of the standard Poincar\'e inequality to control each of the terms on the right-hand side; thus,
\begin{equation*}\abs[bigg]{\fint_{W(Q)} \nabla^k u} 
\leq 
C\ell(Q) \fint_{T(Q)} \abs{\nabla^{k+1} u}.\end{equation*}
Applying the Poincar\'e inequality iteratively in $T(Q)$, if $\Trace \nabla^{j} u=0$ for all $k\leq j\leq m-1$, then 
\begin{equation*}\abs[bigg]{\fint_{W(Q)} \nabla^k u} 
\leq 
C\ell(Q)^{m-k} \fint_{T(Q)} \abs{\nabla^m u}.\end{equation*}

Now, recall that $u_Q$ is the polynomial that satisfies $\fint_{W(Q)} \nabla^k u_Q=\fint_{W(Q)} \nabla^k u$ for all $0\leq k\leq m-1$. We may write $u_Q$ as a polynomial in $(x-x_Q)$ for some fixed $x_Q\in W(Q)$. A straightforward induction argument allows us to control the coefficients of $u_Q$ by the averages of $\nabla^k u$, and thereby to show that
\begin{equation*}\sup_{T(Q)} \abs{\nabla^k u_Q} \leq C \sum_{j=0}^{m-1-k} \ell(Q)^j \fint_{W(Q)} \abs{\nabla^{j+k}u}.\end{equation*}
Thus, 
\begin{equation*}\sup_{T(Q)} \abs{\nabla^k u_Q} \leq C\ell(Q)^{m-k} \fint_{T(Q)} \abs{\nabla^m u}\end{equation*}
and by Lemma~\ref{lem:L:L1},
\begin{equation*}\sup_{T(Q)} \abs{\nabla^k u_Q} \leq C \ell(Q)^{m-k-1+\theta-\pdmnMinusOne/p}\doublebar{\1_{T(Q)}\nabla^m u}_{L^{p,\theta,q}_{av}(\Omega)}.\end{equation*}
Because $p-p\theta>0$, we may easily show that 
\begin{equation*}\doublebar{\1_{T(R)} \nabla^k u_Q}_{L_{av}^{p,\theta,q}(\Omega)}\leq C\ell(Q)^{\pdmnMinusOne/p+1-\theta}  \doublebar{\nabla^k u_Q}_{L^\infty(T(R))}\end{equation*}
and so 
\begin{equation*}\doublebar{\1_{T(Q)} \nabla^k u_Q}_{L^{p,\theta,q}_{av}(\Omega)} \leq
 C \ell(Q)^{m-k}\doublebar{\1_{T(Q)}\nabla^m u}_{L^{p,\theta,q}_{av}(\Omega)}
.\end{equation*}
Combining this estimate with the bound~\eqref{eqn:poincare}, we see that $u=(u-u_Q)+u_Q$ must satisfy the bound~\eqref{eqn:poincare:boundary}, as desired.
\end{proof}

We now use this result to establish density of smooth, compactly supported functions in weighted, averaged Sobolev spaces in Lipschitz domains.


\begin{thm}
\label{thm:smooth:dense}
Suppose that $0<\theta<1$, that $1\leq q<\infty$, and that $\Omega$ is a Lipschitz domain.

If $0<p<\infty$, then $\bigl\{\Phi\big\vert_\Omega:\Phi\in C^\infty_0(\R^\dmn)\bigr\}$ is dense in $\dot W_{m,av}^{p,\theta,q}(\Omega)$.

If $p=\infty$ and $u\in\dot W_{m,av}^{\infty,\theta,q}(\Omega)$, then there is some sequence of smooth, compactly supported functions $\varphi_n$, such that $\langle \arr G,\nabla^m\varphi_n\rangle_\Omega\to \langle \arr G,\nabla^m u\rangle_\Omega$ for all $\arr G\in L_{av}^{1,1-\theta,q'}(\Omega)$. 

Furthermore, suppose that
$u\in \dot W_{m,av}^{p,\theta,q}(\Omega)$ with $\Tr_k^\Omega u=0$ for any $0\leq k\leq m-1$, and that $p>\pmin$. If $\Omega$ is bounded or a Lipschitz graph domain, then there is a sequence of functions $\varphi_n$, smooth and compactly supported in $\Omega$, such that $\varphi_n\to u$ as $\dot W_{m,av}^{p,\theta,q}(\Omega)$-functions (if $p<\infty$) or weakly (if $p=\infty$). If $\R^\dmn\setminus\Omega$ is bounded, then there is a sequence of compactly supported functions $\varphi_n\to u$ such that $\nabla^m\varphi_n=0$ in a neighborhood of $\R^\dmn\setminus\Omega$.
\end{thm}

\begin{proof}
Let $u\in \dot W_{m,av}^{p,\theta,q}(\Omega)$ for some $0<p\leq \infty$, $1\leq q<\infty$ and $0<\theta<1$. We will produce smooth, compactly supported functions that approximate~$u$. The proof will require several steps.

\textbf{Step 1.} 
First, we show that $u$ may be approximated by functions defined in~$\Omega$ that are nonzero only inside some bounded set.

If $\Omega$ is bounded then $u$ itself is such a function, and so there is nothing to prove.

Suppose that $\partial\Omega$ is compact and $\Omega$ is unbounded. Let $\varphi_R=1$ in $B(0,R)$ and $\varphi_R=0$ outside $B(0,2R)$, with $\abs{\nabla^k\varphi_R}\leq CR^{-k}$ for all $0\leq k\leq m$. We consider only $R$ large enough that $\R^\dmn\setminus\overline\Omega\subset B(0,R/2)$. Let $A$ be the annulus ${B(0,2R)\setminus B(0,R)}$, and let $u_R$ be the polynomial of degree $m-1$ so that $\int_A\nabla^k (u-u_R)=0$ for all $0\leq k\leq m-1$. Then $(u-u_R)\varphi_R$ is zero outside $B(0,2R)$. By the Poincar\'e inequality in~$A$ and the norm \eqref{eqn:norm:L}, $(u-u_R)\varphi_R$ lies in $\dot W_{m,av}^{p,\theta,q}(\Omega)$. Furthermore, $(u-u_R)\varphi_R\to u$ in $\dot W_{m,av}^{p,\theta,q}(\Omega)$ as $R\to\infty$ if $p<\infty$; if $p=\infty$ then $\langle \arr G, \nabla^m((u-u_R)\varphi_R)\rangle_\Omega\to \langle \arr G, \nabla^m u\rangle_\Omega$ whenever $\arr G \in L_{av}^{1,1-\theta,q'}(\Omega)$. Notice however that the lower order derivatives of $(u-u_R)\varphi_R$ need not approach the derivatives of~$u$; in particular, if $\Tr_k^\Omega u=0$, then $\Tr_k^\Omega ((u-u_R)\varphi_R)=\Tr_k^\Omega u_R$, not zero.

\begin{figure}
\begin{center}
\begin{tikzpicture}[xscale=1.2, yscale = 0.6]

\def\RI{{(0,0)},{(0.5,0)},{(1,0)}}
\def\RII{{(3,0)},{(3.5,0.3)},{(4,0)}}

\foreach \point [count=\n] in \RI {
        \node[coordinate] (dI-\n) at \point {};
        }
\pgfmathtruncatemacro{\sizeI}{\n}

\foreach \point [count=\n] in \RII {
        \node[coordinate] (dII-\n) at \point {};
        }
\pgfmathtruncatemacro{\sizeII}{\n}

%
%
%
%


\draw (-1,0)  
	\foreach \i in {1,...,\sizeI} {-- (dI-\i)}
	\foreach \i in {1,...,\sizeII} {-- (dII-\i)} 
	-- (5,0) node [right] {$\partial\Omega$};

\draw [fill=white!80!black] 
	(dI-1) \foreach \i in {2,...,\sizeI} {-- (dI-\i)} --  +(0,2)--(0,2)--cycle; 
\node at (1/2,1) {$T(R_1)$};

\draw [fill=white!80!black] (3,0) rectangle (4,2); \node at (3+1/2,1) {$T(R_2)$};

\draw [fill=white!80!black] (0,2) rectangle (1,4); \node at (1/2,3) {$\widetilde W(R_1)$};
\draw [fill=white!80!black] (3,2) rectangle (4,4); \node at (3+1/2,3) {$\widetilde W(R_2)$};

\draw [fill=white!80!black] (0,4) rectangle (4,8); \node at (2,6) {$W(Q)$};

\begin{scope}[shift={(0,-1/2)}]
\draw [ultra thick] (0,0) -- node [below] {$R_1$} (1,0);
\draw [ultra thick] (3,0) -- node [below] {$R_2$} (4,0);
\end{scope}

\begin{scope}[shift={(0,-3/2)}]
\draw [ultra thick] (0,0) -- node [below] {$Q$} (4,0);
\end{scope}

\end{tikzpicture}
\caption{The region $A(Q)$ as a union of the regions $W(Q)$, $T(R)$ and $\widetilde W(R)$. }
\label{fig:smooth:dense}
\end{center}
\end{figure}
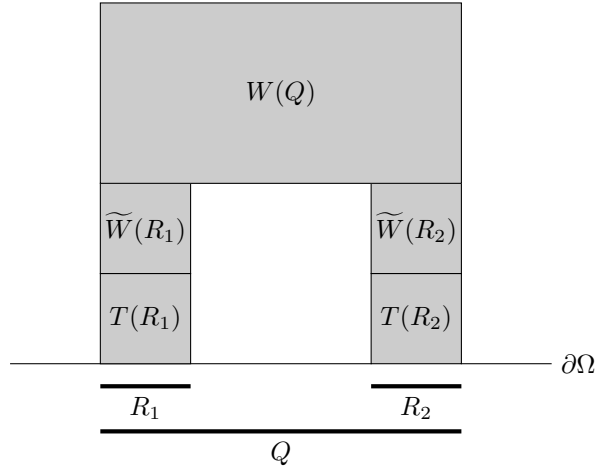

If $\Omega$ is a Lipschitz graph domain, let $Q\subset\R^\dmnMinusOne$ be a cube and adopt the notation of Lemma~\ref{lem:Poincare:L}. In particular, recall the regions $T(Q)$ and $W(Q)$ and the polynomial~$u_Q$.  Let $\varphi_Q$ be supported in $T(Q)$ and identically equal to $1$ in $T((1/2)Q)$, where $(1/2)Q$ is the cube (in $\R^\dmnMinusOne$) concentric to $Q$ with half the side-length. Let $A(Q)={T(Q)\setminus T((1/2)Q)}$. Notice that 
\begin{equation*}A(Q)=W(Q)\cup \bigcup_{R} T(R)\cup \widetilde W(R)\end{equation*}
where the union is over the $4^{\dmnMinusOne}-2^\dmnMinusOne$ dyadic subcubes $R\subset Q\setminus (1/2)Q$ with $\ell(R)=\ell(Q)/4$, and where $\widetilde W(R)$ is a region congruent to $T(R)$ and translated upwards. See Figure~\ref{fig:smooth:dense}.

We now bound the lower-order derivatives of $u-u_Q$ in $A(Q)$; this will allow us to control $\nabla^m (\varphi_Q (u-u_Q))-\nabla^m u$. We consider the regions $W(Q)$, $T(R)$ and $\widetilde W(R)$ separately
By the Poincar\'e inequality in $W(Q)$,
\begin{equation*}\doublebar{\nabla^k (u-u_Q)}_{L^q(W(Q))}\leq C \ell(Q)\doublebar{\nabla^{k+1} (u-u_Q)}_{L^q(W(Q))}\end{equation*}
for any integer $k$ such that $\nabla^{k+1}u\in {L^q(W(Q))}$,
and so by induction, if $0\leq k\leq m$ then
\begin{equation*}\doublebar{\nabla^k (u-u_Q)}_{L^q(W(Q))}\leq C \ell(Q)^{m-k}\doublebar{\nabla^m u}_{L^q(W(Q))}.\end{equation*}

Recall that $W(Q)\subset T(Q)$, and so by formula~\eqref{eqn:dyadic:norm:local} we may express
\begin{equation*}\doublebar{\1_{W(Q)}\nabla^k(u-u_Q)}_{L^{p,\theta,q}_{av}(\Omega)}\end{equation*}
in terms of integrals over $W(R)$ for cubes $R\in \mathcal{G}$. But if $R\in\mathcal{G}$ and $R\neq Q$, then $W(Q)$ and $W(R)$ are disjoint, and so the only nonzero term on the right-hand side of formula~\eqref{eqn:dyadic:norm:local} is the one involving an integral over~$W(Q)$. Thus,
\begin{align*}
\doublebar{\1_{W(Q)}\nabla^k(u-u_Q)}_{L^{p,\theta,q}_{av}(\Omega)}
&\approx \doublebar{\nabla^k (u-u_Q)}_{L^q(W(Q))}\ell(Q)^{-\pdmn/q+\dmnMinusOne+p-p\theta}
.\end{align*}
By the previous inequality
\begin{align*}
\doublebar{\1_{W(Q)}\nabla^k(u-u_Q)}_{L^{p,\theta,q}_{av}(\Omega)}
&\leq 
C\ell(Q)^{m-k}\doublebar{\nabla^m u}_{L^q(W(Q))}
\ell(Q)^{-\pdmn/q+\dmnMinusOne+p-p\theta}
\end{align*}
and a final application of formula~\eqref{eqn:dyadic:norm:local}  yields that
\begin{align*}
\doublebar{\1_{W(Q)}\nabla^k(u-u_Q)}_{L^{p,\theta,q}_{av}(\Omega)}
&\leq 
C\ell(Q)^{m-k}\doublebar{\1_{W(Q)}\nabla^m u}_{L^{p,\theta,q}_{av}(\Omega)}
\\&\leq 
C\ell(Q)^{m-k}\doublebar{\1_{A(Q)}\nabla^m u}_{L^{p,\theta,q}_{av}(\Omega)}
.\end{align*}

Let $R$ be one of the dyadic subcubes mentioned above, and let  $V=W(Q)\cup \widetilde W(R)$. Let $0\leq k\leq m-1$ and let $w=u-u_Q$. Then by elementary arguments and the Poincar\'e inequality in~$V$,
\begin{align*}
\abs[bigg]{\fint_{V} \nabla^k w}
&=
\abs[bigg]{\fint_{W(Q)} \nabla^k w-\fint_{V} \nabla^k w}
=
\abs[bigg]{\fint_{W(Q)} \bigl(\nabla^k w-{\textstyle\fint_V}\nabla^k w\bigr)} 
\\&\leq
\fint_{W(Q)} \abs{\nabla^k w-{\textstyle\fint_V}\nabla^k w}
\leq
\frac{\abs{V}}{\abs{W(Q)}}\fint_{V} \abs{\nabla^k w-{\textstyle\fint_V}\nabla^k w}
\\&\leq
C\fint_{V} \abs{\nabla^k w-{\textstyle\fint_V}\nabla^k w}
\leq
C\ell(Q)\fint_{V} \abs{\nabla^{k+1} w}
.\end{align*}
Now, $\doublebar{\nabla^k w}_{L^q(\widetilde W(R))}
\leq
\doublebar{\nabla^k w}_{L^q(V)}
$, and by the Poincar\'e inequality and H\"older's inequality,
\begin{align*}
\doublebar{\nabla^k w}_{L^q(V)}
&\leq
\doublebar{\nabla^k w-{\textstyle\fint_V}\nabla^k w}_{L^q(V)}
+\abs{V}^{1/q}\abs{\textstyle\fint_V\nabla^k w}
\\&\leq
C\ell(Q)\doublebar{\nabla^{k+1} w}_{L^q(V)}
+\abs{V}^{1/q}C\ell(Q)\fint_{V} \abs{\nabla^{k+1} w}
\\&\leq
C\ell(Q)\doublebar{\nabla^{k+1} w}_{L^q(V)}
.\end{align*}
By induction, and recalling the definitions of~$w$ and~$V$, if $0\leq k\leq m-1$ then
\begin{align*}
\doublebar{\nabla^k (u-u_Q)}_{L^q(\widetilde W(R)\cup W(Q))}
&\leq
C\ell(Q)^{m-k}\doublebar{\nabla^m (u-u_Q)}_{L^q(\widetilde W(R)\cup W(Q))}
.\end{align*}
Let $P(R)$ be the dyadic parent of~$R$. Then $P(R)\in \mathcal{G}$ and $\widetilde W(R)\subset W(P(R))$.
By formula~\eqref{eqn:dyadic:norm:local}, if $r=-\pdmn/q+\dmnMinusOne+p-p\theta$, then
\begin{align*}
\doublebar{\1_{\widetilde W(R)}\nabla^k(u-u_Q)}_{L^{p,\theta,q}_{av}(\Omega)}
&\approx
\ell(P(R))^r\doublebar{\1_{\widetilde W(R)}\nabla^k (u-u_Q)}_{L^q( W(P(R)))}
\\&=2^r\ell(R)^r\doublebar{\nabla^k (u-u_Q)}_{L^q( {\widetilde W(R)} )}
.
\end{align*}
Applying the previous inequality and the fact that $\ell(Q)=2\ell(P(R))=4\ell(R)$, we have that
\begin{align*}
\doublebar{\1_{\widetilde W(R)}\nabla^k(u-u_Q)}_{L^{p,\theta,q}_{av}(\Omega)}
&\leq
C\ell(Q)^{m-k+r}\doublebar{\nabla^m (u-u_Q)}_{L^q(\widetilde W(R)\cup W(Q))}
\\&\leq
C\ell(P(R))^{m-k+r}\doublebar{\1_{A(Q)}\nabla^m (u-u_Q)}_{L^q( W(P(R)))}
\\&\qquad+C\ell(Q)^{m-k+r}\doublebar{\1_{A(Q)}\nabla^m (u-u_Q)}_{L^q( W(Q))}
\end{align*}
and a final application of formula~\eqref{eqn:dyadic:norm:local} yields that
\begin{align*}
\doublebar{\1_{\widetilde W(R)}\nabla^k(u-u_Q)}_{L^{p,\theta,q}_{av}(\Omega)}
&\leq
C\ell(Q)^{m-k}\doublebar{\1_{A(Q)}\nabla^m u}_{L^{p,\theta,q}_{av}(\Omega)}
.\end{align*}

Finally, by Lemma~\ref{lem:Poincare:L},
\begin{equation*}\doublebar{\1_{T(R)}\nabla^k(u-u_R)}_{L^{p,\theta,q}_{av}(\Omega)}
\leq C\ell(Q)^{m-k}\doublebar{\1_{T(R)}\nabla^m u}_{L^{p,\theta,q}_{av}(\Omega)}.\end{equation*}
We thus must bound $u_Q-u_R$.  Let $\widetilde V=W(Q)\cup \widetilde W(R)\cup W(R)$. Arguing as before, we have that
\begin{equation*}
\doublebar{\nabla^k w}_{L^q(\widetilde V)} 
=
\doublebar{\nabla^k (u-u_Q)}_{L^q(\widetilde V)} \leq C\ell(Q)^{m-k} \doublebar{\nabla^m u}_{L^q(\widetilde V)} 
\end{equation*}
and similarly
\begin{equation*}\doublebar{\nabla^k (u-u_R)}_{L^q(\widetilde V)} \leq C\ell(Q)^{m-k} \doublebar{\nabla^m u}_{L^q(\widetilde V)} 
.\end{equation*}
By definition of~$\widetilde V$, and letting $P(R)$ be the dyadic parent of $R$ as before, we have that
\begin{equation*}\doublebar{\nabla^m u}_{L^q(\widetilde V)} 
\leq \doublebar{\nabla^m u}_{L^q(W(Q))}+\doublebar{\1_{A(Q)}\nabla^m u}_{L^q(W(P(R)))}+\doublebar{\nabla^m u}_{L^q(W(R))}
.\end{equation*}
As usual, by formula~\eqref{eqn:dyadic:norm:local} and because $\ell(Q)=2\ell(P(R))=4\ell(R)$, we have that 
\begin{multline*}\doublebar{\nabla^m u}_{L^q(W(Q))}+\doublebar{\1_{A(Q)}\nabla^m u}_{L^q(W(P(R)))}+\doublebar{\nabla^m u}_{L^q(W(R))}
\\\leq
C\ell(Q)^{\pdmn/q-\pdmnMinusOne-p+p\theta}\doublebar{\1_{A(Q)}\nabla^m u}_{L^{p,\theta,q}_{av}(\Omega)}.\end{multline*}
Thus, if $0\leq k\leq m-1$, then
\begin{align*}\doublebar{\nabla^k (u_R-u_Q)}_{L^q(\widetilde V)} 
&\leq 
C\ell(Q)^{m-k+\pdmn/q-\pdmnMinusOne/p-1+\theta}\doublebar{\1_{A(Q)}\nabla^m u}_{L^{p,\theta,q}_{av}(\Omega)}.\end{align*}
But observe that $u_Q$ and $u_R$ are polynomials of degree at most $m-1$. Thus, as in the proof of formula~\eqref{eqn:poincare:boundary}, we may bound the coefficients of $u_Q-u_R$, and so we have a pointwise inequality
\begin{equation*}\doublebar{\nabla^k (u_R-u_Q)}_{L^\infty(T(Q))} 
\leq 
C\ell(Q)^{m-k} \ell(Q)^{-\pdmnMinusOne/p-1+\theta}\doublebar{\1_{A(Q)}\nabla^m u}_{L^{p,\theta,q}_{av}(\Omega)}.\end{equation*}
Again as in the proof of formula~\eqref{eqn:poincare:boundary}, this yields the bound
\begin{equation*}\doublebar{\1_{T(R)} \nabla^k (u_R-u_Q)}_{L_{av}^{p,\theta,q}(\Omega)}
\leq   C\ell(Q)^{m-k} \doublebar{\1_{A(Q)}\nabla^m u}_{L^{p,\theta,q}_{av}(\Omega)}
.\end{equation*}

Combining these estimates, we see that if $0\leq k<m$, then 
\begin{equation*}\doublebar{\1_{A(Q)} (\nabla^k u-\nabla^k u_Q)}_{L_{av}^{p,\theta,q}(\Omega)}\leq C\ell(Q)^{m-k} \doublebar{\1_{A(Q)} \nabla^m u}_{L_{av}^{p,\theta,q}(\Omega)}.\end{equation*}
It is now straightforward to establish that $\varphi_Q (u-u_Q)\to u$ in $\dot W_{m,av}^{p,\theta,q}(\Omega)$ as $Q$ expands to all of $\R^\dmnMinusOne$.

Notice that if $\Tr_{k}^\Omega u=0$ for all $0\leq k\leq m-1$, then we have that $\varphi_Q u\to u$ as $Q$ expands to all of $\R^\dmnMinusOne$, and so in this case we need not renormalize~$u$.

\textbf{Step 2.} We now show that smooth functions are dense. 

Let $v\in \dot W_{m,av}^{p,\theta,q}(\Omega)$ be an approximant to $u$ as produced in Step~1, i.e., let $v$ be zero outside of a bounded set. Let $v_\varepsilon=v*\eta_\varepsilon$, where $\eta_\varepsilon=\varepsilon^{-\pdmn}\eta(x/\varepsilon)$ and where $\eta$ is smooth, nonnegative, supported in $B(0,1)$, and satisfies $\int\eta=1$. Observe that $v_\varepsilon$ is smooth in $\Omega_\varepsilon$, where
\begin{equation*}\Omega_\varepsilon= \bigl\{x\in\Omega:\dist(x,\partial\Omega)>2\varepsilon\bigr\}.\end{equation*}
Because $\{\eta_\varepsilon\}_{\varepsilon>0}$ is a smooth approximate identity, we have that for any fixed~$\delta$, $\1_{\Omega_\delta}\nabla^m v_\varepsilon\to \1_{\Omega_\delta}\nabla^m v$ as $\varepsilon\to 0^+$ in $L_{av}^{p,\theta,q}(\Omega)$, either weakly or strongly. Furthermore,  if $\varepsilon\ll\delta$, then $\1_{\Omega_\varepsilon\setminus\Omega_\delta}\nabla^m v_\varepsilon$ is controlled by $\1_{\Omega\setminus\Omega_{2\delta}}\nabla^m v$, and this second quantity approaches zero in $L_{av}^{p,\theta,q}(\Omega)$ as $\delta\to 0^+$, weakly or strongly; thus, we have that $\1_{\Omega_\varepsilon}\nabla^m v_\varepsilon\to \nabla^m v$ as $\varepsilon\to 0^+$ in $L_{av}^{p,\theta,q}(\Omega)$.

Now, we must extend $v_\varepsilon$ from $\Omega_\varepsilon$ to all of~$\Omega$. 
For ease of visualization, suppose that $\Omega$ is a Lipschitz graph domain, and let $\mathcal{G}$ be a grid of cubes $Q\subset\R^\dmnMinusOne$ of side-length~$C\varepsilon$.  For each such~$Q$, observe that in $W(Q)$, we have that $\abs{\nabla^m v_\varepsilon}\leq C \fint_{ W'(Q)}
\abs{\nabla^m v}$, where $W'(Q)$ is a slightly enlarged version of~$W(Q)$. We may extend $v_\varepsilon$ to a smooth function in such a way that $\abs{\nabla^m v_\varepsilon} \leq C \fint_{ W'(Q)}\abs{\nabla^m v}$ in all of $T(Q)$. Then
\begin{equation*}\int_{T(Q)} \biggl(\fint_{B(x,\Omega)} \abs{\nabla^m v_\varepsilon}^q\biggr)^{p/q} \dist(x,\partial\Omega)^{p-1-p\theta}\,dx
\leq C \biggl(\fint_{\widetilde W(Q)} \abs{\nabla^m v} \biggr)^p \varepsilon^{\dmn+p-1-p\theta}
\end{equation*}
where $B(x,\Omega)=B(x,\dist(x,\partial\Omega)/2)$, as in the definition of $L_{av}^{p,\theta,q}(\Omega)$.
We may sum to see that
\begin{equation*}\doublebar{\1_{\Omega\setminus\Omega_\varepsilon} \nabla^m v_\varepsilon }_{L_{av}^{p,\theta,q}(\Omega)}\leq \doublebar{\1_{\Psi_\varepsilon}\nabla^m v }_{L_{av}^{p,\theta,q}(\Omega)}\end{equation*}
where $\Psi_\varepsilon$ is a small region near the boundary, which shrinks away as $\varepsilon\to 0^+$.
A similar argument is valid in Lipschitz domains with compact boundary. Thus we may extend $v_\varepsilon$ to a smooth function in such a way that $v_\varepsilon\to 0$ in  $L_{av}^{p,\theta,q}(\Omega)$, either weakly or strongly, as $\varepsilon\to 0$.

\textbf{Step 3.} We now prove the second part of the theorem, that is, the special results in the case where $\Trace \partial^\gamma u=0$ on $\partial\Omega$ for all $\gamma\leq m-1$. 

If $\Omega$ is bounded let $v=u$. If $\Omega$ is a Lipschitz graph domain let $v=u\varphi_Q$ for some large cube~$Q$. In both cases $v$ is compactly supported. If $\R^\dmn\setminus \Omega$ is bounded, let $v=(u-u_R)\varphi_R + u_R$, where $R\gg 0$ and where $u_R$ is the polynomial of degree $m-1$ introduced in Step~1. Notice that $v$ is \emph{not} compactly supported but that $v$ equals a polynomial outside of some large ball.

Let $v_\varepsilon = v*\eta_\varepsilon$ as before. Notice that $\nabla^m (u_R*\eta_\varepsilon) = (\nabla^m u_R)*\eta_\varepsilon =0$, and so if $\R^\dmn\setminus \Omega$ is bounded then $v_\varepsilon$ is equal to a polynomial of degree $m-1$ outside of some ball. 
Let $\varphi_\varepsilon$ be smooth, supported in $\Omega_{K\varepsilon}$ and identically equal to $1$ in $\Omega_{2K\varepsilon}$, with $\abs{\nabla^k\varphi_\varepsilon}\leq C \varepsilon^{-k}$ for all $1\leq k\leq m$, where $K$ is a large constant depending on the Lipschitz character of~$\Omega$.

We wish to show that $v_\varepsilon\,\varphi_\varepsilon \to v$.

Recall that $\1_{\Omega_{\varepsilon}}\nabla^m v_\varepsilon\to \nabla^m v$, and so we need only bound $\1_{\Omega_{\varepsilon}}\nabla^m v_\varepsilon - \nabla^m (v_\varepsilon\,\varphi_\varepsilon)$.
Arguing as above, we may see that $\1_{\Omega_{\varepsilon}}\nabla^m v_\varepsilon - \nabla^m v_\varepsilon\,\varphi_\varepsilon\to 0$ in $L^{p,\theta,q}_{m,av}(\Omega)$ or weakly as $\varepsilon\to 0$, and so we need only bound terms of the form $\nabla^k v_\varepsilon \nabla^{m-k}\varphi_\varepsilon$ for $m-k\geq 1$.

If $\Omega$ is a Lipschitz graph domain then by formula~\eqref{eqn:dyadic:norm:local}
\begin{equation*}\doublebar{\nabla^k v_\varepsilon \nabla^{m-k}\varphi_\varepsilon}_{L_{av}^{p,\theta,q}(\Omega)}^p
\approx
	\sum_{Q\in\mathcal{G}} \biggl(\fint_{W(Q)} \abs{\nabla^k v_\varepsilon \nabla^{m-k}\varphi_\varepsilon}^q\biggr)^{p/q} \ell(Q)^{\dmnMinusOne+p-p\theta}
\end{equation*}
where $\mathcal{G}$ is a grid of dyadic cubes in $\R^\dmnMinusOne$.
But $\nabla^{m-k}\varphi_\varepsilon$ is supported only in $\Omega_{K\varepsilon}\setminus \Omega_{2K\varepsilon}$, so 
\begin{equation*}\doublebar{\nabla^k v_\varepsilon \nabla^{m-k}\varphi_\varepsilon}_{L_{av}^{p,\theta,q}(\Omega)}^p
 \approx
	\sum_{\substack{Q\in\mathcal{G}\\ \!\!\!\!\!\!\! (K/C)\varepsilon\leq\ell(Q)\leq CK\varepsilon \!\!\!\!\!\!\!}} \biggl(\fint_{W(Q)} \abs{\nabla^k v_\varepsilon \nabla^{m-k}\varphi_\varepsilon}^q\biggr)^{p/q} \ell(Q)^{\dmnMinusOne+p-p\theta}
.\end{equation*}
Using our bounds on $\varphi_\varepsilon$, we see that
\begin{equation*}\doublebar{\nabla^k v_\varepsilon \nabla^{m-k}\varphi_\varepsilon}_{L_{av}^{p,\theta,q}(\Omega)}^p
 \leq C
	\sum_{\substack{Q\in\mathcal{G}\\ \!\!\!\!\!\!\!\! (K/C)\varepsilon\leq\ell(Q)\leq CK\varepsilon \!\!\!\!\!\!\!\!}} \biggl(\fint_{W(Q)} \abs{\nabla^k v_\varepsilon }^q\biggr)^{p/q} \ell(Q)^{\dmnMinusOne+p-p\theta-pm+pk}
.\end{equation*}
If $K$ is large enough, then as before we may control $\nabla^k v_\varepsilon$ in $W(Q)$ by $\nabla^k v$ in $W'(Q)$, and because $\Tr_k^\Omega v=0$ for all $0\leq k\leq m-1$, we may control $\nabla^k v$ in $W'(Q)$ using Lemma~\ref{lem:Poincare:L}; thus
\begin{equation*}\doublebar{\nabla^k v_\varepsilon \nabla^{m-k}\varphi_\varepsilon}_{L_{av}^{p,\theta,q}(\Omega)}^p
 \leq C
	\sum_{\substack{Q\in\mathcal{G}\\ \!\!\!\!\!\!\! (K/C)\varepsilon\leq\ell(Q)\leq CK\varepsilon \!\!\!\!\!\!\!}} \doublebar{\1_{T(CQ)} \nabla^m v}_{L_{av}^{p,\theta,q}(\Omega)}^p
.\end{equation*}
If $p<\infty$ then the right-hand side approaches zero as $\varepsilon\to 0$, and if $p=\infty$ it is bounded for all~$\varepsilon$ (after replacing sums with appropriate suprema). Thus, $v_\varepsilon\,\varphi_\varepsilon \to v$ in $\dot W_{m,av}^{p,\theta,q}(\Omega)$, weakly or strongly, as desired. If $\partial\Omega$ is compact, notice that $\1_{\Omega_{\varepsilon}}\nabla^m v_\varepsilon - \nabla^m (v_\varepsilon\,\varphi_\varepsilon)=0$ except for a small region near the boundary; working in Lipschitz cylinders and Lipschitz graph domains, as in Definition~\ref{dfn:domain}, we may show that $v_\varepsilon\varphi_\varepsilon \to v$, as desired.
(If $\R^\dmn\setminus\Omega$ is bounded then $v_\varepsilon \varphi_\varepsilon$ is not compactly supported; however, $v_\varepsilon \varphi_\varepsilon= v_\varepsilon \varphi_\varepsilon - u_R*\eta_\varepsilon$ as $\dot W_{m,av}^{p,\theta,q}$-functions, and $v_\varepsilon \varphi_\varepsilon - u_R*\eta_\varepsilon$ is compactly supported and equal to a polynomial in a neighborhood of $\partial\Omega$, as desired.)
\end{proof}

\section{Extensions: Dirichlet boundary data}
\label{sec:extension:dirichlet}

In this section we will prove the following extension theorem; this will show, in effect, that $\dot W\!A^p_\theta(\partial\Omega)\subseteq \{\Tr_{m-1}^\Omega u: u\in\dot W_{m,av}^{p,\theta,q}(\Omega)\}$. In Section~\ref{sec:trace:dirichlet} we will prove the opposite inclusion, showing that these two spaces are equal.

\begin{thm} 
\label{thm:extension}
Suppose that $0<\theta<1$ and that $\pdmnMinusOne/(\dmnMinusOne+\theta) < p\leq\infty$.
Let $\Omega$ be a Lipschitz domain with connected boundary.

Suppose that $\arr \varphi\in \dot W\!A^p_{\theta}(\partial\Omega)$. Then there is some $\Phi\in\dot W_{m,av}^{p,\theta,\infty}(\Omega)$ such that 
$\arr\varphi = \Tr_{m-1}^\Omega \Phi$ and such that
\begin{equation*}\doublebar{\Phi}_{\dot W_{m,av}^{p,\theta,\infty}(\Omega)}\leq C \doublebar{\arr \varphi}_{\dot W\!A^p_{\theta}(\partial\Omega)}.\end{equation*}
In the case $p=1$ this is true whether we use atoms or the norm~\eqref{eqn:Slobodekij} to characterize~$\dot B^{1,1}_\theta(\partial\Omega)$; that is, if $\arr\varphi$ lies in the set in formula~\eqref{eqn:Whitney:dense} then there is an extension $\Phi$ such that both of the bounds
\begin{align*}\doublebar{\Phi}_{\dot W^{1,\theta,q}_{m,av}(\Omega)}
&\leq C \int_{\partial\Omega} \int_{\partial\Omega} \frac{\abs{ \arr\varphi(x)-\arr\varphi(y)}}{\abs{x-y}^{\dmnMinusOne+\theta}}\,d\sigma(x)\,d\sigma(y)
,\\
\doublebar{\Phi}_{\dot W^{1,\theta,q}_{m,av}(\Omega)}
&\leq C \inf\Bigl\{\sum_j\abs{\lambda_j}:\arr \varphi= \arr c_0+\sum_j \lambda_j\, \arr a_j,\> \arr c_0\text{ constant},\>  \arr a_j\text{  atoms}\Bigr\}\end{align*}
are valid.
\end{thm}

As mentioned in Remark~\ref{rmk:p=1}, the $m=1$ cases of this theorem and of Theorem~\ref{thm:trace} imply that the atomic characterization and the norm~\eqref{eqn:Slobodekij} are equivalent in the case $p=1$.

The remainder of Section~\ref{sec:extension:dirichlet} will be devoted to a proof of this theorem.

Our goal is to show that if $\arr\varphi\in \dot W\!A^p_{\theta}(\partial\Omega)$, then $\arr\varphi=\Tr_{m-1}^\Omega\Phi$ for some $\Phi\in \dot W_{m,av}^{p,\theta,\infty}(\Omega)$. 

Recall the definition~\ref{dfn:whitney} of $\dot W\!A^p_{\theta}(\partial\Omega)$. If $p=\infty$ then $\arr\varphi=\Tr_{m-1}^\Omega\varphi$ for some $\varphi$ with $\nabla^{m-1}\varphi\in C^\theta(\Omega)$, while if $p<\infty$ then $\{\Tr_{m-1}^\Omega\varphi:\nabla^m\varphi\in L^\infty(\R^\dmn),\>\varphi $ compactly supported$\}$ is dense in $\dot W\!A^p_{\theta}(\partial\Omega)$. In either case, we may consider only arrays $\arr\varphi$ that satisfy $\arr\varphi=\Tr_{m-1}^\Omega\varphi$ for some $\varphi$ such that $\nabla^{m-1}\varphi$ is H\"older continuous up to the boundary.

Fix such an extension~$\varphi$. 
We claim that there is a function $\Phi\in {\dot W_{m,av}^{p,\theta,\infty}(\Omega)}$ with $\Tr_{m-1}^\Omega\Phi=\Tr_{m-1}^\Omega\varphi$ that satisfies
\begin{equation*}\doublebar{\Phi}_{\dot W_{m,av}^{p,\theta,\infty}(\Omega)}\leq C \doublebar{\Tr_{m-1} \varphi}_{\dot B^{p,p}_\theta(\partial\Omega)}=C \doublebar{\arr\varphi}_{\dot W\!A^p_{\theta}(\partial\Omega)}.\end{equation*} 
This suffices to prove the theorem.

We will follow closely the proof of \cite[Proposition 7.3]{MazMS10}. The main differences in our case are, first, that \cite[Proposition 7.3]{MazMS10} does not discuss the case $p\leq 1$, and second, that we have chosen to work with homogeneous spaces.

Let $\varphi_\gamma(y)=\partial^\gamma\varphi(y)$ for any multiindex $\gamma$ with $\abs{\gamma}\leq m-1$. 
Define
\begin{equation*}
P_\gamma(x,y)=\sum_{\zeta\geq\gamma,\>\abs{\zeta}\leq m-1} \frac{1}{(\zeta-\gamma)!} \varphi_{\zeta}(y)\,(x-y)^{\zeta-\gamma}
\end{equation*}
and let $P(x,y)=P_{\vec 0}(x,y)$. Notice that $p(x)=P_\gamma(x,y)$ is a Taylor expansion of $\partial^\gamma\varphi(x)$ around the point $x=y$; in particular, $p(x)$ is a polynomial in~$x$, and if $\abs{\gamma}=m-1$ then $P_\gamma(x,y)=\varphi_\gamma(y)$. Furthermore, $\partial_x^\delta P_\gamma(x,y)=P_{\gamma+\delta}(x,y)$.

Define
\begin{equation*}\mathcal{E}\arr \varphi(x)=\int_{\partial\Omega} K(x,y)\,P(x,y)\,d\sigma(y)
\end{equation*}
for all $x\in\Omega$, where $K(x,y):\Omega\times\partial\Omega\mapsto\R$ is a kernel that satisfies the requirements
\begin{align*}
\int_{\partial\Omega} K(x,y)\,d\sigma(y)&=1 &&\text{for all $x\in\Omega$}
,\\
\abs{\partial^\gamma_x K(x,y)}&\leq \frac{C_\gamma}{ \dist(x,\partial\Omega)^{\dmnMinusOne+\abs{\gamma}}}
 &&\text{for all $x\in\Omega$, all $y\in\partial\Omega$, and all $\gamma\geq 0$}
 ,\\
 K(x,y)&=0 && \text{whenever $\abs{x-y}\geq 2\dist(x,\partial\Omega)$}.
\end{align*}

We wish, first, to bound $\nabla^m \mathcal{E}\arr\varphi(x)$, and, second, to show that $\Tr_{m-1}^\Omega\mathcal{E}\arr\varphi=\arr\varphi$.

Let $x\in\Omega$. We assume first that $\dist(x,\partial\Omega)\leq r_\Omega/C$, where $r_\Omega$ is the natural length scale of Definition~\ref{dfn:domain}. (If $\Omega$ is a Lipschitz graph domain then this is true for all $x\in\Omega$.)
Then for any multiindex~$\alpha$,
\begin{align*}
\partial^\alpha\mathcal{E}\arr \varphi(x)
&=
	\sum_{\delta \leq \alpha} \frac{\alpha!}{\delta!(\alpha-\delta)!} \int_{\partial\Omega} \partial_x^{\alpha-\delta} K(x,y)\,\partial_x^{\delta} P(x,y)\,d\sigma(y)
.\end{align*}
Observe that if $\abs{\delta}>m-1$ then $\partial_x^\delta P(x,y)=0$, and so we may disregard terms of higher order. Furthermore, recall that $\int K(x,y)\,d\sigma(y)=1$ is independent of~$x$, and so $\int \partial_x^{\alpha-\delta} K(x,y)\,d\sigma(y)=0$ whenever $\delta<\alpha$. Applying these facts, we see that for every $z\in\partial\Omega$,
\begin{align*}
\partial^\alpha\mathcal{E}\arr \varphi(x)
&=
	\sum_{\abs{\delta}\leq m-1,\>\delta < \alpha} \frac{\alpha!}{\delta!(\alpha-\delta)!} \int_{\partial\Omega} \partial_x^{\alpha-\delta} K(x,y)\,(\partial_x^{\delta} P(x,y)-\partial_x^{\delta} P(x,z))\,d\sigma(y)
	\\&\qquad+
	 \int_{\partial\Omega} K(x,y)\,\partial_x^{\alpha} P(x,y)\,d\sigma(y)
.\end{align*}
From \cite[p.~177]{Ste70} we have the formula
\begin{equation*}
\partial_x^{\delta} P(x,y)-\partial_x^{\delta} P(x,z)
=\sum_{\zeta\geq\delta,\>\abs{\zeta}\leq m-1} \frac{1}{(\zeta-\delta)!} (\varphi_{\zeta}(y)-P_{\zeta}(y,z))\, (x-y)^{\zeta-\delta}.
\end{equation*}
This formula may also be verified by observing, first, that it is valid if $\abs{\delta}=m-1$, and, second, that it is valid if $x=y$ for all~$\delta$ and that differentiating both sides with respect to~$x$ yields the same formula with $\abs{\delta}$ increased.

So 
\begin{align*}
\partial^\alpha\mathcal{E}\arr \varphi(x)
&=
	\sum_\delta 	\sum_\zeta \frac{\alpha!}{\delta!(\alpha-\delta)!} 
	\int_{\partial\Omega} \partial_x^{\alpha-\delta} K(x,y)
	\frac{(x-y)^{\zeta-\delta}}{(\zeta-\delta)!} (\varphi_{\zeta}(y)-P_{\zeta}(y,z))\,d\sigma(y)
	\\&\qquad+
	 \int_{\partial\Omega} K(x,y)\,\partial_x^{\alpha} P(x,y)\,d\sigma(y)
\end{align*}
where the sums are over all $\delta$ with $\delta<\alpha$ and $\abs{\delta}\leq m-1$, and over all $\zeta$ with $\zeta\geq\delta$ and $\abs{\zeta}\leq m-1$. Notice that if $\abs{\alpha}\geq m$ then the second term vanishes. 

Let $\Delta(x)=\partial\Omega\cap B(x,2\dist(x,\partial\Omega))$. Recall that by assumption, if $K(x,y)\neq 0$ then $y\in\Delta(x)$. Furthermore, we assumed $\dist(x,\partial\Omega)<r_\Omega/C$, and so we have that $\sigma(\Delta(x))\approx \dist(x,\partial\Omega)^\dmnMinusOne$.

If $z\in\Delta(x)$, then we have the bound
\begin{align*}
\abs{\nabla^{m}\mathcal{E}\arr\varphi(x)}
&\leq
	C\sum_{j=0}^{\!\!\!m-1\!\!\!} 	\sum_{\abs{\zeta}\leq m-1} 
	\int_{\Delta(x)} \abs{\nabla_x^{m-j} K(x,y)}
	\abs{x-y}^{\abs{\zeta}-j} \abs{\varphi_{\zeta}(y)-P_{\zeta}(y,z)}\,d\sigma(y)
\\&\leq
	C
	\sum_{\abs{\zeta}\leq m-1} 
	\dist(x,\partial\Omega)^{1-\pdmn-m+\abs{\zeta}} 
	\int_{\Delta(x)}
	\abs{\varphi_{\zeta}(y)-P_{\zeta}(y,z)}\,d\sigma(y)
.\end{align*}

We may average over all $z\in \Delta(x)$ to see that
\begin{align*}
\abs{\nabla^{m}\mathcal{E}\arr\varphi(x)}
&\leq
	C
	\sum_{\abs{\zeta}\leq m-1} 	 
	\int_{\Delta(x)}
	\int_{\Delta(x)}
	\frac{\abs{\varphi_{\zeta}(y)-P_{\zeta}(y,z)}}
	{\dist(x,\partial\Omega)^{2\pdmnMinusOne+m-\abs{\zeta}}}
	\,d\sigma(y)\,d\sigma(z)
.\end{align*}
If $1\leq q<\infty$, then by H\"older's inequality
\begin{multline}\label{eqn:extension:bound:2}
\abs{\nabla^{m}\mathcal{E}\arr\varphi(x)}^q
\leq
	\sum_{\abs{\zeta}\leq m-1}
	\frac{C(q)}{\dist(x,\partial\Omega)^{\dmnMinusOne+mq-q\abs{\zeta}}}
	\\
	\times	 	 
	\int_{\Delta(x)}
	\int_{\Delta(x)}
	\frac{\abs{\varphi_{\zeta}(y)-P_{\zeta}(y,z)}^q}
	{\dist(x,\partial\Omega)^{\dmnMinusOne}}
	\,d\sigma(y)\,d\sigma(z)
.\end{multline}
We now must bound the quantity $\abs{\varphi_{\zeta}(y)-P_{\zeta}(y,z)}$.

If $\abs{\zeta}=m-1$ then $P_\zeta(y,z)=\varphi_{\zeta}(z)$. If $\abs{\zeta}<m-1$, recall that $p(y)=P_\zeta(y,z)$ is the Taylor polynomial for $\varphi_\zeta$ expanded around the base point $y=z$. We may thus use standard error estimates for Taylor polynomials to bound $\varphi_\zeta(y)-P_\zeta(y,z)$.

Recall that $\dist(x,\partial\Omega)\leq r_\Omega/C$. Then $\Delta(x)\subset\partial V_j$ for some Lipschitz graph domain, as in Definition~\ref{dfn:domain}. Let $V_j=\{(x',t):t>\psi(x')\}$ for some Lipschitz function~$\psi$. Let $z=(z',\psi(z'))$.

Now, let $\widetilde\Delta(z',r)$ be the ball in $\R^\dmnMinusOne$ centered at $z'$ of radius~$r$. Let $\eta$ be a Lipschitz function defined on $\widetilde\Delta(z',r)$ with $\eta(z')=0$, so that we may bound $\eta(y')=\eta(y')-\eta(z')$ by an appropriate integral of~$\nabla\eta$. It is an elementary exercise in multivariable calculus to establish that
\begin{equation*}
\int_{\widetilde\Delta(z',r)} \frac{\abs{\eta(y')}}{\abs{y'-z'}^\dmnMinusOne}\,dy'
\leq
	r\int_{\widetilde\Delta(z',2r)} \frac{\abs{\nabla\eta(y')}}{\abs{y'-z'}^\dmnMinusOne}\,dy'
.\end{equation*}
Let $q\geq 1$ and let $\theta$ be a Lipschitz function defined on $\widetilde\Delta(z',r)$ with $\theta(z')=0$. Applying the previous inequality to the function $\eta(y')=\abs{\theta(y')}^q$ and using H\"older's inequality, we see that
\begin{equation*}
\int_{\widetilde\Delta(z',r)} \frac{\abs{\theta(y')}^q}{\abs{y'-z'}^\dmnMinusOne}\,dy'
\leq
	q^q\,r^q\int_{\widetilde\Delta(z',r)} \frac{\abs{\nabla\theta(y')}^q}{\abs{y'-z'}^\dmnMinusOne}\,dy'
.\end{equation*}
We now choose $\theta(y')=\varphi_\zeta(y',\psi(y'))-P_\zeta((y',\psi(y')),z)$; $\theta$ is then a Lipschitz function, albeit is not smooth. 
We then have that 
\begin{equation*}
	\int_{\Delta(z,r)}
	\frac{\abs{\varphi_{\zeta}(y)-P_{\zeta}(y,z)}^q}
	{\abs{y-z}^{\dmnMinusOne}}
	\,d\sigma(y)
\leq
C(q)\, r^q
	\sum_{\abs{\xi}=\abs{\zeta}+1} 	 
	\int_{\Delta(z,r)}
	\frac{\abs{\varphi_{\xi}(y)-P_{\xi}(y,z)}^q}
	{\abs{y-z}^{\dmnMinusOne}}
	\,d\sigma(y)
\end{equation*}
for any $\abs{\zeta}< m-1$, where ${\Delta(z,r)}=\{(s',\psi(s')):s'\in {\widetilde\Delta(z',r)}\}$. 
We may choose some $r\approx\dist(x,\partial\Omega)$ such that $\Delta(x)\subset\Delta(z,r)$ for all $z\in\Delta(x)$; if $\dist(x,\partial\Omega)$ is small enough then we may also choose $r$ small enough that 
$\Delta(z,r)\subset\partial\Omega\cap\partial V$. 

By induction, we have that 
\begin{multline*}
\abs{\nabla^{m}\mathcal{E}\arr\varphi(x)}^q
\leq
	\sum_{\abs{\zeta}= m-1}
	\frac{C(q)}{\dist(x,\partial\Omega)^{\dmnMinusOne+q}}
	\\
	\times
	\int_{\Delta(x)}
	\int_{\Delta(z,r)}
	\frac{\abs{\varphi_{\zeta}(y)-P_{\zeta}(y,z)}^q}
	{\abs{y-z}^{\dmnMinusOne}}
	\,d\sigma(y)\,d\sigma(z)
.\end{multline*}
But if $\abs{\zeta}=m-1$ then $P_\zeta(y,z)=\varphi_\zeta(z)$, and so 
\begin{equation}
\label{eqn:extension:bound:3}
\abs{\nabla^{m}\mathcal{E}\arr\varphi(x)}^q
\leq
	\frac{C(q)}{\dist(x,\partial\Omega)^{\dmnMinusOne+q}}
	\int_{\Delta(x)}
	\int_{\Delta(z,r)}
	\frac{\abs{\arr\varphi(y)-\arr\varphi(z)}^q}
	{\abs{y-z}^{\dmnMinusOne}}
	\,d\sigma(y)\,d\sigma(z)
\end{equation}
for all $x\in\Omega$ with $\dist(x,\partial\Omega)< r_\Omega/C_0$.

If $\partial\Omega$ is compact then we must consider $x$ with $\dist(x,\partial\Omega)\geq r_\Omega/C_0$. Notice that if $r_\Omega/2C_0<\dist(x,\partial\Omega)<r_\Omega/C_0$, then 
\begin{equation}\label{eqn:extension:bound:4}
\abs{\nabla^{m}\mathcal{E}\arr\varphi(x)}^q
\leq
	\frac{C(q)}{r_\Omega^{\dmnMinusOne+q}}
	\int_{\partial\Omega}
	\int_{\partial\Omega}
	\frac{\abs{\arr\varphi(y)-\arr\varphi(z)}^q}
	{\abs{y-z}^{\dmnMinusOne}}
	\,d\sigma(y)\,d\sigma(z)
.\end{equation}
The right-hand side is independent of~$x$.

Let $\eta$ be a smooth cutoff function such that $\eta(x)=1$ when $\dist(x,\partial\Omega)<r_\Omega/2C_0$ and $\eta(x)=0$ when $\dist(x,\partial\Omega)>r_\Omega/C_0$. Let $\widetilde P_{\arr\varphi}$ be the polynomial of degree $m-1$ that satisfies
\begin{equation*}\int_{r_\Omega/2C_0<\dist(x,\partial\Omega)<r_\Omega/C_0} \bigl(\partial^\gamma \mathcal{E}\arr \varphi(x)-\partial^\gamma \widetilde P_{\arr\varphi}(x)\bigr)\,dx=0\end{equation*}
for all $\abs\gamma\leq m-1$.

Let $\widetilde{\mathcal{E}}\arr\varphi=\eta(\mathcal{E}\arr \varphi-\widetilde P_{\arr\varphi})+\widetilde P_{\arr\varphi}$. 
We claim that $\widetilde{\mathcal{E}}\arr\varphi$ satisfies the bound~\eqref{eqn:extension:bound:3} for all $x\in\Omega$.

There are three cases to consider.
If $\dist(x,\partial\Omega)\leq r_\Omega/2C_0$, then ${\nabla^{m}\widetilde{\mathcal{E}}\arr\varphi(x)} = {\nabla^{m}{\mathcal{E}}\arr\varphi(x)}$. If $\dist(x,\partial\Omega)\geq r_\Omega/2C_0$, then ${\nabla^{m}\widetilde{\mathcal{E}}\arr\varphi(x)}=0$. If $r_\Omega/2C_0< \dist(x,\partial\Omega) <r_\Omega/C_0$, then 
\begin{equation*}\abs{\nabla^{m}\widetilde{\mathcal{E}}\arr\varphi(x)}
\leq C\sum_{j=0}^m \abs{\nabla^{m-j} \eta(x)} \abs{\nabla^{j} ( \mathcal{E}\arr \varphi(x)- \widetilde P_{\arr\varphi}(x))}.
\end{equation*}
Let $\widetilde \Omega = \{x:{r_\Omega/2C_0<\dist(x,\partial\Omega)<r_\Omega/C_0}\}$. If $C_0$ is large enough, then $\widetilde\Omega$ is connected.
If $\abs\gamma=j$, then
\begin{align*}\abs{\partial^\gamma ( \mathcal{E}\arr \varphi(x)- \widetilde P_{\arr\varphi}(x))}
&=
	\abs[bigg]{\partial^\gamma ( \mathcal{E}\arr \varphi(x)- \widetilde P_{\arr\varphi}(x))-\fint_{\widetilde\Omega}\partial^\gamma( \mathcal{E}\arr \varphi- \widetilde P_{\arr\varphi})}
\\&\leq 
	C r_\Omega \doublebar{\nabla\partial^\gamma( \mathcal{E}\arr \varphi- \widetilde P_{\arr\varphi})}_{L^\infty(\widetilde\Omega)}.
\end{align*}
An induction argument yields the bound 
\[\abs{\nabla^{j} ( \mathcal{E}\arr \varphi(x)- \widetilde P_{\arr\varphi}(x))}\leq Cr_\Omega^{m-j}\doublebar{\nabla^m \mathcal{E}\arr \varphi}_{L^\infty(\widetilde\Omega)}.\] Applying the bound~\eqref{eqn:extension:bound:4} and imposing the bound $\abs{\nabla^{m-j}\eta}\leq Cr_\Omega^{j-m}$, we have that
\begin{equation*}\abs{\nabla^{m}\widetilde{\mathcal{E}}\arr\varphi(x)}^q
\leq
	\frac{C(q)}{r_\Omega^{\dmnMinusOne+q}}
	\int_{\partial\Omega}
	\int_{\partial\Omega}
	\frac{\abs{\arr\varphi(y)-\arr\varphi(z)}^q}
	{\abs{y-z}^{\dmnMinusOne}}
	\,d\sigma(y)\,d\sigma(z)
\end{equation*}
for all $x\in\widetilde\Omega$.

Thus, we have that
\begin{equation}
\label{eqn:extension:bound}
\abs{\nabla^{m}\widetilde{\mathcal{E}}\arr\varphi(x)}^q
\leq
	\frac{C(q)}{\dist(x,\partial\Omega)^{\dmnMinusOne+q}}
	\int_{\Delta'(x)}
	\int_{\Delta'(x)}
	\frac{\abs{\arr\varphi(y)-\arr\varphi(z)}^q}
	{\abs{y-z}^{\dmnMinusOne}}
	\,d\sigma(y)\,d\sigma(z)
\end{equation}
for all $x\in\Omega$. Here $\Delta'(x)=\{y\in\partial\Omega:\abs{x-y}<C_1\dist(x,\partial\Omega)\}$ for some $C_1$ sufficiently large; in particular, we require $C_1$ to be large enough that, if $\dist(x,\partial\Omega)>r_\Omega/2C_0$, then $\partial\Omega=\Delta'(x)$.

By letting $\Delta''(x)=\{y\in\partial\Omega:\abs{x-y}<C_2\dist(x,\partial\Omega)\}$ for some $C_2>C_1$ large enough, we may establish the bound
\begin{align*}
\sup_{B(x,\dist(x,\partial\Omega)/2)}
\abs{\nabla^{m}\widetilde{\mathcal{E}}\arr\varphi}^q
&\leq
	\frac{C(q)}{\dist(x,\partial\Omega)^{\dmnMinusOne+q}}
	\\&\qquad\qquad\times
	\int_{\Delta''(x)}
	\int_{\Delta''(x)}
	\frac{\abs{\arr\varphi(y)-\arr\varphi(z)}^q}
	{\abs{y-z}^{\dmnMinusOne}}
	\,d\sigma(y)\,d\sigma(z)
.\end{align*}

If $p=\infty$, take $q=1$. Then ${\arr \varphi}$ lies in the space $\dot C^\theta(\partial\Omega)={\dot B^{\infty,\infty}_\theta(\partial\Omega)}$ of H\"older continuous functions. Thus
\begin{align*}
\abs{\nabla^{m}\widetilde{\mathcal{E}}\arr\varphi(x)}
&\leq
	\frac{C\doublebar{\arr\varphi}_{\dot B^{\infty,\infty}_\theta(\partial\Omega)}}
	{\dist(x,\partial\Omega)^{\pdmnMinusOne+1}}
	\int_{\Delta''(x)}
	\int_{\Delta''(x)}
	\frac{\abs{y-z}^{\theta}}
	{\abs{y-z}^{\dmnMinusOne}}
	\,d\sigma(y)\,d\sigma(z)
\\&\leq
	\frac{C\doublebar{\arr\varphi}_{\dot B^{\infty,\infty}_\theta(\partial\Omega)}}
	{\dist(x,\partial\Omega)^{1-\theta}}
\end{align*}
for all $x\in\Omega$, and so $\doublebar{\nabla^{m}\widetilde{\mathcal{E}}\arr\varphi}_{L_{av}^{\infty,\theta,\infty}(\Omega)}\leq C \doublebar{\arr\varphi}_{\dot B^{\infty,\infty}_\theta(\partial\Omega)}$.

If $1\leq p<\infty$, then we let $q=p$ and see that
\begin{multline*}
\int_\Omega
\doublebar{\nabla^m\widetilde{\mathcal{E}}\arr\varphi}_{L^\infty(B(x,\Omega))}^p
\,\dist(x,\partial\Omega)^{p-1-p\theta}\,dx
\\\begin{aligned}
&\leq
	C(p)\int_\Omega
	\int_{\Delta''(x)}
	\int_{\Delta''(x)}
	\frac{\abs{\arr\varphi(y)-\arr\varphi(z)}^p}
	{\abs{y-z}^{\dmnMinusOne}}
	\,d\sigma(y)\,d\sigma(z)
	\,\dist(x,\partial\Omega)^{-p\theta-\dmn}\,dx
.\end{aligned}\end{multline*}
Interchanging the order of integration we see that
\begin{multline*}
\int_\Omega
\doublebar{\nabla^m\widetilde{\mathcal{E}}\arr\varphi}_{L^\infty(B(x,\Omega))}^p
\,\dist(x,\partial\Omega)^{p-1-p\theta}\,dx
\\\begin{aligned}
&\leq
	C(p)
	\int_{\partial\Omega}
	\int_{\partial\Omega}
	\frac{\abs{\arr\varphi(y)-\arr\varphi(z)}^p}
	{\abs{y-z}^{\dmnMinusOne}}
	\int_{A(y,z)}\dist(x,\partial\Omega)^{-p\theta-\dmn}\,dx
	\,d\sigma(y)\,d\sigma(z)
\end{aligned}\end{multline*}
where $A(y,z)= \{x\in\Omega: y\in\Delta''(x), \>z\in\Delta''(x)\}$. Notice that if $x\in A(y,z)$ then $$\dist(x,\partial\Omega)\approx\abs{x-y}\approx\abs{x-z};$$ thus, it may be readily seen that the inner integral is at most $C\abs{y-z}^{-p\theta}$, and so 
\begin{multline*}
\int_\Omega
\doublebar{\nabla^m\widetilde{\mathcal{E}}\arr\varphi}_{L^\infty(B(x,\Omega))}^p
\,\dist(x,\partial\Omega)^{p-1-p\theta}\,dx
\\\begin{aligned}
&\leq
	C(p)
	\int_{\partial\Omega}
	\int_{\partial\Omega}
	\frac{\abs{\arr\varphi(y)-\arr\varphi(z)}^p}
	{\abs{y-z}^{\dmnMinusOne+p\theta}}
	\,d\sigma(y)\,d\sigma(z)
\end{aligned}\end{multline*}
as desired.

Finally, suppose that $\pmin<p\leq 1$. Again take $q=1$. Recall that $\arr\varphi=\sum_j \lambda_j \arr a_j$, where each $\arr a_j$ is an atom supported in $B(x_j,r_j)\cap\partial\Omega$, and where $\doublebar{\arr\varphi}_{\dot B^{p,p}_\theta(\partial\Omega)}^p \approx \sum_j \abs{\lambda_j}^p$.
Then
\begin{multline*}
\int_\Omega
\doublebar{\nabla^m\widetilde{\mathcal{E}}\arr\varphi}_{L^\infty(B(x,\Omega))}^p
\,\dist(x,\partial\Omega)^{p-1-p\theta}\,dx
\\\begin{aligned}
&\leq
	C\int_\Omega
	\biggl( 
	\int_{\Delta''(x)}
	\int_{\Delta''(x)}
	\frac{\abs{\arr\varphi(y)-\arr\varphi(z)}}
	{\abs{y-z}^{\dmnMinusOne}}
	\,d\sigma(y)\,d\sigma(z) \biggr)^p \!\!\!
	\,\dist(x,\partial\Omega)^{p-1-p\theta-p\pdmn}\,dx
.\end{aligned}\end{multline*}
But if $p\leq 1$, then
\begin{multline*}
\biggl( 
	\int_{\Delta''(x)}
	\int_{\Delta''(x)}
	\frac{\abs{\arr\varphi(y)-\arr\varphi(z)}}
	{\abs{y-z}^{\dmnMinusOne}}
	\,d\sigma(y)\,d\sigma(z) \biggr)^p
\\\begin{aligned}
&=
	\biggl(\int_{\Delta''(x)}
	\int_{\Delta''(x)}
	\frac{\abs{\sum_j\lambda_j(\arr a_j(y)-\arr a_j(z))}}
	{\abs{y-z}^{\dmnMinusOne}}
	\,d\sigma(y)\,d\sigma(z)\biggr)^p
\\&\leq
	\sum_j \abs{\lambda_j}^p \biggl(\int_{\Delta''(x)}
	\int_{\Delta''(x)}
	\frac{\abs{\arr a_j(y)-\arr a_j(z)}}
	{\abs{y-z}^{\dmnMinusOne}}
	\,d\sigma(y)\,d\sigma(z)\biggr)^p
.\end{aligned}\end{multline*}
Now, if $\arr a_j$ is not identically zero in $\Delta''(x)$, then $r_j+C_2\dist(x,\partial\Omega)>\abs{x-x_j}$, so either $\abs{x-x_j}<2r_j$ or $\abs{x-x_j}\approx \dist(x,\partial\Omega)$. If $\abs{x-x_j}<2r_j$, then by Definition~\ref{dfn:besov},
\begin{multline*}
\biggl(\int_{\Delta''(x)}
	\int_{\Delta''(x)}
	\frac{\abs{\arr a_j(y)-\arr a_j(z)}}
	{\abs{y-z}^{\dmnMinusOne}}
	\,d\sigma(y)\,d\sigma(z)\biggr)^p
\\\begin{aligned}
&\leq
	r_j^{p\theta-p-\pdmnMinusOne}
	\biggl(\int_{\Delta''(x)}
	\int_{\Delta''(x)}
	\frac{1}
	{\abs{y-z}^{d-2}}
	\,d\sigma(y)\,d\sigma(z)\biggr)^p
\\&\leq C r_j^{p\theta-p-\pdmnMinusOne}\dist(x,\partial\Omega)^{p\pdmn}
.\end{aligned}\end{multline*}
In the other case, if $2r_j<\abs{x-x_j}\approx \dist(x,\partial\Omega)$, then because $a_j$ is supported in ${B(x_j,2r_j)\cap\partial\Omega}$ we have that
\begin{multline*}
\biggl(\int_{\Delta''(x)}
	\int_{\Delta''(x)}
	\frac{\abs{\arr a_j(y)-\arr a_j(z)}}
	{\abs{y-z}^{\dmnMinusOne}}
	\,d\sigma(y)\,d\sigma(z)\biggr)^p
\\\begin{aligned}
&\leq
	\biggl(\int_{B(x_j,2r_j)\cap\partial\Omega}
	\int_{B(x_j,2r_j)\cap\partial\Omega}
	\frac{\abs{\arr a_j(y)-\arr a_j(z)}}
	{\abs{y-z}^{\dmnMinusOne}}
	\,d\sigma(y)\,d\sigma(z)\biggr)^p
	\\&\qquad+
	2\biggl(\int_{B(x_j,2r_j)\cap\partial\Omega}
	\int_{\Delta''(x)\setminus B(x_j,2r_j)}
	\frac{\abs{\arr a_j(z)}}
	{\abs{y-z}^{\dmnMinusOne}}
	\,d\sigma(y)\,d\sigma(z)\biggr)^p
.\end{aligned}\end{multline*}
We bound the first integral as before. To bound the second integral, we observe that
\begin{equation*}{\Delta''(x)\setminus B(x_j,2r_j)} \subset \bigcup_{k=0}^K B(x_j,2^{k+1}r_j)\setminus B(x_j,2^{k}r_j)\end{equation*}
where $K=C \ln (\dist(x,\partial\Omega)/r_j)$. Furthermore,
\begin{equation*}{\int_{B(x_j,2r_j)\cap\partial\Omega}
	\int_{\partial\Omega \cap B(x_j,2^{k+1}r_j)\setminus B(x_j,2^{k}r_j)}
	\frac{1}
	{\abs{y-z}^{\dmnMinusOne}}
	\,d\sigma(y)\,d\sigma(z)}
	\leq C r_j^{\dmnMinusOne}
\end{equation*}
and so by Definiton~\ref{dfn:besov},
\begin{multline*}
\biggl(\int_{\Delta''(x)}
	\int_{\Delta''(x)}
	\frac{\abs{\arr a_j(y)-\arr a_j(z)}}
	{\abs{y-z}^{\dmnMinusOne}}
	\,d\sigma(y)\,d\sigma(z)\biggr)^p
\\\begin{aligned}
&\leq
	Cr_j^{\pdmnMinusOne(p-1)+\theta p}\bigl(\ln(\dist(x,\partial\Omega)/r_j)\bigr)^p
.\end{aligned}\end{multline*}
Thus,
\begin{multline*}
\int_\Omega
\doublebar{\nabla^m\widetilde{\mathcal{E}}\arr\varphi}_{L^\infty(B(x,\Omega))}^p
\,\dist(x,\partial\Omega)^{p-1-p\theta}\,dx
\\\begin{aligned}
&\leq
	C\sum_j \abs{\lambda_j}^p 
	\int_{\abs{x-x_j}<2r_j}
	r_j^{p\theta-p-\pdmnMinusOne}
	\,\dist(x,\partial\Omega)^{p-1-p\theta}\,dx
	\\&\qquad
	+C\sum_j \abs{\lambda_j}^p 
	\int_{{2r_j<\abs{x-x_j} }}
	\frac{\bigl(\ln(\abs{x-x_j}/r_j)\bigr)^p} {r_j^{\pdmnMinusOne(1-p)-\theta p} } 
	\abs{x-x_j}^{p-1-p\theta-p\pdmn}\,dx.
\end{aligned}\end{multline*}
The first integral converges because $p>0$ and $\theta<1$, while the second integral converges because $p>\pmin$. Thus
\begin{equation*}
\int_\Omega
\doublebar{\nabla^m\widetilde{\mathcal{E}}\arr\varphi}_{L^\infty(B(x,\Omega))}^p
\,\dist(x,\partial\Omega)^{p-1-p\theta}\,dx
\leq C\sum_j \abs{\lambda_j}^p
\end{equation*}
as desired.

We now need to show that $\Tr_{m-1}\mathcal{E}\arr\varphi=\arr\varphi$. Recall that if $\abs{\gamma}= m-1$, then for all $z\in\partial\Omega$ and all $x\in\Omega$ sufficiently close to~$\partial\Omega$, we have that
\begin{align*}
\partial^\gamma\mathcal{E}\arr \varphi(x)
&=
	\sum_\delta 	\sum_\zeta \frac{\gamma!}{\delta!(\gamma-\delta)!} 
	\int_{\partial\Omega} \partial_x^{\gamma-\delta} K(x,y)
	\frac{(x-y)^{\zeta-\delta}}{(\zeta-\delta)!} (\varphi_{\zeta}(y)-P_{\zeta}(y,z))\,d\sigma(y)
	\\&\qquad+
	 \int_{\partial\Omega} K(x,y)\,\varphi_\gamma(y)\,d\sigma(y)
\end{align*}
where the sums are over all $\delta$ with $\delta<\gamma$ and $\abs{\delta}\leq m-1$, and over all $\zeta$ with $\delta\leq\zeta$ and $\abs{\zeta}\leq m-1$. 
Observe that because $P_\gamma(y,z)=\varphi_\gamma(z)$, we have that
\begin{align*}\int_{\partial\Omega} K(x,y)\,\varphi_\gamma(y)\,d\sigma(y) 
&= \int_{\partial\Omega} K(x,y)\,(\varphi_\gamma(y)-P_\gamma(y,z))\,d\sigma(y)
	\\&\qquad+\int_{\partial\Omega} K(x,y)\,\varphi_\gamma(z)\,d\sigma(y)
\end{align*}
and so we may write
\begin{align*}
\partial^\gamma\mathcal{E}\arr \varphi(x)
&=
	\sum_\delta 	\sum_\zeta \frac{\gamma!}{\delta!(\gamma-\delta)!} 
	\int_{\partial\Omega} \partial_x^{\gamma-\delta} K(x,y)
	\frac{(x-y)^{\zeta-\delta}}{(\zeta-\delta)!} (\varphi_{\zeta}(y)-P_{\zeta}(y,z))\,d\sigma(y)
	\\&\qquad
	+\int_{\partial\Omega} K(x,y)\,\varphi_\gamma(z)\,d\sigma(y)
\end{align*}
where the sums are now over all $\delta$ with $\delta\leq\gamma$. 
Recall that by assumption on~$K$ the second integral is equal to~$\varphi_\gamma(z)$; we need only show that as $x\to z$ in some sense the first term vanishes.

Fix some $z\in \partial\Omega$. 
Recall that
\begin{equation*}P_\zeta(y,z) =\sum_{\abs{\xi}\leq m-1-\abs{\zeta}} \frac{(y-z)^\xi}{\xi!} \partial^\xi\varphi_{\zeta}(z). \end{equation*}
Let $f(r)=\varphi_\zeta(z+r(y-z))$. Then
\begin{equation*}\varphi_\zeta(y)-P_\zeta(y,z) = f(1)-\sum_{j=0}^{m-1-\abs{\zeta}} \frac{1}{j!} f^{(j)}(0).\end{equation*}
By induction, we may establish that
\begin{equation*}f(1)-\sum_{j=0}^{n} \frac{1}{j!} f^{(j)}(0)
=\int_0^1 \int_{0}^{r_1}\dots \int_0^{r_{n-1}} \bigr(f^{(n)}(r_n)-f^{(n)}(0)\bigl)\,dr_n \dots dr_2\,dr_1.
\end{equation*}
Notice that this is not quite the standard form of the Taylor remainder of single-variable calculus.
Then 
\begin{align*}
\abs{f^{(n)}(r)-f^{(n)}(0)}
&\leq 
	C(\theta)\,r^\theta \doublebar{f^{(n)}}_{\dot C^\theta((0,1))}
	\leq
	C(\theta)\,r^\theta\abs{y-z}^{n+\theta}\doublebar{\nabla^n \varphi_\zeta}_{\dot C^\theta(\R^\dmn)}
.\end{align*}
Let $n=m-1-\abs{\zeta}$. If $p=\infty$ then by assumption $\doublebar{\nabla^{m-1} \varphi}_{\dot C^\theta(\R^\dmn)}<\infty$, while if $p<\infty$ then by assumption $\varphi$ is smooth and compactly supported. In either case, we have that
\begin{equation*}\abs{\varphi_\zeta(y)-P_\zeta(y,z)}\leq C\abs{y-z}^{m-1-\abs{\zeta}+\theta}\doublebar{\nabla^{m-1} \varphi}_{\dot C^\theta(\R^\dmn)}\end{equation*}
and the right-hand side is finite.

Recall that if $j\geq 0$, then $\abs{\nabla_x^j K(x,y)}\leq C_j\dist(x,\partial\Omega)^{1-\pdmn-j}$. Furthermore, recall that $K(x,y)=0$ unless $\abs{x-y}<2\dist(x,\partial\Omega)$. Finally, observe that $\dist(x,\partial\Omega)\leq \abs{x-z}$.
Thus, 
\begin{equation*}
	\int_{\partial\Omega} 
	\abs[big]{\partial_x^{\gamma-\delta} K(x,y)
	 (x-y)^{\zeta-\delta} (\varphi_{\zeta}(y)-P_{\zeta}(y,z))}\,d\sigma(y)
	 \leq
	 C\doublebar{\nabla^{m-1} \varphi}_{\dot C^\theta(\R^\dmn)}
	 \abs{x-z}^{\theta}
\end{equation*}
and so $\partial^\gamma \mathcal{E}\arr\varphi(x)\to \varphi_\gamma(z)$ as $\abs{x-z}\to 0$. This completes the proof.

\section{Traces: Dirichlet boundary data}
\label{sec:trace:dirichlet}

In this section we complete the proof of Theorem~\ref{thm:dirichlet} by proving the following theorem.

\begin{thm}
\label{thm:trace}
Suppose that 
$0<\theta<1$, that $\pdmnMinusOne/(\dmnMinusOne+\theta)<p\leq\infty$, and that $1\leq q\leq \infty$. Let $\Omega$ be a Lipschitz domain with connected boundary. 

Then the trace operator $\Tr_{m-1}$ is bounded $\dot W_{m,av}^{p,\theta,q}(\Omega)\mapsto \dot W\!A^p_{\theta}(\partial\Omega)$.

If $p=1$, this is true whether we use atoms or the norm~\eqref{eqn:Slobodekij} to characterize $\dot B^{p,p}_\theta(\partial\Omega)$; that is, 
\begin{align*}\int_{\partial\Omega} \int_{\partial\Omega} \frac{\abs{\Tr_{m-1}^\Omega \Phi(x)-\Tr_{m-1}^\Omega \Phi(y)}}{\abs{x-y}^{\dmnMinusOne+\theta}}\,d\sigma(x)\,d\sigma(y)
&\leq C \doublebar{\Phi}_{\dot W^{1,\theta,q}_{m,av}(\Omega)}
,\\
\inf\Bigl\{\sum_j\abs{\lambda_j}:\Tr_{m-1}^\Omega \Phi= \arr c_0+\sum_j \lambda_j\,\arr  a_j,\> \arr c_0\text{ constant},\> \arr  a_j\text{  atoms}\Bigr\}
&\leq C\doublebar{\Phi}_{\dot W^{1,\theta,q}_{m,av}(\Omega)}
\end{align*}
for all $\Phi\in {\dot W^{1,\theta,q}_{m,av}(\Omega)}$.
\end{thm}

As mentioned in Remark~\ref{rmk:p=1}, the $m=1$ cases of this theorem and of Theorem~\ref{thm:trace} imply that the atomic characterization and the norm~\eqref{eqn:Slobodekij} are equivalent in the case $p=1$.

The remainder of Section~\ref{sec:trace:dirichlet} will be devoted to a proof of this theorem.


\subsection{The case \texorpdfstring{$p=\infty$}{of Holder continuous functions}} 

In this section we will prove Theorem~\ref{thm:trace} in the case $p=\infty$. 

We must show that if $\varphi\in \dot W_{m,av}^{\infty,\theta,q}(\Omega)$, then $\Tr_{m-1}^\Omega\varphi\in \dot W\!A^\infty_\theta(\partial\Omega)$ with $\dot W\!A^\infty_\theta$-norm controlled by the $\dot W_{m,av}^{\infty,\theta,q}(\Omega)$-norm of~$\varphi$.
Recall from the definition~\ref{dfn:whitney} that 
\begin{gather*}\dot W\!A^\infty_\theta(\partial\Omega)=\{\Tr_{m-1}^\Omega\Phi:\nabla^{m-1}\Phi\in \dot C^\theta(\Omega)\},\\
\doublebar{\Tr_{m-1}^\Omega \varphi}_{\dot W\!A^\infty_\theta(\partial\Omega)}= \doublebar{\Tr_{m-1}^\Omega \varphi}_{\dot C^\theta(\partial\Omega)}.\end{gather*} Thus, to prove Theorem~\ref{thm:trace} in the case $p=\infty$, we must show both that $\Tr_{m-1}^\Omega \varphi$ is H\"older continuous, and that there is a function $\Phi=T\varphi$ in $\dot C^{m-1,\theta}(\Omega)$ such that $\Tr_{m-1}^\Omega \varphi=\Tr_{m-1}^\Omega T\varphi$. 


Furthermore, recall that we are using the Sobolev space definition of the trace map. That is, by Lemma~\ref{lem:L:L1}, if $\varphi\in \dot W_{m,av}^{\infty,\theta,q}(\Omega)$, then $\varphi \in \dot W^1_m(V)$ for any $V\subset \Omega$ bounded, and so we may define $\Tr_{m-1}^\Omega$ on $\dot W_{m,av}^{\infty,\theta,q}(\Omega)$ using its definition on $\dot W^1_{m,loc}(\overline\Omega)$.

Let $\delta(x)$ be the adapted distance function introduced in the proof of \cite[Theorem~7]{Dah86B}.
Specifically, if $V$ is a Lipschitz graph domain $V=\{(x',t):x'\in\R^\dmnMinusOne,\>t>\psi(x)\}$, let $\rho(x',t)=ct+\theta_t*\psi(x)$, where $\theta$ is smooth, compactly supported, and integrates to~$1$, and where $\theta_t(y)=t^{-\pdmnMinusOne}\theta(y/t)$. It is possible to choose $c$ large enough that $\partial_t \rho(x',t)>1$ for all $x'\in\R^\dmnMinusOne$ and all $t>0$. We let $\delta(x',t)$ satisfy $\rho(x',\delta(x',t))=(x',t)$. Then 
$\delta$ satisfies 
\begin{equation}\label{eqn:distance-function}\delta(x)\approx\dist(x,\partial\Omega) \quad\text{and}\quad \abs{\nabla^k\delta(x)}\leq C\dist(x,\partial\Omega)^{1-k}\end{equation}
for all $0\leq k\leq m+1$. Using a partition of unity argument, we may construct a function $\delta(x)$ that satisfies the conditions~\eqref{eqn:distance-function} if $\Omega$ is a bounded Lipschitz domain as well.

Suppose that $\phi\in \dot W^1_{m,loc}(\overline\Omega)$.
As in Section~\ref{sec:extension:dirichlet}, let $p(x)=P(x,y)$ be the Taylor polynomial of $\phi$ about the point $y$ of order $m-1$,
\begin{equation*}P(x,y)=\sum_{\abs{\zeta}\leq m-1} \frac{1}{\zeta!} \partial^\zeta\phi(y)\, (x-y)^\zeta.\end{equation*}

Let $\eta$ be smooth, radial and compactly supported, with $\int_{\R^\dmn} \eta=1$. We will impose further conditions on $\eta$ momentarily. Let $K(x,y)= \delta(x)^{-\pdmn}\eta\bigl(\delta(x)^{-1}(y-x)\bigr)$, so that $\int_\Omega K(x,y)\,dy=1$ for each $x\in\Omega$. (We will use this kernel $K$ on $\Omega\times\Omega$; this differs from the kernel of Section~\ref{sec:extension:dirichlet} inasmuch as that kernel was used on $\Omega\times\partial\Omega$.)

Define
\begin{equation*}T\phi(x)=\int_\Omega K(x,y)\,P(x,y)\,dy
.\end{equation*}
Then $T\phi$ is locally $C^{m+1}$ in~$\Omega$. We will show that, if $V\subset\Omega$ is a bounded set, then $T$ is a bounded operator $\dot W^1_{m}(U)\mapsto \dot W^1_m(V)$ for some bounded set $U$ with $V\subseteq U\subseteq\Omega$. We will also show that if $\phi$ is smooth, then $\nabla^{m-1}T\phi$ is continuous up to the boundary and satisfies $\nabla^{m-1}\phi=\nabla^{m-1} T\phi$ on~$\partial\Omega$; by the definition of the trace map, this implies that $\Tr_{m-1}^\Omega\phi=\Tr_{m-1}^\Omega T\phi$ for any $\phi\in \dot W^1_{m,loc}(\overline\Omega)$. Finally, we will show that if $\varphi \in  \dot W_{m,av}^{\infty,\theta,q}(\Omega)$ then $\nabla^{m-1} T\varphi$ is H\"older continuous in~$\Omega$, as desired.

Suppose that $\gamma$ is a multiindex with $\abs{\gamma}=m-1$ or $\abs{\gamma}=m$. Then
\begin{align*}\partial^\gamma T\phi(x) 
&= \sum_{\xi\leq \gamma} 
	\frac{\gamma!}{\xi!(\gamma-\xi)!} \int_\Omega \partial_x^\xi K(x,y)\, \partial_x^{\gamma-\xi} P(x,y)\,dy
.\end{align*}


By definition of $P(x,y)$, we have that
\begin{align*}\partial^\gamma T\phi(x) 
&= 
	 \sum_{\substack{\xi\leq \gamma\\ \hbox to 0pt{\hss$\scriptstyle\abs{\zeta}\leq m-1 ,\> \zeta\geq \gamma-\xi$\hss}}  } C_{\gamma,\xi,\zeta} 
	\int_\Omega \partial_x^\xi K(x,y)\,\partial^\zeta\phi(y)\, (x-y)^{\zeta+\xi-\gamma}\,dy
\end{align*}
for some constants $C_{\gamma,\xi,\zeta}$. 

Let $a>0$ be a number such that $K(x,y)$ (regarded as a function of~$y$) is supported in ${B(x,a\dist(x,\partial\Omega))}$; by choosing $\eta$ appropriately we may make $a$ as small as we like.
Let $\widetilde P_x(y)$ be the polynomial of degree $m-1$ such that
\begin{equation*}\int_{B(x,a\dist(x,\partial\Omega))}
\bigl(\partial^\zeta\phi(y) - \partial_y^\zeta\widetilde P_{x}(y)\bigr)\,dy=0\end{equation*}
for any multiindex $\zeta$ with $0\leq\abs{\zeta}\leq m-1$.

Then
\begin{align*}\partial^\gamma T\phi(x) 
&= 
	 \sum_{\substack{\abs\zeta\leq m-1 \\ \gamma-\zeta\leq \xi\leq\gamma}  } C_{\gamma,\zeta,\xi} 
	\int_\Omega \partial_x^\xi K(x,y)\,\partial_y^\zeta(\phi(y)-\widetilde P_x(y))\, (x-y)^{\zeta+\xi-\gamma}\,dy
	\\&\qquad+
	 \sum_{\substack{\abs\zeta\leq m-1 \\ \gamma-\zeta\leq \xi\leq\gamma}  } C_{\gamma,\zeta,\xi}
	\int_\Omega \partial_x^\xi K(x,y)\,\partial_y^\zeta \widetilde P_x(y)\, (x-y)^{\zeta+\xi-\gamma}\,dy
\\&=I(x)+II(x)
.\end{align*}
By definition of~$K$,
\begin{align*}
\abs{I(x)}&\leq
\sum_{\substack{\abs\zeta\leq m-1 \\ \gamma-\zeta\leq \xi\leq\gamma}  } C_{\gamma,\zeta} 
\dist(x,\partial\Omega)^{\abs\zeta-\abs\gamma-\pdmn} \int_{B(x,a\dist(x,\partial\Omega))} 
\abs{\partial_y^\zeta(\phi(y)-\widetilde P_x(y))} \,dy
.\end{align*}
We may control the integral by the Poincar\'e inequality, and so
\begin{equation*}\abs{I(x)}\leq C\dist(x,\partial\Omega)^{m-\abs{\gamma}}\fint_{B(x,a\dist(x,\partial\Omega))} \abs{\nabla^m\phi}.\end{equation*}
In particular, notice that if $\abs{\gamma}=m-1$ and $\phi$ is smooth then $I(x)\to 0$ as $x\to\partial\Omega$, and so $\partial^\gamma T\phi = II$ on~$\partial\Omega$. 

We now consider the second term $II(x)$.
We impose the additional requirement that $\int\eta(y)\,y^\zeta\,dy=0$ for all $\zeta$ with $1\leq\abs{\zeta}\leq m$; this implies that $\int K(x,y) \,p(y)\,dy=p(x)$ for any polynomial of degree at most~$m$. Thus, 
\begin{align*}
II (x)
&=
	\sum_{\substack{\abs\zeta\leq m-1 \\ \gamma-\zeta\leq \xi\leq\gamma}  } C_{\gamma,\zeta,\xi}\,
	\partial_z^{\xi} \bigl(\partial_z^\zeta\widetilde P_x(z)(x-z)^{\zeta+\xi-\gamma}\bigr)\big\vert_{z=x}
\\&=
	\sum_{\substack{\abs\zeta\leq m-1 \\ \gamma-\zeta\leq \xi\leq\gamma}  } C_{\gamma,\zeta,\xi}
	\sum_{\alpha\leq\xi}
	\frac{\xi!}{\alpha!(\xi-\alpha)!}
	\bigl(\partial_z^{\zeta+\alpha}\widetilde P_x(z) \, \partial_z^{\xi-\alpha}(x-z)^{\zeta+\xi-\gamma}\bigr)\big\vert_{z=x}
.\end{align*}
Notice that $\partial_z^{\xi-\alpha}(x-z)^{\zeta+\xi-\gamma}\big\vert_{z=x}=0$ unless $\alpha=\gamma-\zeta$, in which case it is a constant depending only on $\zeta$, $\xi$ and~$\gamma$. Thus, there is some constant $C_\gamma$ such that 
\begin{align*}
II (x)
&=
	C_\gamma
	\bigl(\partial_z^{\gamma}\widetilde P_x(z) \bigr)\big\vert_{z=x}
.\end{align*}
If $\abs{\gamma}=m$ then $\partial_z^\gamma \widetilde P_x(z)=0$ and so $II(x)=0$. 
If $\abs{\gamma}=m-1$, then 
\begin{equation*}\partial_z^\gamma \widetilde P_x(z) = \fint_{B(x,a\dist(x,\partial\Omega))} \partial^\gamma \phi
\quad\text{and so}\quad
II(x)
=
	C_\gamma
	\fint_{B(x,a\dist(x,\partial\Omega))} \partial^\gamma \phi
.\end{equation*}
We now claim that $C_\gamma=1$ whenever $\abs{\gamma}=m-1$. This may be most easily seen by observing that, if $\phi(x)$ is a polynomial of degree $m-1$, then $P(x,y)=\phi(x)$ and so $T\phi(x)=\phi(x)$, and also that $\widetilde P_x(y)=\phi(y)$ and so $I(x)=0$. In particular, if $\phi(x)=x^\gamma$ then
\begin{align*}
\gamma!=\partial^\gamma x^\gamma
&= \partial^\gamma T\phi(x) = II (x)
= 
C_\gamma
	\fint_{B(x,a\dist(x,\partial\Omega))} \partial^\gamma y^\gamma \,dy
	=C_\gamma \,\gamma!
\end{align*}
and so $C_\gamma=1$.

By our above bound on $I(x)$, if $\phi \in \dot W^1_{m,loc}(\overline\Omega)$, then
\begin{align*}
\abs{\nabla^m T\phi(x)}&\leq C \fint_{B(x)} \abs{\nabla^m \phi},\\
\abs[bigg]{\nabla^{m-1} T\phi(x)-\fint_{B(x)} \nabla^{m-1} \phi}
&\leq C \dist(x,\partial\Omega)\fint_{B(x)} \abs{\nabla^m \phi}
\end{align*}
where $B(x)={B(x,a\dist(x,\partial\Omega))}$.
Thus, if $\phi$ is smooth, then $\nabla^{m-1}T\phi(x)$ is continouous up to the boundary and satisfies $\nabla^{m-1}T\phi=\nabla^{m-1}\phi$ on~$\partial\Omega$. Furthermore, using a Whitney decomposition, we see that that $T$ is bounded on $\dot W^1_{m,loc}(\overline\Omega)$, and so by density $\Tr_{m-1}^\Omega T\phi=\Tr_{m-1}^\Omega\phi$ for all $\phi\in \dot W^1_{m,loc}(\overline\Omega)$.

We now return to the case of functions $\varphi\in \dot W_{m,av}^{\infty,\theta,q}(\Omega)$. By the definition \eqref{eqn:norm:L} and by H\"older's inequality,  if $q\geq 1$ then
\begin{equation*}
\fint_{B(x,\dist(x,\partial\Omega)/2)} \abs{\nabla^m \varphi} \leq \doublebar{\varphi}_{\dot W_{m,av}^{\infty,\theta,q}(\Omega)} \dist(x,\partial\Omega)^{\theta-1}.\end{equation*}
Thus
\begin{equation*}
\abs{\nabla^m T\varphi(x)}\leq
	C \doublebar{\varphi}_{\dot W_{m,av}^{\infty,\theta,q}(\Omega)} \dist(x,\partial\Omega)^{\theta-1}
.\end{equation*}
From this we may easily show that, if $\Omega$ is a Lipschitz domain, then $\nabla^{m-1}T\varphi$ is H\"older continuous in~$\Omega$ with exponent $\theta$ and $\dot C^\theta$-norm $C\doublebar{\varphi}_{\dot W_{m,av}^{\infty,\theta,q}(\Omega)} $. Thus, $\Tr_{m-1}^\Omega \varphi=\Tr_{m-1}^\Omega T\varphi$ lies in the space $\dot W\!A^\infty_\theta(\partial\Omega)$, as desired.

\subsection{The case \texorpdfstring{$p<\infty$}{of finite p}} We now consider traces of $\dot W_{m,av}^{p,\theta,q}(\Omega)$ for $p<\infty$.
If $\Omega=\R^\dmn_+$ is a half-space, then the following trace theorem was established in \cite{BarM16A}.
\begin{thm}[{\cite[Theorems~6.3 and~6.9]{BarM16A}}] \label{thm:trace:half-space}
Suppose  $1\leq q\leq \infty$, $0<\theta<1$ and $\pdmnMinusOne/(\dmnMinusOne+\theta)<p<\infty$. Then the trace operator $\Trace$ extends to an operator that is bounded 
\begin{equation*}\Trace:\dot W_{1,av}^{p,\theta,q}(\R^\dmn_+)\mapsto \dot B^{p,p}_\theta(\R^\dmnMinusOne).\end{equation*}
\end{thm}
Observe that we may extend Theorem~\ref{thm:trace:half-space} to any Lipschitz graph domain $\Omega=\{(x',t):t>\psi(x')\}$ by means of the change of variables $(x',t)\mapsto (x',t-\psi(x'))$. 
To complete the proof in the case $m=1$, we need only extend Theorem~\ref{thm:trace:half-space} to Lipschitz domains with compact boundary.

Let $\Omega$ be such a domain, and let $u\in \dot W_{1,av}^{p,\theta,q}(\Omega)$.
Let $\{\varphi_j\}$ be a set of smooth functions such that $\sum_{j=1}^n\varphi_j=1$ in a neighborhood of $\partial\Omega$, where each $\varphi_j$ is supported in the ball $B(x_j,(3/2)r_j)$, where $x_j$ and $r_j$ are as in Definition~\ref{dfn:domain}. 

By Lemma~\ref{lem:L:L1}, we have that $\nabla u\in L^1(B(0,R)\cap\Omega)$ for any $R>0$. Let $u_\Omega=\fint_{\partial\Omega} u\,d\sigma$. Let $u_j(x)=(u(x)-u_\Omega)\varphi_j(x)$. Then $u(x)=u_\Omega+\sum_j u_j(x)$. Notice that constants have $\dot B^{p,p}_\theta(\partial\Omega)$-norm zero, and so we may neglect the $u_\Omega$ term.

We first show that $u_j\in \dot W_{1,av}^{p,\theta,q}(\Omega)$. Let the tents $T(Q)$ be as in Lemma~\ref{lem:Poincare:L}. Notice that $\varphi_j$ is supported in a tent $T(Q_j)$ for some cube~$Q$. By Lemma~\ref{lem:Poincare:L}, we have that $\varphi_j (u-u_{Q_j})\in \dot W_{1,av}^{p,\theta,q}(\Omega)$, where $u_{Q_j}=\fint_{W(Q)} u$. 
By Lemma~\ref{lem:L:L1}, and by boundedness of the trace map from $L^1(U)\cap \dot W^1_1(U)$ to $L^1(\partial U)$ for any Lipschitz domain~$U$, we have that $\abs{u_{Q_j}-u_\Omega}\leq C r_\Omega^{\theta-\pdmnMinusOne/p} \doublebar{\nabla u}_{L_{av}^{p,\theta,q}(\Omega)}$. This implies that $\doublebar{u_j}_{\dot W_{1,av}^{p,\theta,q}(\Omega)} \leq C \doublebar{u}_{\dot W_{1,av}^{p,\theta,q}(\Omega)}$.

If $V_j$ is the Lipschitz graph domain associated to the point $x_j$ in Definition~\ref{dfn:domain}, we have that $u_j\in \dot W_{1,av}^{p,\theta,q}(V_j)$ and so $\Trace u_j\in \dot B^{p,p}_\theta(\partial V_j)$. We now prove the following lemma.

\begin{lem} Let $\Omega$ be a Lipschitz domain, $V$ a Lipschitz graph domain, and suppose that $B(x_0,2r)\cap\Omega=B(x_0,2r)\cap V$, for some $x_0\in \partial\Omega$ and some $r>0$.
Let $0<\theta<1$ and let $\pdmnMinusOne/(\dmnMinusOne+\theta)<p\leq\infty$.

If $f$ is supported in $B(x_0,(3/2)r)\cap\partial\Omega$ and $f\in \dot B^{p,p}_\theta(\partial V)$, then $f\in \dot B^{p,p}_\theta(\partial \Omega)$ with $\doublebar{f}_{\dot B^{p,p}_\theta(\partial \Omega)}\leq C\doublebar{f}_{\dot B^{p,p}_\theta(\partial V)}$. (If $p=1$ we may use either atomic norms or the norm~\eqref{eqn:Slobodekij}.)
\end{lem}

\begin{proof}
Suppose first that $1\leq p\leq \infty$. We must bound the norm~\eqref{eqn:Slobodekij}. We will divide $\partial\Omega$ into the two regions ${\partial\Omega\cap B(x_0,2r)}$ and $\partial\Omega\setminus B(x_0,2r)$; because the norm \eqref{eqn:Slobodekij} involves two integrals over~$\partial\Omega$, this leaves us with four integrals to bound.

Because ${\partial\Omega\cap B(x_0,2r)}={\partial V \cap B(x_0,2r)}$, we have that
\begin{equation*}\int_{\partial\Omega\cap B(x_0,2r)}
\int_{\partial\Omega\cap B(x_0,2r)} \frac{\abs{ f(x)- f(y)}^p}{\abs{x-y}^{\dmnMinusOne+p\theta}}\,d\sigma(x)\,d\sigma(y)
\leq C\doublebar{f}_{\dot B^{p,p}_\theta(\partial V)}^p
.\end{equation*}
Because $f$ is supported in $B(x_0,(3/2)r)\subset B(x_0,2r)$, we
have that 
\begin{equation*}\int_{\partial\Omega\setminus B(x_0,2r)}
\int_{\partial\Omega\setminus B(x_0,2r)} \frac{\abs{ f(x)-f(y)}^p}{\abs{x-y}^{\dmnMinusOne+p\theta}}\,d\sigma(x)\,d\sigma(y)
=0
.\end{equation*}
By symmetry, and because $f$ is supported in $B(x_0,(3/2)r)$, we need only bound
\begin{equation*}\int_{\partial\Omega\setminus B(x_0,2r)}
\int_{\partial \Omega \cap B(x_0,(3/2)r)} \frac{\abs{ f(x)}^p}{\abs{x-y}^{\dmnMinusOne+p\theta}}\,d\sigma(x)\,d\sigma(y)
.\end{equation*}
We have a bound in $V$, that is,
\begin{equation*}\int_{\partial V\setminus B(x_0,2r)}
\int_{\partial \Omega \cap B(x_0,(3/2)r)} \frac{\abs{ f(x)}^p}{\abs{x-y}^{\dmnMinusOne+p\theta}}\,d\sigma(x)\,d\sigma(y)
\leq C\doublebar{f}_{\dot B^{p,p}_\theta(\partial V)}^p
.\end{equation*}
But if $y\notin B(x_0,2r)$ and $x\in B(x_0,(3/2)r)$, then $\abs{x-y}\approx\abs{x_0-y}$. Thus
\begin{equation*}\int_{\partial V\setminus B(x_0,2r)} \frac{d\sigma(y)}{\abs{x_0-y}^{\dmnMinusOne+p\theta}}
\int_{\partial \Omega\cap B(x_0,2r)} {\abs{ f(x)}^p}\,d\sigma(x)
\leq C\doublebar{f}_{\dot B^{p,p}_\theta(\partial V)}^p.
\end{equation*}
Estimating the first integral, we see that 
\begin{equation*}
\int_{\partial \Omega\cap B(x_0,2r)} {\abs{ f(x)}^p}\,d\sigma(x)
\leq Cr^{p\theta}\doublebar{f}_{\dot B^{p,p}_\theta(\partial V)}^p.
\end{equation*}
Again using the relation $\abs{x-y}\approx\abs{x_0-y}$, we see that 
\begin{equation*}\int_{\partial\Omega\setminus B(x_0,2r)}
\int_{\partial\Omega\cap B(x_0,(3/2)r)} \frac{\abs{ f(x)}^p}{\abs{x-y}^{\dmnMinusOne+p\theta}}\,d\sigma(x)\,d\sigma(y)
\leq C\doublebar{f}_{\dot B^{p,p}_\theta(\partial V)}^p.\end{equation*}
Thus, $ f\in \dot B^{p,p}_\theta(\partial\Omega)$, as desired.

If $\pdmnMinusOne/(\dmnMinusOne+\theta)<p\leq 1$, recall that we characterize $\dot B^{p,p}_\theta(\partial\Omega)$ using atoms. Thus, we may write $f=\sum_k \lambda_k \, a_k$, where $a_k$ is a $\dot B^{p,p}_\theta(\partial V)$-atom and where $\sum_k \abs{\lambda_k}^p\approx \doublebar{f}_{\dot B^{p,p}_\theta(\partial V)}$. 
We now must write $f$ as a sum of $\dot B^{p,p}_\theta(\partial\Omega)$-atoms. 

For any function $h$, let $h^{x_0,2r}=\fint_{B(x_0,2r)\cap\partial V} h\,d\sigma$. 
Let $\varphi$ be a smooth cutoff function, supported in $B(x_0,2r)$ and identically equal to 1 in $B(x_0,(3/2)r)$. Then
\begin{equation*}f=f\varphi = (f-f^{x_0,2r})\varphi + f^{x_0,2r}\varphi
	=f^{x_0,2r}\varphi + \sum_k \lambda_k (a_k-a_k^{x_0,2r})\varphi.\end{equation*}
We claim that $f^{x_0,2r}\varphi=\lambda\,a$ for some atom $a$ and some $\lambda$ with $\abs{\lambda}\leq C \doublebar{f}_{\dot B^{p,p}_\theta(\partial\Omega)}$, and that $(a_k-a_k^{x_0,2r})\varphi$ is a bounded multiple of an atom or sum of two atoms. This suffices to show that $f\in \dot B^{p,p}_\theta(\partial\Omega)$.

We begin with $(a_k-a_k^{x_0,2r})\varphi$.
If $r_k\geq r$, let $\tilde a_k=(a_k-a_k^{x_0,2r})\varphi$. By the bound on $\nabla_\tau a_k$, we have that $\abs{a_k-a_k^{x_0,2r}}\leq C r_k^{\theta-1-\pdmnMinusOne/p} r$ in $\supp\varphi$. Thus, $\abs{\nabla_\tau \tilde a_k}\leq C r_k^{\theta-1-\pdmnMinusOne/p}$. If $\theta<1$ then the exponent is negative, and so $\abs{\nabla_\tau\tilde a_k}\leq C r^{\theta-1-\pdmnMinusOne/p}$. Furthermore, $\tilde a_k$ is supported in $B(x_0,2r)$, and so is a constant multiple of a $\dot B^{p,p}_\theta(\partial\Omega)$-atom.

If $r_k\leq r$, then $\abs{\nabla(a_k\varphi)}\leq Cr_k^{\theta-1-\pdmnMinusOne/p}$ and $a_k\varphi$ is supported in $\supp a_k\cap\supp\varphi \subset B(x_k,r_k)\cap\partial \Omega$, and so $a_k\varphi$ is a multiple of an atom. Furthermore, $\abs{a_k^{x_0,r}}\leq C r_k^{\dmnMinusOne+\theta-\pdmnMinusOne/p} r^{-\pdmnMinusOne}$, and so $\abs{\nabla ( a_k^{x_0,r}\varphi)}\leq Cr_k^{\dmnMinusOne+\theta-\pdmnMinusOne/p} r^{-\dmn}$. If $p>\pmin$, then the exponent of $r_k$ is positive and so $\abs{\nabla ( a_k^{x_0,r}\varphi)}\leq Cr^{\theta-1-\pdmnMinusOne/p}$. Because $\varphi$ is supported in $B(x_0,2r)$, this means that $a_k^{x_0,r}\varphi$ is also a bounded multiple of an atom.

We are left with the term $f^{x_0,2r}\varphi$. 
We begin by bounding the average value of~$f$. 
Observe that
\begin{equation*}\int_{B(x_0,2r)\cap\partial V} \abs{(f-f^{x_0,2r})\varphi}\,d\sigma
\leq
\sum_k \abs{\lambda_k} \int_{B(x_0,2r)\cap\partial V} \abs{(a_k-a_k^{x_0,2r})\varphi}\,d\sigma.
\end{equation*}
By the above arguments, $(a_k-a_k^{x_0,2r})\varphi$ is a multiple of an atom (or two) with characteristic length scale at most~$r$; thus,
\begin{equation*}\int_{B(x_0,2r)\cap\partial V} \abs{(a_k-a_k^{x_0,2r})\varphi}\,d\sigma
\leq C r^{\dmnMinusOne + \theta-\pdmnMinusOne/p}.\end{equation*}
If $p\leq 1$, then 
\begin{equation*}\int_{B(x_0,2r)\cap\partial V} \abs{(f-f^{x_0,2r})\varphi}\,d\sigma
\leq
C r^{\dmnMinusOne + \theta-\pdmnMinusOne/p} \Bigl(\sum_k \abs{\lambda_k}^p\Bigr)^{1/p}
\end{equation*}
and by the definition of the $\dot B^{p,p}_\theta(\partial V)$-norm,
\begin{equation*}\int_{B(x_0,2r)\cap\partial V} \abs{(f-f^{x_0,2r})\varphi}\,d\sigma
\leq
C r^{\dmnMinusOne + \theta-\pdmnMinusOne/p} \doublebar{f}_{\dot B^{p,p}_\theta(\partial V)}.
\end{equation*}
Because $f=0$ in $B(x_0,2r)\setminus B(x_0,(3/2)r)$, we have that
\begin{equation*}\abs{f^{x_0,2r}}\int_{\partial V\cap B(x_0,2r)\setminus B(x_0,(3/2)r)} \abs{\varphi}\,d\sigma
\leq
C r^{\dmnMinusOne + \theta-\pdmnMinusOne/p} \doublebar{f}_{\dot B^{p,p}_\theta(\partial V)}
\end{equation*}
and estimating the left-hand integral, we see that
\begin{equation*}\abs{f^{x_0,2r}}
\leq
C r^{\theta-\pdmnMinusOne/p} \doublebar{f}_{\dot B^{p,p}_\theta(\partial V)}.
\end{equation*}
Observe that $r^{\theta-\pdmnMinusOne/p}\varphi$ is a multiple of a ${\dot B^{p,p}_\theta(\partial \Omega)}$-atom, and so $f^{x_0,2r}\varphi=\lambda\,a$ for some ${\dot B^{p,p}_\theta(\partial \Omega)}$-atom $a$ and some $\abs{\lambda}\leq C\doublebar{f}_{\dot B^{p,p}_\theta(\partial V)}$, as desired.
\end{proof}

Thus, $\Trace u_j\in \dot B^{p,p}_\theta(\partial\Omega)$ for each~$j$.
This completes the proof of Theorem~\ref{thm:trace} in the case $m=1$.

To extend to the case $m>1$, observe that if $u\in \dot W_{m,av}^{p,\theta,q}(\Omega)$, then by definition $\partial^\gamma u\in \dot W_{1,av}^{p,\theta,q}(\Omega)$ for any $\gamma$ with $\abs{\gamma}=m-1$; thus $\Trace^\Omega \partial^\gamma u\in \dot B^{p,p}_\theta(\partial\Omega)$.

By Theorem~\ref{thm:smooth:dense}, smooth functions are dense in $\dot W_{m,av}^{p,\theta,q}(\Omega)$, and if $\varphi$ is smooth then $\Tr_{m-1}^\Omega\varphi$ lies in $\dot W\!A^p_{\theta}(\partial\Omega)$, a closed subspace of $(\dot B^{p,p}_\theta(\partial\Omega))^r$; thus, this is also true for more general $u\in \dot W_{m,av}^{p,\theta,q}(\Omega)$.
This completes the proof.

\section{Extensions: Neumann boundary data}
\label{sec:extension:neumann}

We have now established that $\dot W\!A^p_\theta(\partial\Omega)=\{\Tr u:u\in \dot W_{m,av}^{p,\theta,q}(\Omega)\}$, that is, that the space of Whitney-Besov arrays is the space of Dirichlet traces of $\dot W_{m,av}^{p,\theta,q}(\Omega)$-functions. We would like to similarly identify the space of Neumann traces $\NN=\{\M_m^\Omega\arr G:\arr G\in L_{av}^{p,\theta,q}(\Omega),\allowbreak\>\Div_m\arr G=0\}$.

In this section we show that, if $\Omega$ is any Lipschitz domain and if $\pmin<p\leq\infty$, then $\dot N\!A^p_{\theta-1}(\partial\Omega)\subseteq \NN$. We will not be able to prove the reverse inequality in general, but in Section~\ref{sec:trace:neumann} we will establish that $\dot N\!A^p_{\theta-1}(\partial\Omega)=\NN$ in some special cases.

\begin{thm} 
\label{thm:extension:Neumann}
Suppose that $0<\theta<1$ and that $\pmin<p\leq\infty$.
Let $\Omega$ be a Lipschitz domain with connected boundary.

Suppose that $\arr g\in \dot N\!A^p_{\theta-1}(\partial\Omega)$.
Then there is some $\arr G\in L_{av}^{p,\theta,\infty}(\Omega)$ such that $\Div_m\arr G=0$ in $\Omega$, $\arr g=\M_m^\Omega\arr G$, and such that
\begin{equation*}\doublebar{\arr G}_{L_{av}^{p,\theta,\infty}(\Omega)}\leq C \doublebar{\arr g}_{\dot B^{p,p}_\theta(\partial\Omega)}.\end{equation*}
\end{thm}

The remainder of Section~\ref{sec:extension:neumann} will be devoted to a proof of this theorem.


\subsection{The case \texorpdfstring{$p>1$}{p>1}}
\label{sec:neumann:extension:p>1}
Let $\arr g\in \dot N\!A^p_{\theta-1}(\partial\Omega)$. 
Observe that by Theorem~\ref{thm:trace} and by the duality characterization of $\dot N\!A^p_{\theta-1}(\partial\Omega)$, the operator $T_{\arr g}$, given by
\begin{equation*}T_{\arr g}( \Phi) = \bigl\langle \arr g, \Tr_{m-1} \Phi\bigr\rangle_{\partial\Omega},\end{equation*}
is a well-defined, bounded linear operator on $\dot W_{m,av}^{p',1-\theta,1}(\Omega)$. We may regard the space $\dot W_{m,av}^{p',1-\theta,1}(\Omega)$ as a closed subspace of $(L_{av}^{p',1-\theta,1}(\Omega))^r$, where $r$ is the number of multiindices~$\alpha$ with $\abs{\alpha}=m$. By the Hahn-Banach theorem we may extend $T_{\arr g}$ to a linear operator (of the same norm) on all of $(L_{av}^{p',1-\theta,1}(\Omega))^r$. Because the dual space to $L_{av}^{p',1-\theta,1}(\Omega)$ is $L_{av}^{p,\theta,\infty}(\Omega)$, there is some $\arr G$ with 
\begin{equation*}\doublebar{\arr G}_{L_{av}^{p,\theta,\infty}(\Omega)}\approx \doublebar{\arr g}_{(\dot W\!A^{p'}_{1-\theta}(\partial\Omega))^*}
=\doublebar{\arr g}_{\dot N\!A^{p}_{\theta-1}(\partial\Omega)}
\end{equation*}
that satisfies
\begin{equation*}\bigl\langle \arr G, \nabla^m F\bigr\rangle_{\Omega}
= T_{\arr g}(F)
= \bigl\langle \arr g,\Tr_{m-1} F\bigr\rangle_{\partial\Omega}\end{equation*}
for all $F\in \dot W_{m,av}^{p',1-\theta,1}(\Omega)$.
In particular, if $\Tr_{m-1} F=0$ then $\bigl\langle \arr G, \nabla^m F\bigr\rangle_{\Omega}=0$, and so $\Div_m\arr G=0$. We then have that $\arr g=\M_m^\Omega\arr G$, as desired.

\subsection{The case \texorpdfstring{$p\leq1$}{p<1}}
\label{sec:neumann:extension:p<1}
We now turn to the case $p\leq 1$; recall that in this case $\dot B^{p,p}_{\theta-1}(\partial\Omega)$ receives an atomic characterization. We will use the following two lemmas.

\begin{lem}[{\cite[Theorem~3.2]{MayMit04A}}]\label{lem:harmonic:C1} Let $\Omega$ be bounded $C^1$ domain. If $0<\theta<1$ and $\pdmnMinusOne/(\dmnMinusOne+\theta)< p\leq1$, then the Neumann problem for the Laplacian is well-posed in~$\Omega$ in the sense that, for every $g\in B^{p,p}_{\theta-1}(\partial\Omega)$, there is a unique function $u$ that satisfies
\begin{equation*}\Delta u=0 \text{ in }\Omega,
\quad
\MM_{\mat I}^\Omega u=g \text{ on }\partial\Omega,
\quad
\doublebar{u}_{B^{p,p}_{\theta+1/p}(\Omega)}\leq C \doublebar{g}_{B^{p,p}_{\theta-1}(\partial\Omega)}
.\end{equation*}
\end{lem}
Notice that because $\Delta$ is a second-order operator, $\MM_{\mat I}^\Omega u$ is a single function rather than an array; if $\nabla u$ is continuous up to the boundary then we have an explicit formula $\MM_{\mat I}^\Omega u=\nu\cdot \nabla u$, where $\nu$ is the unit outward normal vector.

The norm $\doublebar{u}_{B^{p,p}_{\theta+1/p}(\Omega)}$ is different from the norms we prefer to use in this paper. However, using the atomic decomposition of ${B^{p,p}_{\theta+1/p}(\Omega)}$ (see \cite{FraJ85}), it is straightforward to establish that if $p<1$ then 
\begin{equation*}\doublebar{\nabla u}_{L_{av}^{p,\theta,1}(\Omega)}
\leq C \doublebar{u}_{B^{p,p}_{\theta+1/p}(\Omega)}.\end{equation*}
Because $u$ is harmonic, we have that $\doublebar{\nabla u}_{L_{av}^{p,\theta,\infty}(\Omega)}\leq C \doublebar{\nabla u}_{L_{av}^{p,\theta,1}(\Omega)}$, and so we may replace the $B^{p,p}_{\theta+1/p}(\Omega)$-norm in Lemma~\ref{lem:harmonic:C1} by a $\dot W_{m,av}^{p,\theta,\infty}(\Omega)$-norm. (If $u$ is harmonic then the ${B^{p,p}_{\theta+1/p}(\Omega)}$-norm is equivalent to the $\dot W_{1,av}^{p,\theta,\infty}(\Omega)$-norm for $p\geq 1$ as well; see  \cite[Theorem~4.1]{JerK95}.)


The second lemma we will require is well known in the theory of second-order divergence-form elliptic equations and may be verified using elementary multivariable calculus.

\begin{lem}\label{lem:second:change}
Let $\Psi:\Omega\mapsto V$ be any bilipschitz change of variables and let $\mat J_\Psi$ be the Jacobean matrix, so $\nabla ( u\circ \Psi)=\mat J_\Psi^T \, (\nabla  u)\circ\Psi$. Let $\mat A$ be a matrix-valued function.
%
Let $\widetilde{\mat A}$ be such that
\begin{align*}
\mat J_\Psi\,\widetilde {\mat A}\,\mat J_\Psi^T&=\abs{\mat J_\Psi}\, (\mat A\circ\Psi)
\end{align*}
where $\abs{\mat J_\Psi}$ denotes the determinant of the matrix.

Let $u\in W^2_{1}(V)$ and let $\varphi\in W^2_{1}(V)$.
Then
\begin{equation*}\int_\Omega \nabla\tilde \varphi \cdot \widetilde {\mat A}\nabla\tilde u = \int_V \nabla\varphi\cdot {\mat A} \nabla u \end{equation*}
where $\tilde u=u\circ\Psi$ and $\tilde \varphi=\varphi\circ\Psi$. In particular, $\Div {\mat A}\nabla u=0$ in $V$ if and only if $\Div {\widetilde {\mat A}}\nabla \tilde u=0$ in~$\Omega$, and the conormal derivative $\nu\cdot \widetilde {\mat A}\nabla \tilde u = \MM_{{\widetilde {\mat A}}}^\Omega \tilde u$ is zero on some $\Delta\subset\partial\Omega$ if and only if $\nu\cdot \mat A\nabla u = \MM_{\mat A}^V u$ is zero on $\Psi(\Delta)\subset\partial V$.
\end{lem}
One may use Lemma~\ref{lem:second:change} to relate the conormal derivatives of $u$ and~$\tilde u$ even when they are not zero.

Let $a$ be a $\dot B^{p,p}_{\theta-1}(\partial\Omega)$-atom, supported in the surface ball $B(x_0,r)\cap\partial\Omega$. Our goal is to construct the Neumann extension of~$a$.
If $\partial \Omega$ is compact, then we may assume that $r$ is small enough that $B(x_0,4r)\subset B(x_j,2r_j)$ for one of the points $x_j$ of Definition~\ref{dfn:domain}. Let $V=V_j$ be the associated Lipschitz graph domain of Definition~\ref{dfn:domain}. (If $\partial\Omega$ is not compact then $\Omega$ is itself a Lipschitz graph domain; let $V=\Omega$.)

It suffices to show that, for all such atoms~$a$, and for all $\gamma$ with $\abs{\gamma}=m-1$, there exists some $\arr G\in L_{av}^{p,\theta,q}(\Omega)$, with norm at most~$C$, such that $\Div_m \arr G=0$ in $\Omega$ and such that if $F\in \dot W^2_{m,loc}(\overline\Omega)$,
\begin{equation*}\langle \arr G, \nabla^m  F\rangle_\Omega = \langle a, \partial^\gamma F\rangle_{\partial\Omega}. \end{equation*}

Now, observe that there is some Lipschitz function $\psi$ and some coordinate system such that $V=\{(x',t):t>\psi(x')\}$. Let $U$ be the Lipschitz cylinder given by
\begin{equation*}U=\{(x',t): \abs{x'-x_0'}<2r,\>\psi(x')<t<\psi(x')+r\}.\end{equation*}
Let $\Delta=\partial V\cap\partial U$.
Notice that $a$ is supported in~$\Delta$, and so we may extend $a$ by zero to a $\dot B^{p,p}_{\theta-1}(\partial U)$-atom.

Let $\widetilde B$ be the ball in $\R^d$ of radius~$r$ centered at the origin; then there is some bilipschitz change of variables $\Psi:\overline{\widetilde B}\mapsto \overline U$ with $\doublebar{\nabla \Psi}_{L^\infty}+ \doublebar{\nabla (\Psi^{-1})}_{L^\infty}\leq C$, where $C$ depends only on the Lipschitz character of $\Omega$. We may choose $\Psi$ such that $\widetilde\Delta=\Psi^{-1}(\Delta)$ is a hemisphere.

Let $\tilde a$ be the function defined on $\partial\widetilde B$ that satisfies
\begin{equation*}\int_{\partial\widetilde B} \varphi(\Psi(x))\,\tilde a(x)\,d\sigma(x)
= \int_{\partial U} \varphi(x)\,a(x)\,dx\end{equation*}
for all smooth, compactly supported test functions~$\varphi$; notice $a(\Psi(x))=\tilde a(x)\,\omega(x)$ for some real-valued function $\omega$ that is bounded above and below. In particular, $\abs{\tilde a(x)}\leq C\doublebar{a}_{L^\infty(\partial\Omega)}\leq Cr^{\theta-1/p\pdmnMinusOne}$ and $\int_{\partial\widetilde B} \tilde a(x)\,d\sigma(x)=0$; thus $\tilde a$ is a (bounded multiple of a) $\dot B^{p,p}_{\theta-1}(\partial\widetilde B)$-atom.

By Lemma~\ref{lem:harmonic:C1}, we have that there is some harmonic function $\tilde u$ with $\nu\cdot\nabla\tilde u=\tilde a$ on $\partial\widetilde B$; by the remarks following that lemma, we have that $\nabla\tilde u \in L_{av}^{p,\theta,\infty}(\widetilde B)$. Because $p\leq 1$ and $\dist(x,\widetilde\Delta)\geq \dist(x,\partial\widetilde B)$ for any $x\in\widetilde B$, we have that
\begin{equation*}\int_{\widetilde B} \sup_{B(x,\dist(x,\partial\widetilde B)/2)} \abs{\nabla \tilde u}^p\dist(x,\widetilde\Delta)^{p(1-\theta)-1}\,dx\leq C\doublebar{\nabla\tilde u}_{L_{av}^{p,\theta,\infty}(\widetilde B)}=C.\end{equation*}

Now, we extend $\tilde u$ to a function defined on all of $\R^\dmn$ by letting $\tilde u(x)=\tilde u(r^2\abs{x}^{-2}x)$ for all $x\notin \widetilde B$. It is straightforward to establish that $\tilde u$ is then harmonic away from $\supp \tilde a\subseteq\widetilde\Delta$. Using Lemma~\ref{lem:L:L1} and standard pointwise bounds on harmonic functions, we may show that 
\begin{equation*}\int_{\widetilde B} \sup_{B(x,\dist(x,\widetilde\Delta)/2)} \abs{\nabla \tilde u}^p\dist(x,\widetilde\Delta)^{p-1-p\theta}\,dx\leq C.\end{equation*}

Let $u(\Psi(x))=\tilde u(x)$. We will construct $\arr G$ from $\1_U\nabla u$. Thus we must estimate $\nabla u$. Notice that
\begin{equation*}\int_{U} \sup_{B(x,\dist(x,\Delta)/C)} \abs{\nabla  u}^p\dist(x,\Delta)^{p-1-p\theta}\,dx\leq C.\end{equation*}
But if $x\in U$ then $\dist(x,\Delta)\approx\dist(x,\partial\Omega)$ and so
\begin{equation*}\int_\Omega \sup_{B(x,\dist(x,\partial\Omega)/C)} \1_U\abs{\nabla  u}^p\dist(x,\partial\Omega)^{p-1-p\theta}\,dx\leq C.\end{equation*}
Thus $\1_U\nabla u\in L^{p,\theta,\infty}_{av}(\Omega)$.

We now consider the Neumann boundary values of~$u$.
By Lemma~\ref{lem:second:change}, there is a bounded matrix $\mat A$ such that 
\begin{equation*}\int_U \nabla\varphi \cdot \mat A\nabla u = \int_{\widetilde B} \nabla \tilde\varphi\cdot \nabla \tilde u\end{equation*}
for all smooth, compactly supported functions~$\varphi$.
But by the definition of conormal derivative,
\begin{equation*}\int_{\widetilde B} \nabla \tilde\varphi\cdot \nabla \tilde u = \int_{\partial\widetilde B} \tilde \varphi\, \tilde a \,d\sigma\end{equation*}
and by definition of~$\tilde a$,
\begin{equation*}\int_{\partial\widetilde B} \tilde \varphi\, \tilde a \,d\sigma = \int_{\partial U}  \varphi\,  a \,d\sigma.\end{equation*}
Recall that we chose a multiindex $\gamma$ with $\abs{\gamma}=m-1$. If $\abs{\alpha}=m$ and $\alpha>\gamma$, then there is some coordinate vector $\vec e_i$ with $1\leq i\leq \dmn$ and with $\alpha=\gamma+\vec e_i$; let $G_{\alpha} = \1_U (A\nabla u)_{i}$. If $\abs{\alpha}=m$ and $\alpha\not>\gamma$, let $G_{\alpha}=0$.

Then for any smooth, compactly supported function~$F$,
\begin{align*}
\langle \arr G, \nabla^m  F\rangle_\Omega
&=
	\int_\Omega \overline{\1_U\mat A\nabla u}\cdot \nabla \partial^\gamma F
=
	\int_U \overline{\mat A\nabla u}\cdot \nabla \partial^\gamma F
=
	\int_{\partial U} \overline a \,\partial^\gamma F\,d\sigma
\end{align*}
as desired.

\section{Traces: Neumann boundary data}
\label{sec:trace:neumann}

In the previous section, we established that
\begin{equation*}\dot N\!A^p_{\theta-1}(\partial\Omega) \subseteq \{\M_m^\Omega\arr G:\arr G\in L_{av}^{p,\theta,q}(\Omega),\allowbreak\>\Div_m\arr G=0\}.\end{equation*}
We conclude our study of Dirichlet and Neumann boundary values by establishing that, in certain special cases, the reverse inclusion is valid. Specifically, we will establish the reverse inclusion in the case $p>1$ (Theorem~\ref{thm:trace:Neumann}), in the case $\Omega=\R^\dmn_+$ (Theorem~\ref{thm:trace:Neumann:halfspace}), and in the case where $m=1$ and $\Omega$ is a Lipschitz graph domain (Corollary~\ref{cor:trace:Neumann:2}).

We conjecture that the reverse inclusion is true even in the case $m\geq 2$, $\pmin<p\leq 1$ and for $\Omega\neq\R^\dmn_+$ an arbitrary Lipschitz domain with connected boundary.

\begin{thm}\label{thm:trace:Neumann}
Suppose that $0<\theta<1$, that $1<p\leq\infty$, and that $1\leq q \leq \infty$.
Let $\Omega$ be a Lipschitz domain with connected boundary.

If $\arr G\in L_{av}^{p,\theta,q}(\Omega)$ and $\Div_m\arr G=0$ in~$\Omega$, then $\M_m^\Omega\arr G\in \dot N\!A^p_{\theta-1}(\partial\Omega)$. 

\end{thm}

\begin{proof}
Choose some $\arr G\in L_{av}^{p,\theta,1}(\Omega)$. 

Recall that $\dot N\!A^p_{\theta-1}(\partial\Omega)$ is the dual space to $\dot W\!A^{p'}_{1-\theta}(\partial\Omega)$. Let $\arr \varphi\in \dot W\!A^{p'}_{1-\theta}(\partial\Omega)$; then by Theorem~\ref{thm:extension} there is some $\Phi\in \dot W_{m,av}^{p',1-\theta,\infty}(\Omega)$ with $\Tr_{m-1}^\Omega\Phi=\arr\varphi$.

We then have that
\begin{align*}\langle \arr\varphi, \M_m^\Omega\arr G\rangle_{\partial\Omega}
&= \langle \nabla^m\Phi, \arr G\rangle_\Omega
\leq C\doublebar{\nabla^m\Phi}_{L_{av}^{p',1-\theta,\infty}(\Omega)} \doublebar{\arr G}_{L_{av}^{p,\theta,1}(\Omega)}
\\&\leq 
	C \doublebar{\arr\varphi}_{\dot W\!A^{p'}_{1-\theta}(\partial\Omega)} \doublebar{\arr G}_{L_{av}^{p,\theta,1}(\Omega)}
.\end{align*}
Thus, 
$\M_m^\Omega\arr G\in \dot N\!A^p_{\theta-1}(\partial\Omega)$ with $\doublebar{\M_m^\Omega\arr G}_{\dot N\!A^p_{\theta-1}(\partial\Omega)}\leq C\doublebar{\arr G}_{L_{av}^{p,\theta,1}(\Omega)}$, as desired.
\end{proof}

\begin{thm}\label{thm:trace:Neumann:halfspace}
Suppose that $0<\theta<1$, that $\pdmnMinusOne/(\dmnMinusOne+\theta)<p\leq 1$, and that $1<q\leq\infty$.

If $\arr G\in L_{av}^{p,\theta,q}(\R^\dmn_+)$ and $\Div_m\arr G=0$ in~$\R^\dmn_+$, then $\M_m^{\smash{\R^\dmn_+}\vphantom{\R}}\arr G\in \dot N\!A^p_{\theta-1}(\R^\dmnMinusOne)$. 

\end{thm}

Before presenting the (somewhat involved) proof of Theorem~\ref{thm:trace:Neumann:halfspace}, we will mention an important corollary in the case $m=1$.

\begin{cor}\label{cor:trace:Neumann:2}
Let $\theta$, $p$ and $q$ be as in Theorem~\ref{thm:trace:Neumann:halfspace}. Let
\begin{equation*}\Omega=\{(x',t):x'\in\R^\dmnMinusOne,\>t>\psi(x)\}\end{equation*}
for some Lipschitz function~$\psi$. Suppose that $m=1$.

If $\vec G\in L_{av}^{p,\theta,q}(\Omega)$ and $\Div\vec G=0$ in~$\Omega$, then $\MM_1^\Omega\vec G\in \dot B^{p,p}_{\theta-1}(\partial\Omega)$. 
\end{cor}

\begin{proof}
We apply the change of variables $\Psi(x',t)=(x',t+\psi(x'))$; then $\Psi(\R^\dmn_+)=\Omega$.

Let $\varphi$ be smooth and compactly supported. Define
\begin{equation*}\widetilde\varphi(x)=\varphi(\Psi(x)),\qquad \vec H(x) = \abs{\mat J_\Psi(x)}\, \mat J_\Psi(x)^{-1} \vec G(\Psi(x))= \mat J_\Psi(x)^{-1} \vec G(\Psi(x))\end{equation*}
where $J_\Psi(x)$ is the Jacobian matrix, so that $\nabla\widetilde\varphi(x)=J_\Psi(x)^T \nabla\varphi(\Psi(x))$.
An elementary argument in multivariable calculus (compare Lemma~\ref{lem:second:change}) establishes that
\begin{equation}\label{eqn:change:neumann}\int_\Omega \nabla\varphi\cdot \vec G = \int_{\R^\dmn_+} \nabla\widetilde\varphi\cdot \vec H.\end{equation}
In particular, observe that $\Div\vec H=0$ in~$\R^\dmn_+$. 
Also, $\dist(x,\partial \R^\dmn_+) \approx \dist(\Psi(x),\partial \Omega)$, and so $\vec G\in L_{av}^{p,\theta,q}(\Omega)$ if and only if $\vec H\in L_{av}^{p,\theta,q}( \R^\dmn_+)$. Thus, by Theorem~\ref{thm:trace:Neumann:halfspace}, we have that $\MM_1^{\smash{\R^\dmn_+}\vphantom{\R}}\vec H\in \dot B^{p,p}_{\theta-1}(\R^\dmnMinusOne)$.

Furthermore, by formula~\eqref{eqn:change:neumann}, we have that
\begin{equation*}\int_{\partial \Omega}\varphi(x)\MM_1^\Omega\vec G(x)\,d\sigma(x) = \int_{\partial \R^\dmn_+}\varphi(\Psi(x))\MM_1^{\smash{\R^\dmn_+}\vphantom{\R}}\vec H(x)\,d\sigma(x)\end{equation*}
and so $\MM_1^{\smash{\R^\dmn_+}\vphantom{\R}} \vec H(x)=\MM_1^\Omega\vec G(\Psi(x))\,s(x)$, where $s(x)$ is the infinitesimal change of area (essentially, the Jacobian determinant of the change of variables $\Psi:\partial \R^\dmn_+\mapsto \partial \Omega$). 

Observe that the atomic definition~\ref{dfn:besov} implies that $\MM_1^{\smash{\R^\dmn_+}\vphantom{\R}} \vec H\in \dot B^{p,p}_{\theta-1}(\R^\dmnMinusOne)$ if and only if $\MM_1^\Omega\vec G\in \dot B^{p,p}_{\theta-1}(\partial \Omega)$, as desired.
\end{proof}

\begin{proof}[Proof of Theorem~\ref{thm:trace:Neumann:halfspace}]



Let $\varphi$ be smooth and compactly supported; for notational convenience we will also take $\varphi$ real-valued. Let $\varphi_j(x)=\partial_t^j \varphi(x,t)\big\vert_{t=0}$. We then have that
\begin{equation*}\Tr_{m-1} \varphi = \Tr_{m-1} \sum_{j=0}^{m-1} \frac{1}{j!} t^j \varphi_j(x)\,\eta(t)\end{equation*}
where $\eta$ is a smooth cutoff function identically equal to $1$ near $t=0$.

Observe that $\langle \nabla^m \varphi, \arr G\rangle_{\R^\dmn_+}$ depends only on the functions $\varphi_j$ and on $\arr G$, and so there exist functions $M_j \arr G$ such that
\begin{equation*}\langle \nabla^m \varphi, \arr G\rangle_{\R^\dmn_+} 
= \sum_{j=0}^{m-1}\langle \varphi_j, M_j \arr G\rangle_{\partial\R^\dmn_+}
= \sum_{j=0}^{m-1}\langle \partial_\dmn^j\varphi, M_j \arr G\rangle_{\partial\R^\dmn_+}
.\end{equation*}

Notice that $\{M_j\arr G\}_{j=0}^{m-1}$ is not equal to our Neumann trace $\M_m^{\smash{\R^\dmn_+}\vphantom{\R}}\arr G$ but is closely related. In particular, observe that each $M_j\arr G$ is a well-defined function but that $\M_m^{\smash{\R^\dmn_+}\vphantom{\R}}\arr G$ is an equivalence class of functions.
We will first bound $M_j\arr G$ for each $0\leq j\leq m-1$, and then use $M_j\arr G$ to construct a representative of $\M_m^{\smash{\R^\dmn_+}\vphantom{\R}}\arr G$ that lies in $\dot B^{p,p}_{\theta-1}(\R^\dmnMinusOne)$.

Fix some $j$ with $0\leq j\leq m-1$. We  will use Daubechies wavelets to show that $M_j\arr G\in \dot B^{p,p}_{\theta+j-m}(\R^\dmnMinusOne)$.
The homogeneous Daubechies wavelets were constructed in \cite[Section~4]{Dau88}. We will need the following properties.
\begin{lem}
For any integer $N>0$ there exist real functions $\psi$ and $\varphi$ defined on $\R$ that satisfy the following properties.
\begin{itemize}
\item $\abs{\frac{d^k}{dx^k}\psi(x)}\leq C(N)$, $\abs{\frac{d^k}{dx^k}\varphi(x)}\leq C(N)$ for all $k<N$,
\item $\psi$ and $\varphi$ are supported in the interval $(-C(N),1+C(N))$,
\item $\int_{\R} \varphi(x)\,dx\neq 0$, $\int_{\R} \psi(x)\,dx = \int_{\R} x^k\,\psi(x)\,dx=0$ for all $0\leq k<N$.
\end{itemize}
Furthermore, suppose we let $\psi_{i,m}(x)=2^{i/2}\psi(2^i x-m)$ and $\varphi_{i,m}(x)=2^{i/2}\varphi(2^i x-m)$. Then $\{\psi_{i,m}:i,m\in\Z\}$ is an orthonormal basis for $L^2(\R)$, and  if $i_0$ is an integer then $\{\varphi_{i_0,m}:m\in\Z\}\cup\{\psi_{i,m}:m\in\Z,i\geq i_0\}$ is also an orthonormal basis for $L^2(\R)$.
\end{lem}
The functions $\varphi$ and $\psi$ are often referred to as a scaling function and a wavelet, or as a father wavelet and a mother wavelet.

We may produce an orthonormal basis of $L^2(\R^\dmnMinusOne)$ from these wavelets by considering the $2^\dmnMinusOne-1$ functions  $\Psi^\ell(x) = \eta_1(x_1)\,\eta_2(x_2)\dots\eta_\dmnMinusOne(x_\dmnMinusOne)$, where for each $i$ we have that either $\eta_i(x)=\varphi(x)$ or $\eta_i(x)=\psi(x)$, and where $\eta_k(x)=\psi(x)$ for at least one~$k$. Let $\Psi^\ell_{i,m}=2^{i\pdmnMinusOne/2}\Psi^\ell(2^i x-m)$; then $\{\Psi^\ell_{i,m}:1\leq\ell\leq 2^\dmnMinusOne-1,\> i\in\Z,\>m\in\Z^\dmnMinusOne\}$ is an orthonormal basis for~$L^2(\R^\dmnMinusOne)$. Notice that we may instead index the wavelets $\Psi^\ell_{i,m}$ by dyadic cubes~$Q$, with $\Psi^\ell_{i,m}=\Psi^\ell_Q$ if $Q=\{2^{-i}(y+m):y\in[0,1]^\dmnMinusOne\}$. We then have that $\Psi^\ell_Q$
has the following properties:
\begin{itemize}
\item $\Psi_Q^\ell$ is supported in $CQ$,
\item $\abs{\partial^\beta\Psi^\ell_Q(x)}\leq C(N)\ell(Q)^{-\pdmnMinusOne/2-\abs{\beta}}$ whenever $\abs{\beta}<N$,
\item $\int_{\R^\dmnMinusOne} x^\beta \Psi^\ell_Q(x)\,dx=0$ whenever  $\abs{\beta}<N$.
\end{itemize}

Because $\{\Psi_Q^\ell\}$ is an orthonormal basis of $L^2(\R^\dmnMinusOne)$, we have that if $f\in L^2(\R^\dmnMinusOne)$ then  
\begin{equation}
\label{eqn:besov:wavelet}
f(x)=\sum_{Q}\sum_{\ell=1}^{2^\dmnMinusOne-1} \langle f, \Psi_Q^\ell\rangle \Psi_Q^\ell(x).\end{equation}
By \cite[Theorem~4.2]{Kyr03}, if $f\in \dot B^{p,p}_{\sigma}(\R^\dmnMinusOne)$ for some $0<p\leq\infty$ and some $\sigma\in\R$, the decomposition~\eqref{eqn:besov:wavelet} is still valid. Furthermore, we have the inequality 
\begin{equation}
\label{eqn:besov:wavelet:norm}
\doublebar{f}_{\dot B^{p,p}_\sigma(\R^\dmnMinusOne)}^p\leq C\sum_Q\sum_{\ell=1}^{2^\dmnMinusOne-1}
\abs{\langle f, \Psi_Q^\ell\rangle }^p \ell(Q)^{\pdmnMinusOne(1-p/2)-p\sigma}.\end{equation}
The reverse inequality is also proven in \cite[Theorem~4.2]{Kyr03}; however, we will only use the direction stated above. 
Thus, to bound $M_j\arr G$, we need only analyze $\langle M_j\arr G, \Psi_Q^\ell\rangle_{\R^\dmnMinusOne}$. 

Let $\varphi(x,t) = \Psi_Q^\ell(x) \frac{1}{j!} t^j \eta(t)$. Then by definition of $M_j\arr G$,
\begin{equation*}\langle  \Psi_Q^\ell, M_j\arr G \rangle_{\R^\dmnMinusOne}
= \langle \nabla^m \varphi, \arr G\rangle_{\R^\dmn_+}.\end{equation*}
We choose the smooth cutoff function $\eta$ in the definition of $\varphi$ so that $\eta(t)=1$ if $t<\ell(Q)$ and $\eta(t)=0$ if $t>2\ell(Q)$, with the usual bounds on the derivatives of~$\eta$. We then have that
\begin{equation*}\langle \nabla^m \varphi, \arr G\rangle_{\R^\dmn_+}
= \sum_{\abs{\alpha}=m} \frac{1}{j!} \int_{\R^\dmn_+} \partial^\alpha (t^j\,\eta(t)\,\Psi_Q^\ell(x))\,G_\alpha(x,t)\,dx\,dt
.\end{equation*}
Because $\eta$ and $\Psi_Q^\ell$ are compactly supported, we have that
\begin{equation*}\langle \nabla^m \varphi, \arr G\rangle_{\R^\dmn_+}
= \sum_{\abs{\alpha}=m} \frac{1}{j!} 
\int_0^{2\ell(Q)} \int_{CQ} \partial^\alpha (t^j\,\eta(t)\,\Psi_Q^\ell(x))\,G_\alpha(x,t)\,dx\,dt
.\end{equation*}
Applying our bounds on the derivatives of $\Psi_Q^\ell$ and~$\eta$, we see that
\begin{equation*}\langle \nabla^m \varphi, \arr G\rangle_{\R^\dmn_+}
\le	
	C \ell(Q)^{j-\pdmnMinusOne/2-m}
	\int_0^{2\ell(Q)} \int_{CQ} \abs{\arr G(x,t)}\,dx\,dt
.\end{equation*}
Thus, by the bound \eqref{eqn:besov:wavelet:norm},
\begin{align*}\doublebar{M_j\arr G}_{\dot B^{p,p}_\sigma(\R^\dmnMinusOne)}^p
&\leq
	C \sum_Q \sum_{\ell=1}^{2^\dmnMinusOne-1}
	\abs{\langle \Psi_Q^\ell, M_j\arr G\rangle_{\R^\dmnMinusOne} }^p \ell(Q)^{\pdmnMinusOne(1-p/2)-p\sigma}
\\&\leq
	C \sum_Q 
	\ell(Q)^{\pdmnMinusOne(1-p)-p\sigma+pj-pm}
	\biggl(\int_0^{2\ell(Q)} \int_{CQ} \abs{\arr G}\biggr)^p
.\end{align*}
Recalling Lemma~\ref{lem:L:L1}, we set $\sigma=\theta+j-m$, so that
\begin{align*}\doublebar{M_j\arr G}_{\dot B^{p,p}_{\theta+j-m}(\R^\dmnMinusOne)}^p
&\leq
	C \sum_Q 
	\ell(Q)^{\pdmnMinusOne(1-p)-p\theta}
	\biggl(\int_0^{2\ell(Q)} \int_{CQ} \abs{\arr G}\biggr)^p
\end{align*}
which by Lemma~\ref{lem:L:L1} is at most $C	\doublebar{\arr G}_{L_{av}^{p,\theta,1}(\R^\dmn_+)}^p$.

We have now bounded $M_j\arr G$. We wish to show that some representative of $\M_m^{\smash{\R^\dmn_+}\vphantom{\R}}\arr G$ lies in $\dot B^{p,p}_{\theta-1}(\R^\dmnMinusOne)$.

Recall from \cite[Section~5.2.3]{Tri83} that the partial derivative operator $\partial^\zeta$ is a bounded operator from $\dot B^{p,p}_\sigma(\R^\dmnMinusOne)$ to $\dot B^{p,p}_{\sigma-\abs\gamma}(\R^\dmnMinusOne)$, and that the Laplace operator $-\Delta$ is a bounded operator $\dot B^{p,p}_\sigma(\R^\dmnMinusOne)\mapsto \dot B^{p,p}_{\sigma-2}(\R^\dmnMinusOne)$ with a bounded inverse. Let $$g_j = (-\Delta)^{j-m+1} M_j\arr G;$$ then $g_j \in {\dot B^{p,p}_{\theta-j+m-2}(\R^\dmnMinusOne)}$. For each multiindex $\gamma$, let $\gamma=(\gamma_\parallel,\gamma_\perp)$, where $\gamma_\parallel$ is a multiindex in $\N^\dmnMinusOne$ and where $\gamma_\perp=\gamma_\dmn$ is an integer. For each $\gamma$ with $\abs{\gamma}=m-1$, define
\begin{equation*}
g_\gamma = \frac{(m-1-\gamma_\perp)!}{\gamma_\parallel!} \partial^{\gamma_\parallel} g_{\gamma_\perp}
.\end{equation*}
Then $g_\gamma \in {\dot B^{p,p}_{\theta-1}(\R^\dmnMinusOne)}$. Now, 
\begin{equation*}
\langle \Tr_{m-1}^{\smash{\R^\dmn_+}\vphantom{\R}} \varphi, \arr g\rangle_{\partial\R^\dmn_+}
=
	\sum_{\abs{\gamma}=m-1} 
	\int_{\R^\dmnMinusOne} \partial^{\gamma_\parallel} \varphi_{\gamma_\perp}(x) \, g_\gamma(x)\,dx
.\end{equation*}
We integrate by parts to see that
\begin{equation*}
\langle \Tr_{m-1}^{\smash{\R^\dmn_+}\vphantom{\R}} \varphi, \arr g\rangle_{\partial\R^\dmn_+}
=
	\sum_{\abs{\gamma}=m-1} 
	(-1)^{m-1-\gamma_\perp} 
	\int_{\R^\dmnMinusOne} \varphi_{\gamma_\perp}(x) \, \partial^{\gamma_\parallel} g_{\gamma}(x)\,dx
.\end{equation*}
We have that
\begin{equation*}\sum_{\abs{\gamma_\parallel}=k} \frac{k!}{\gamma_\parallel!} \partial^{2\gamma_\parallel} = \Delta^k.\end{equation*}
Applying this formula and using the definition of $g_\gamma$, we see that
\begin{align*}
\langle \Tr_{m-1}^{\smash{\R^\dmn_+}\vphantom{\R}} \varphi, \arr g\rangle_{\partial\R^\dmn_+}
&=
	\sum_{\abs{\gamma_\perp}=0}^{m-1}
	\int_{\R^\dmnMinusOne} \varphi_{\gamma_\perp}(x) \, (-\Delta)^{m-1-\gamma_\perp} g_{\gamma_\perp}(x)\,dx
\\&=
	\sum_{\abs{\gamma_\perp}=0}^{m-1}
	\int_{\R^\dmnMinusOne} \varphi_{\gamma_\perp}(x) \, M_{\gamma_\perp} \arr G(x)\,dx
\\&=
	\sum_{j=0}^{m-1}
	\langle \varphi_j, M_j\arr G\rangle_{\R^\dmnMinusOne} 
= 
	\langle \nabla^m\varphi, \arr G\rangle_{\R^\dmn_+} 
.\end{align*}
Thus, $\arr g$ is a representative of $\M_m^\Omega\arr G$, and $\arr g\in {\dot B^{p,p}_{\theta-1}(\R^\dmnMinusOne)}$, as desired.
\end{proof}

\bibliographystyle{amsalpha}
\bibliography{bibli.bib}

\end{document}